\documentclass{article}

\usepackage[T1]{fontenc}
\usepackage{multicol,graphicx,xcolor}
\usepackage{authblk}
\usepackage{amsthm, amsmath, amssymb}
\usepackage{mathtools}
\usepackage[subnum]{cases}
\usepackage[shortlabels]{enumitem}
\usepackage[
  hmarginratio={1:1},     
  vmarginratio={1:1},     
  textwidth=15cm,        
  textheight=21cm,
  heightrounded,          
]{geometry}
\usepackage{hyperref}

\usepackage{tikz}
\tikzstyle{nodo}=[circle,draw,fill,inner sep=0pt, minimum size=0.5*width("k")]
\tikzstyle{infinito}=[circle,inner sep=0pt,minimum size=0mm]
\usetikzlibrary{shapes.geometric,calc,decorations.markings,math}
\usetikzlibrary{trees,decorations.pathreplacing}
\usetikzlibrary{decorations.fractals}

\definecolor{lightred}{rgb}{1,0.87,0.87}
\definecolor{lightgreen}{rgb}{0.8,1,0.8}
\hypersetup{%
  bookmarksnumbered, linkcolor=blue, citecolor=red,
  bookmarks, breaklinks,
  linkbordercolor=lightred,
  citebordercolor=lightgreen,
}

\theoremstyle{plain}
\newtheorem{theorem}{Theorem}
\newtheorem{corollary}[theorem]{Corollary}
\newtheorem{lemma}[theorem]{Lemma}
\newtheorem{proposition}[theorem]{Proposition}
\theoremstyle{definition}
\newtheorem{definition}[theorem]{Definition}
\theoremstyle{remark}
\newtheorem{remark}[theorem]{Remark}
\newtheorem{example}[theorem]{Example}

\numberwithin{equation}{section}
\numberwithin{theorem}{section}

\makeatletter

\let\epsilon=\varepsilon
\let\phi=\varphi
\@ifundefined{leqslant}{}{\let\le=\leqslant \let\leq=\le}
\@ifundefined{geqslant}{}{\let\ge=\geqslant \let\geq=\ge}
\@ifundefined{varnothing}{}{\let\emptyset=\varnothing}

\newcommand{\IN}{{\mathbb N}}
\newcommand{\IR}{{\mathbb R}}
\newcommand{\A}{{\mathcal A}}
\newcommand{\C}{{\mathcal C}}
\newcommand{\compact}{{\mathrm{\textup{c}}}}
\newcommand{\Cc}{\C_\compact}
\newcommand{\G}{{\mathcal G}}
\newcommand{\GG}{{\mathbf G}}
\newcommand{\GGfin}{{\mathbf G}_{\mathrm{fin}}}
\newcommand{\K}{{\mathcal K}}
\newcommand{\V}{{\mathcal V}}

\newcommand{\vv}{\text{\upshape \textsc{v}}}
\newcommand{\edge}{{\mathrm{\textup{e}}}}
\newcommand{\intd}{\,{\operatorfont d}}
\providecommand{\coloneq}{:=}

\DeclareMathOperator{\supp}{supp}
\DeclareMathOperator{\sign}{sign}
\DeclareMathOperator{\morse}{m}
\DeclareMathOperator{\e}{e}
\DeclareMathOperator{\dist}{dist}
\DeclareMathOperator{\spanned}{span}
\newcommand{\loc}{{\mathrm{\textup{loc}}}}
\newcommand{\limplies}{\Rightarrow}

\newcommand{\abs}[1]{\mathopen| #1\mathclose|}
\newcommand{\biggabs}[1]{\biggl| #1\biggr|}
\newcommand{\dd}{\@ifstar{\dd@txt}{\dd@frac}}
\newcommand{\dd@frac}[2]{
  \frac{{\operator@font d}#1}{{\operator@font d}#2}}
\newcommand{\dd@txt}[2]{
  {\operator@font d}#1 / {\operator@font d}#2}
\newcommand{\incident}{\succ}

\newcommand{\intervalcc}[1]{[#1]}
\newcommand{\intervalco}[1]{[#1\mathclose)}
\newcommand{\intervaloc}[1]{\mathopen(#1]}
\newcommand{\intervaloo}[1]{\mathopen(#1\mathclose)}

\newcommand{\FIXME}{\@ifstar{\FIXMEdel}{\FIXMEadd}}
\newcommand{\FIXMEdel}[1]{\textcolor{gray}{#1}}
\newcommand{\FIXMEadd}[1]{\textcolor{green!60!black}{#1}}

\def\@oddfoot{\hfill\today\hfill}
\def\@evenfoot{\hfill\today\hfill}

\makeatother
\author{Pablo Carrillo \footnote{pablo.carrillo-martinez@univ-fcomte.fr}}  \affil{{\it Universit\'e Marie et Louis Pasteur, CNRS, LmB (UMR 6623), F-25000, Besan\c{c}on, France}}

\author{Colette De Coster \footnote{colette.decoster@uphf.fr}} \affil{ Univ. Polytechnique Hauts-de-France, INSA Hauts-de-France, CERAMATHS - Laboratoire de Matériaux Céramiques et de Mathématiques, F-59313 Valenciennes, France }

\author{Damien Galant \footnote{damien.galant@umons.ac.be}} \affil{Univ. Polytechnique Hauts-de-France, INSA Hauts-de-France, CERAMATHS - Laboratoire de Matériaux Céramiques et de Mathématiques, F-59313 Valenciennes, France et
Département de Mathématique, Université de Mons, Place du Parc, 20, B-7000
Mons, Belgium}

\author{Louis Jeanjean\footnote{louis.jeanjean@univ-fcomte.fr}}
\affil{{\it Universit\'e Marie et Louis Pasteur, CNRS, LmB (UMR 6623), F-25000, Besan\c{c}on, France}}

\author{Christophe Troestler\footnote{christophe.troestler@umons.ac.be}}  \affil{Département de Mathématique, Université de Mons, Place du Parc, 20, B-7000
Mons, Belgium}

\title{Blow-up analysis and a priori bounds for NLS equations on metric graphs}
\date{\today}

\begin{document}

\maketitle

\begin{abstract}
  We consider, on a connected metric graph $\G$, a family of nonlinear
  Schr\"odinger equations
  \begin{equation}
    \label{stat Lnls}
    -u'' + W_n(x) u + \lambda_n u = \rho_n(x)|u|^{p-2}u,
    \quad n \in \IN.
    \tag{$*$}
  \end{equation}
  We assume that $p > 2$, $(W_n)$, $(\rho_n) \subseteq L^{\infty}(\G)$
  with $\rho_n \geq 0$, $|W_n|_{L^\infty(\G)}$ and
  $|\rho_n|_{L^\infty(\G)}$ are bounded and $\lambda_n \to +\infty$.
  Given $n \in \IN$, we call \emph{solution} a function
  $u_n \in H^1(\G)$ which satisfies~\eqref{stat Lnls} for that
  $n\in \IN$ together with the Kirchhoff conditions at the vertices.
  Focusing on the limiting behavior of sequences
  $(u_n) \subseteq H^1(\G)$ of solutions as $\lambda_n \to + \infty$
  and assuming that the Morse index $\morse(u_n)$ of $u_n$ is
  uniformly bounded, we establish, the existence of a finite subset of
  blow-up points away from which, up to a subsequence, $|u_n|$ has a
  global exponential decay.  These points are generally a strict
  subset of the blow-up points, and their number is estimated by the
  bound on the Morse index of $(u_n)$. It is the first time that this
  global exponential decay property is established on graphs even if
  one consider only signed solutions.  In the last part of the paper
  we derive various results of a priori bounds on the solutions in
  $L^\infty$ and $L^2$. Our blow-up analysis, combined with ODE
  arguments allows, for frequently considered classes of graphs, to
  obtain a fairly complete picture of the relationships between
  the number of nodal regions, Morse index, $L^\infty$ and $L^2$
  norms of solutions.
\end{abstract}

\medbreak

{\small \noindent \text{Keywords:} Nonlinear Schr\"odinger equations;
 Metric graphs; Blow-up analysis; Exponential decay; A priori bounds.}

\medbreak

{\small \noindent \text{Acknowledgements:}
This work has been carried out in the framework of the Project NQG (ANR-23-CE40-0005-01), funded by the French National Research Agency (ANR).
P.\,Carrillo, C.\,De Coster, D.\,Galant and L.\,Jeanjean thank the ANR for its support.
This work was initiated when D.\,Galant was
an F.R.S.-FNRS Research Fellow.}

\medskip

{\small \noindent \text{Statements and Declarations:}
    The authors have no relevant financial or non-financial interests to disclose.}

{\small \noindent \text{Data availability:}
    Data sharing is not applicable to this article as no datasets were generated or analysed during the current study.}

\section{Introduction and main results}
\label{intro}

The study of nonlinear Schr\"odinger (NLS) equations
on metric graphs has attracted
immense attention over the last few decades,
as can be seen for instance in the survey paper \cite{KNP}
and in the many references therein. NLS equations on graphs
appears in the study of Bose-Einstein condensates
on ramified structures (see e.g.\,\cite{AST11})
or in the study of networks of optic fibers
(see e.g.\,the discussion in \cite{No}).

Throughout this paper, we consider connected metric graphs
$\G = (\mathcal{E}, \V)$,
where $\mathcal{E}$ is the set of edges
and $\V$ is the set of vertices.
All the graphs we will consider belong
to the following class of graphs
(see \cite[Definition 2.1]{DeDoGaSeTr}).
\begin{definition}
    We denote by $\GG$ the class of metric graphs
    $\G$ that are connected, have at most countably
    many edges, where all vertices are adjacent to finitely
    many edges, and where the infimum of the length
    of all edges is positive.
\end{definition}
A large part of our analysis will focus
on the following class of graphs.
\begin{definition}
    We denote by $\GGfin$ the class of metric graphs
    in $\GG$ with finitely many vertices and edges.
\end{definition}

We recall that a graph $\G \in \GG$ can be naturally
seen as a complete metric space by considering
the shortest path distance on it, which we will
denote by $\dist$. We say that a graph $\G$ is
\emph{compact} if it has finitely many
edges, all of finite length. For graphs in $\GG$,
this is equivalent to being compact as a metric space.

\medskip
Contrary to most of the existing literature,
in this paper we shall not address existence issues.
Instead, assuming that their existence holds,
we are interested in deriving asymptotic properties of solutions for some class of NLS equations set on a graph in $\GGfin$.
 We consider, for each $n \in \IN$, the NLS problem:
\begin{equation}
  \label{eq:edo}
  \begin{cases}
    -u_n'' + W_n(x) u_n
    + \lambda_n u_n = \rho_n(x) \abs{u_n}^{p-2} u_n
    &\text{on every edge } \edge \in \mathcal{E},\\
    u_n \text{ is continuous on } \G,\\[1\jot]
    \displaystyle
    \sum_{\edge \incident \vv} \dd{u_{n,\edge}}{x}(\vv)=0
    &\text{at every vertex } \vv \in \V.
  \end{cases}
\end{equation}
In \eqref{eq:edo} the notation $\edge \incident \vv$ means that the edge $\edge$ is incident at
$\vv$, and the notation $\dd*{u_{\edge}}{x}(\vv)$ stands for  the derivatives away from the vertex $\vv$. The
last equation is the so-called \emph{Kirchhoff boundary condition}.

Given $n \in \mathbb{N}$, we
call solution a function  $u_n \in H^1(\G)$ which satisfies \eqref{eq:edo} for that $n \in \mathbb{N}$.

\medbreak

We assume that $p >2$ and that the potentials $W_n(x)$ and $\rho_n(x)$  satisfy the following set of conditions:
\begin{enumerate}[label={$(H)$}]
\item\label{H}%
    \begin{enumerate}[(i)]
    \item
      there exists $\bar W \in \IR$ such that,
        for almost all $x\in \G$
        and all $n \in \mathbb{N}$,
        $|W_n(x)|\le \bar W$;
    \item\label{H:bounds-on-rhon} there exists $b>0$ such that, for
        almost all $x\in \G$
        and all $n\in \mathbb N$,
        $0\leq \rho_n(x)\leq b$;

        \item there exists $a\in\,(0,b)$
        such that for every $e\in \mathcal{E}$, either,
        for all $n\in \mathbb  N$, $\rho_n(x)\geq a$
        on $e$ or for all $n\in \mathbb N$,
        $\rho_n\equiv 0$ on $e$.
  \end{enumerate}
\end{enumerate}

Assumption~\ref{H} is satisfied in particular if $W_n(x)=W(x)$
with $W\in L^{\infty}(\G)$ and  $\rho_n(x)= \gamma_n \rho(x)$ with $(\gamma_n)$ a bounded sequence with $\gamma_n\geq \epsilon>0$ and $\rho\in L^{\infty}(\G)$ which satisfies, for every $e\in \mathcal{E}$ either $\inf_e \rho(x)>0$ or
  $\rho\equiv 0$ on $e$ (as for example $\rho=\chi_\K$ the characteristic function of a subgraph $\K$ of $\G$). Observe that we do not require that $\rho_n$ is continuous on $\G$.\medskip

In the sequel, we denote by $\morse(u_n)$ the Morse index of
$u_n \in H^1(\G)$, see Definition~\ref{def:Morse index}.  Our main
result is the following.

\begin{theorem}
  \label{main}
  Let $\G \in \GGfin$, $p>2$,
  and assume that~\ref{H} holds.
  Let $(u_n) \subseteq H^1(\G) \backslash \{0\}$
  be a sequence of solutions
  to~\eqref{eq:edo}
  with
  \begin{equation*}
    \lim_{n\to\infty}\lambda_n= + \infty
    \qquad \text{ and } \qquad
    \forall n,\ \morse(u_n)\leq m^*
  \end{equation*}
  for some constant $m^*$.
  Then, passing if
  necessary to subsequences, there exist a number
  $m \in \{1,\dotsc, m^*\}$, sequences of points $(x^1_n), \dotsc,
  (x^m_n)$ in $\G$ and sequences of positive numbers
  $(R^1_n), \dotsc,\linebreak[1] (R^m_n)$
  with $R^i_n \to +\infty$
  such that
  \begin{gather}
    \label{eq:dist-points}
    \forall i \ne j,\qquad
    \lambda_n^{1/2} \dist(x^i_n, x^j_n) \to +\infty ,\\[1\jot]
    \label{eq:loc-max}
    \forall i,\qquad
    \abs{u_n(x^i_n)} = \max_{B(x^i_n, R^i_n \lambda_n^{-1/2})} \abs{u_n}
    \to + \infty.
  \end{gather}
  Moreover, there exists a constant $C>0$ such that,
  for all $n$ and all $x \in \G$,
  \begin{equation}
    \label{eq:exp-bound}
    \abs{u_n(x)}
    \le C \lambda_n^{1/(p-2)} \exp\bigl( -\tfrac{1}{2} \lambda_n^{1/2}
    d_n(x) \bigr),
    \quad \text{where } \, d_n(x) \coloneq \min_{1\le i \le m} \dist(x, x^i_n).
  \end{equation}
  Finally, for all $q \ge 1$,
  \begin{equation}
    \label{eq:mass-convergence}
    \lim_{n \to \infty}
    \lambda_n^{\frac{1}{2} - \frac{q}{p-2}}
    \int_\G \abs{u_n}^q \intd x
    \in \intervaloo{0,+\infty}.
  \end{equation}
\end{theorem}

Theorem \ref{main} is established through a detailed blow-up analysis of the sequences $(u_n) \subseteq H^1(\G)$. We first observe, in Lemma~\ref{lower-bound-max|u|}, that as $\lambda_n \to + \infty$, $u_n$ must necessarily blow-up.
Then, in Lemmas~\ref{concentration at a local maximum}
and~\ref{concentration at a local maximum v2}, we
  extract a subset of the blow-up points and we characterize the
  explosion behavior in the vicinity of these points, both when they
either accumulate in the interior of one edge or tend to infinity, or when
they accumulate at a vertex. In the first case, the limit problem is
an NLS equation set on $\mathbb R$ while in the latter case, the limit problem is posed on a \emph{star-graph $\G_k$}
(i.e.\ a graph made of a central vertex $\vv$
and $k$  half-lines
attached to it, observe that  $\G_2$
corresponds to $\mathbb R$).
Next, we show that the number
of our selected blow-up points is
bounded from above by $m^*$. Namely, there exist only
$m \in \{1,\dotsc, m^*\}$ sequences
$(x^1_n), \dotsc, (x^m_n)$ in $\G$ of such blow-up
points which, in addition,
cannot be \emph{too close} one from another.  Observe that, contrary to what happens for
    positive radial solutions on domains (see e.g.~\cite[Corollary 3.2]{Espetal}),
    here the number $m$ of selected blow-up points can be less
    than the number of local maxima of $|u_n|$ for $n$ large (see Remark \ref{Rem Strict} for more details). Finally, the derivation of the global decay estimate \eqref{eq:exp-bound} appears to be new on metric graphs and  is a major achievement of this work. It relies on obtaining a maximum principle (Proposition \ref{order-preserving-principle}), which is also of independent interest. Combining Proposition \ref{order-preserving-principle} with delicate constructions of comparison functions allows to derive \eqref{eq:exp-bound}. We point out that a difficulty to establish  \eqref{eq:exp-bound} lies in the presence of the vertices.
Indeed, we have to justify that exponential decay can propagate through them.
Having established the decay estimate, the limit \eqref{eq:mass-convergence} follows directly.
\medskip

Now, let us comment on the presence of the
weights $W_n(x)$ and $\rho_n(x)$.

The terms $W_n(x)$ correspond to the presence
of \emph{external potentials} in the nonlinear
Schrödinger equations. Notoriously, such nonautonomous
equations can be harder to study than their
autonomous counterparts, as can be seen
from studies on $\IR^N$.
Moreover, while the study
of NLS equations with ``Dirac-type'' interactions
at the vertices has undergone a lot of recent development
(see e.g.~\cite[Section 4.3]{KNP} and the references therein),
there are still not many works considering
NLS equations with potentials on graphs.

As for the term $\rho_n$, it serves at least two purposes.
First, by taking $\rho_n$ equal to the indicator function
of the \emph{compact core} of a graph (namely, the set of all its
bounded edges), one recovers the so-called
\emph{localized nonlinearity} case, that has undergone
a lot of interest recently (see e.g.\,the review papers
\cite{BCT-2019} or \cite[Section 4.4]{KNP}
and more recent developments such as
\cite{BCJS-2023, CGJT-2025}).
Secondly, when studying solutions to NLS equations
where a possible lack of a priori bounds is present,
one often has to resort
to the so-called ``monotonicity trick''.
This leads to the introduction of
a family of equations parameterized
by a non-negative real parameter $\rho$
living in a compact interval.
Note that in recent versions of the ``monotonicity trick'',
see in particular \cite{BCJS-2024},
information on the Morse index on the solutions
of the parameterized equations
can be obtained directly from some
``abstract'' variational principle.
Hence, families of solutions
as handled in Theorem~\ref{main}
naturally appear in recent research.

\medskip

In the present work, our focus is to derive ``once
and for all'' an asymptotic description of solutions
in the limit $\lambda_n \to +\infty$, trying to address
a large class of NLS equations.

\medskip

Let us mention that results
in the spirit of Theorem \ref{main} have been previously obtained in \cite{BCJS-2023, CJS-2022}, themselves inspired by the  blow-up analysis performed on a bounded domain of $\IR^N$ in \cite{EspPet}, see also
\cite{DHR2004}. More precisely, in \cite{CJS-2022}, the authors consider  the particular case where $\G$ is compact and $W_n$, $\rho_n$ are constants with $\rho_n\in \,(0,1]$ while, in \cite{BCJS-2023}, it was assumed that the $W_n$ are constants  and that $\rho_n(x)=\gamma_n \rho(x)$ with $\rho$ the characteristic function of the  metric subgraph of $\G$  consisting of all the bounded edges of $\G$ and $\gamma_n$ constants with $\gamma_n\in \,(0,1]$. In this last case note that, just as in our setting, the function $\rho$ is not necessarily continuous on $\G$.

Here, we extend the results of \cite{BCJS-2023, CJS-2022} in several
directions. First, in \cite{BCJS-2023, CJS-2022} only positive
functions $u_n \in H^1(\G)$ were considered. To pass from the
treatment of positive to possibly sign-changing functions we need to
derive a new Liouville-type result, see Proposition
\ref{stable-outside-a-compact}. Second, we allow the problem to be
non-autonomous with the presence of the potentials $W_n$ and
$\rho_n$. These less stringent assumptions should be useful in
future literature on the NLS equation on graphs. Moreover, the global
exponential decay estimate \eqref{eq:exp-bound} was not considered in
\cite{BCJS-2023, CJS-2022}.

\medbreak

The derivation of Theorem \ref{main} is,
in particular, motivated by the search
of a priori bounds for the solutions to \eqref{eq:edo}.
The need for these bounds appeared recently
in the study of the existence of \emph{normalized solutions}
(namely, solutions whose $L^2$ norm is prescribed)
for problems of the type of \eqref{eq:edo}
in the so-called mass-supercritical case,
namely when $p >6$,
see \cite{BCJS-2023, CGJT-2025, CJS-2022}.
There, $\lambda_n$ corresponds to a
Lagrange multiplier and thus is an unknown.
As a result, having information on it can prove crucial
to obtain convergence properties
of the associated sequence $(u_n) \subseteq H^1(\G)$.

Regarding that direction note that, as a direct consequence of Theorem \ref{main}, we have,

\begin{corollary}
  \label{info-norms}
  Let $\G \in \GGfin$ and $p>2$.  Assume that \ref{H} holds.
  Let $(u_n) \subseteq  H^1(\G)$ be a sequence of solutions
  to~\eqref{eq:edo} which satisfy, for some $c>0$ and $m^*>0$,
  \begin{equation*}
    \text{for all } n \in \IN, \qquad
    \int_{\G}|u_n|^2 \intd x = c
    \quad \text{and}\quad
    \morse(u_n) \le m^*.
  \end{equation*}
  Then, for any $p>2$ with $p \ne 6$,
    the sequence $(\lambda_n) \subseteq \IR$
    is bounded from above.
\end{corollary}

Let us mention that this result has already proved
useful in \cite{CGJT-2025} to obtain
the infinite multiplicity of normalized solutions
to a NLS equation with localized nonlinearities when $p > 6$.

\medskip
In the final part of the paper,
we establish relations between properties
of the solutions, namely between their
$L^{\infty}$ norm, their $L^2$ norm,
their Morse index and the number of their nodal zones.

We begin by an inequality between the number
of nodal zones and the Morse index
(Proposition~\ref{Morse_nodal_zones}),
remarking that some care is required when $\rho$ vanishes.

We then derive lower bounds
on $\lambda$ for a given Morse index
(Propositions \ref{lower_bound_lambda_compact}
and \ref{lower_bound_lambda_half-lines}).
The possibility to have $\rho \equiv 0$
on some edges and the presence of solutions
vanishing on edges creates some interesting behavior
that we illustrate through examples.

Next, we establish $L^\infty$ bounds for solutions
with a bounded Morse index (Proposition~\ref{L_infty_G})
and bounded values of $\lambda$,
a result similar to the one obtained
in the pioneering work of Bahri and Lions~\cite{BaLi}.
In our case,
we exploit the unidimensional nature of the domains
to resort to an ODE argument of Hartman \cite{Ha}.
In compact graphs when $\inf_{\G} \rho > 0$,
the $L^\infty$ bound and the bound on the Morse index
are actually equivalent
(Corollary~\ref{corollary_bounds_compact}).
We show that this is in general not the case
for noncompact graphs in $\GG$.
We also come across curiosities,
such as \emph{compactly supported $H^1$
solutions with infinite Morse index}
(Example~\ref{compactly_supported_solutions}).

We then focus on the $L^2$ norms of solutions,
with the aim to better understand the problem
of normalized solutions for our equations.
When $\lambda \to +\infty$, all the information
we need is given by Theorem~\ref{main},
in particular by \eqref{eq:mass-convergence}.
In compact graphs, it then suffices to combine
this result with our $L^\infty$ bound
to describe the full picture of the $L^2$ norms
of solutions (Theorem~\ref{thm:masses_compact})
for a bounded Morse index.
In non-compact settings, the $L^\infty$
information is insufficient
and we provide results in the settings of the localized
nonlinearity (Theorem~\ref{thm:locnonlin})
and of the usual NLS equation on the half-lines
(Theorem~\ref{thm:nls}), taking profit
of the explicit nature of the solutions on the half-lines.
In non-compact cases, the $L^2$ norms may blow up
as $\lambda \to 0^+$. We bound the rate at which this
may happen in terms of powers of $\lambda$ and show
that our exponents in the bounds are optimal.
We also observe that on non-compact graphs,
those rates can be different in the localized nonlinearity case and in the usual NLS equation.
\medbreak

The paper is organized as follows. In Section \ref{sec: Preliminaries} we recall some classical definitions and results concerning the Morse index and we establish a Liouville type result in Proposition~\ref{stable-outside-a-compact} and our maximum principle, Proposition~\ref{order-preserving-principle}.
In Section \ref{sec: blow-up} we perform the blow-up analysis and we prove Theorem \ref{main}.
Finally Section~\ref{sec:A_Priori} is devoted to the study of various properties of solutions  with
a bounded Morse index as described just above.

\medbreak

\noindent\textbf{Notations:}

Any bounded edge $\edge$ is identified with a closed bounded interval
$I_{\edge}$, typically $[0,\ell_{\edge}]$ with $\ell_{\edge}$ being
the length of $\edge$, while each unbounded edge is identified with a
closed half-line $I_{\edge}=[0,+\infty)$,
in which case $\ell_\edge \coloneq +\infty$.
The notation $\dd*{u_{\edge}}{x}(\vv)$ stands for
$u'_{\edge}(0)$ or $-u'_{\edge}(\ell_{\edge})$,
according to whether
the vertex $\vv$ is identified with $0$ or $\ell_{\edge}$.
For further details on analysis on metric
(quantum) graphs, one can e.g.\ refer to the monograph \cite{BK}.

A function on metric graph $u: \G \to \IR$ is identified with a vector of functions $\{u_{\edge}\}$, where each $u_{\edge}$ is defined on the corresponding interval $I_{\rm e}$ such that $u|_{\edge}=u_{\edge}$. Endowing each edge with Lebesgue measure, one can define the space $L^p(\G)$ in a natural way, with norm
given by the sum or series
\begin{equation*}
  \|u\|_{L^p(\G)}^p = \sum_{\edge \in \mathcal{E}} \|u_{\edge}\|_{L^p(\edge)}^p.
\end{equation*}
The Sobolev space $H^1(\G)$ consists of the set of \emph{continuous}
functions $u: \G \to \IR$ such that $u_{\edge} \in H^1(0,
\ell_{\edge})$ for every edge $\edge$ (recall that
$\ell_{\edge}$ can be $+\infty$)
and whose $H^1(\G)$ norm is finite, where
\begin{equation*}
  \|u\|_{H^1(\G)}^2
  = \sum_{\edge \in \mathcal{E}} \big(\|u_{\edge}'\|_{L^2(\edge)}^2
  + \|u_{\edge}\|_{L^2(\edge)}^2\big).
\end{equation*}
The spaces $L^p_\loc(\G)$ and $H^1_\loc(\G)$ are defined as usual.  The
notation $\Cc(\G)$ denotes the space of continuous functions with compact
support in $\G$.

To keep notations light, we will use
$\inf$/$\sup$/$\varliminf$/$\varlimsup$ even for functions
only defined almost everywhere. In this case, they
are to be understood as referring
to the \emph{essential} infimum,
supremum, limit inferior and limit superior.

\section{Some preliminaries and a maximum principle}

\label{sec: Preliminaries}

In this section, we consider the problem
\begin{equation}
  \label{eq:edo2}
  \begin{cases}
    -u'' + W(x) u  =  \rho(x) \abs{u}^{p-2} u
    &\text{on every edge } \edge \in \mathcal{E},\\
    u \text{ is continuous on }\G,
    \\[1\jot]
    \displaystyle
    \sum_{\edge \incident \vv} \dd{u_{\edge}}{x}(\vv)=0
    &\text{at every vertex } \vv \in \V,
  \end{cases}
\end{equation}
where $p >2$,  $W \in L^{\infty}(\G)$ and $\rho \in L^{\infty}(\G)$.

\begin{definition}
  \label{def:weak solution}
  Let $\G \in \GG$.
  We say that $u \in H^1_\loc(\G)$ is a \emph{solution} to \eqref{eq:edo2} if
  \begin{equation}
    \label{eq:weak-formulation}
    \forall \varphi\in  H^1(\G)\cap \Cc(\G), \quad \int_\G
    \big(u'\varphi' + W(x)\, u \,\phi\big) \intd x = \int_\G
    \rho(x)|u|^{p-2}u \,\phi\intd x.
  \end{equation}
\end{definition}

\begin{definition}
  \label{def:Morse index}
  Let $\G \in \GG$, $u \in H^1_\loc(\G)$ be a solution to \eqref{eq:edo2}.
  Consider the quadratic form associated
  to the linearization of \eqref{eq:edo2} at $u$:
  \begin{equation*}
    Q_u(\varphi; \mathcal{G})
    \coloneq \int_\G \Big(\abs{\phi'}^2 + W(x)\abs{\phi}^2
    - (p-1)\rho(x) |u(x)|^{p-2}|\phi|^2\Big) \intd x,
  \end{equation*}
  defined on $H^1(\G) \cap \Cc(\G)$.

  \begin{enumerate}
  \item The \emph{Morse index} of $u$,
    denoted $\morse(u)$ is the maximum
    integer $m$ such that there exists a subspace
    $X \subset H^1(\G) \cap \Cc(\G)$ of dimension $m$ with the property
    \begin{equation*}
    \forall \phi \in X \setminus\{0\},\quad
    Q_u(\varphi; \mathcal{G})
    < 0 ,
    \end{equation*}
    or $+\infty$ if there is no such maximal integer.

  \item We say that $u$ is \emph{stable outside
    a compact set $C \subseteq \G$}, if
    \[
        \forall \phi \in H^1(\G) \cap \Cc(\G), \quad
        \text{ if } \supp \phi \subseteq \G \setminus C,
        \text{ then }
        Q_u(\phi; \G) \ge 0.
    \]
  \end{enumerate}
\end{definition}

\begin{remark}
    The definition of Morse index for a
    partial differential equation
    set on a domain $\Omega$ of $\IR^N$ uses
    test functions in the space $\C^\infty_\compact(\Omega)$.
    Here we use instead test functions in
    $H^1(\G) \cap \Cc(\G)$ because of the presence
    of the vertices of $\G$.
    Note as well the density
    of $H^1(\G) \cap C_c(\G)$ in $H^1(\G)$
    (see e.g.\,\cite[Remark~3.2]{DeDoGaSeTr}).
\end{remark}

We now state the first main result of this section,
Proposition \ref{stable-outside-a-compact},
which will imply, in view of Lemma \ref{Stable via Morse},
that the solutions of some equations
that will appear through blow-up procedures,
decay to $0$ at infinity whenever they have
a finite Morse index.
In turn, this
will enable us to deduce the mentioned Liouville
type result, Corollary \ref{Liouville-lambda=0}.

\begin{proposition}
    \label{stable-outside-a-compact}
    Let $\G \in \GGfin$ be a noncompact metric graph,
    $p>2$ and
    $W$, $\rho \in L^\infty(\G)$ be nonnegative functions.
    Assume that $u \in H^1_\loc(\G) \cap L^\infty(\G)$
    is a solution to \eqref{eq:edo2}
    that is stable outside a compact set $C \subset \G$.
    Then $u' \in L^2(\G)$ and  $\rho \abs{u}^p \in L^1(\G)$.
    In particular, $u$ is uniformly continuous.
    \smallbreak

    Moreover, for every half-line $e$,
    \begin{itemize}
        \item if
        $\displaystyle\varliminf_{\substack{x\to\infty\\x\in e}}
        W(x)>0$ then $u_{\mid e}\in H^1(e)$ and
        $\displaystyle\lim_{\substack{x\to\infty\\x\in e}}
        u(x)=0$;
        \item if
        $\displaystyle\varliminf_{\substack{x\to\infty\\x\in e}}
        \rho(x)>0$ then
        $u_{\mid e} \in L^p(e)$ and
        $\displaystyle\lim_{\substack{x\to\infty\\x\in e}}
        u(x)=0$.
    \end{itemize}
\end{proposition}

\begin{remark}
\label{rem1}
Observe that, thanks to the previous proposition,
if there exists $\epsilon>0$
such that $W(x) \geq \epsilon$ on $\G$, a solution
$u \in H^1_\loc(\G) \cap L^\infty(\G)$ to \eqref{eq:edo2}
that  is stable outside a compact set $C \subset \G$
is in fact in $H^1(\G)$.
\end{remark}

\begin{proof}[Proof of Proposition \ref{stable-outside-a-compact}]
In this proof, for $R>0$, we  denote by $B_R$ the union
of the compact core $\mathcal{K}$ of $\G$
(i.e.\ the metric subgraph of $\G$ consisting
of all the bounded edges of $\G$) and of the
  initial segments of length $R$ of all the half-lines.
The proof is inspired by~\cite{RTT-JFA1998}.

 For any $\phi \in H^1(\G) \cap \Cc(\G)$, using
  $u\phi^2$ as a test function in the weak formulation of \eqref{eq:edo2},
	we get
  \begin{equation}
    \label{eq:u-phi2}
    \int_\G (u')^2 \phi^2 + 2 \int_\G u \phi u' \phi'
    + \int_\G W(x) u^2 \phi^2
    = \int_\G \rho(x) \abs{u}^p \phi^2.
  \end{equation}
  \medbreak

  \noindent\textit{Step 1: $\rho \abs{u}^p \in L^1(\G)$.}
  \medbreak

  Let $R_0 >0$ be large enough  so that $C \subseteq B_{R_0}$.
	 For each $R > 2R_0$,
  consider a function $\phi_{1,R} \in H^1(\G) \cap \Cc(\G)$ such that
  \begin{equation*}
    \begin{cases}
      \phi_{1,R} = 0 & \text{on } B_{R_0} \cup \complement B_{2R},\\
      \phi_{1,R} = 1 & \text{on } B_R \setminus B_{2R_0},\\
      0 \le \phi_{1,R} \le 1& \text{on } \G,\\
      \abs{\phi'_{1,R}} \le 2/R_0 & \text{on } B_{2R_0},\\
      \abs{\phi'_{1,R}} \le 2/R & \text{on } \complement B_R,
    \end{cases}
  \end{equation*}
  where $\complement A$ denotes the complement
  of the set $A$. The stability condition
  tested with the function $u\phi_{1,R}$ yields
  \begin{equation*}
    \int_\G (u')^2 \phi_{1,R}^2
    + 2 \int_\G u\phi_{1,R} u'\phi_{1,R}' + \int_\G u^2 (\phi'_{1,R})^2
    +  \int_\G W(x) u^2 \phi_{1,R}^2
    \ge (p - 1) \int_\G \rho(x) \abs{u}^p \phi_{1,R}^2.
  \end{equation*}
Thanks to~\eqref{eq:u-phi2},
 we obtain
  \begin{equation}
   \label{eq:maj-by-L2}
   \int_\G u^2 (\phi'_{1,R})^2
    \ge (p-2) \int_\G \rho(x) \abs{u}^p \phi_{1,R}^2
    \ge (p-2) \int_{B_R \setminus B_{2R_0}} \rho(x) \abs{u}^p .
    \end{equation}
    Let us denote $D$ a positive constant depending on $R_0$, $u$,
    $\G$ but independent of $R$ that may change at every
    occurrence.   We
  observe that the left-hand of \eqref{eq:maj-by-L2} can be bounded as follows
  \begin{align*}
    \int_\G u^2 (\phi'_{1,R})^2
    &= \int_{B_{2R_0} \setminus B_{R_0}} u^2 (\phi'_{1,R})^2
      + \int_{B_{2R} \setminus B_{R}} u^2 (\phi'_{1,R})^2 \\
    &\le \frac{4}{R_0^2} \int_{B_{2R_0} \setminus B_{R_0}} u^2
      + \frac{4}{R^2} \int_{B_{2R} \setminus B_{R}} u^2 \\
    &\le D \Bigl(1 + R^{-2} \int_{B_{2R} \setminus B_{R}} u^2 \Bigr).
  \end{align*}
  Recalling that $u$ belongs to $L^{\infty}(\G)$ and that there are finitely
  many half-lines, one obtains
  \begin{equation*}
    \int_\G u^2 (\phi'_{1,R})^2
    \le D \bigl(1 + R^{-2} R \bigr)
    \le 2D.
  \end{equation*}
  Therefore, passing to the limit $R \to +\infty$
  in~\eqref{eq:maj-by-L2} shows that $\rho \abs{u}^p \in L^1(\G)$.

  \medskip

  \noindent\textit{Step 2: $u'\in L^2(\G)$.}
  \medbreak

  Now let us consider the functions $\phi_{2,R} \in H^1(\G) \cap \Cc(\G)$ such that
  \begin{equation*}
    \begin{cases}
      \phi_{2,R} = 1 & \text{on } B_R,\\
      \phi_{2,R} = 0 & \text{on } \complement B_{2R},\\
      0 \le \phi_{2,R} \le 1& \text{on } \G,\\
      \abs{\phi'_{2,R}} \le 2/R & \text{on } \G.
    \end{cases}
  \end{equation*}
Again, using the fact that $u\in L^{\infty}(\G)$, we establish
  \begin{equation*}
    \int_\G u^2 (\phi'_{2,R})^2=
    \int_{B_{2R}\setminus B_R} u^2 (\phi'_{2,R})^2
    \le \frac{4}{R^2} \int_{B_{2R} \setminus B_{R}} u^2
    \le \frac{D}{R}
    \xrightarrow[R \to +\infty]{} 0.
  \end{equation*}
  The Cauchy-Schwarz inequality then implies that for every $\delta>0$, for $R$ large enough
  \begin{equation*}
    \biggabs{\int_\G u\phi_{2,R} u'\phi'_{2,R}}
    \le \biggl(\int_\G u^2 (\phi'_{2,R})^2 \biggr)^{1/2}
    \biggl( \int_\G (u')^2 \phi_{2,R}^2 \biggr)^{1/2}
    \le \delta \biggl( \int_\G (u')^2 \phi_{2,R}^2 \biggr)^{1/2} .
  \end{equation*}
  Using~\eqref{eq:u-phi2} with  $\phi_{2,R}$ and the above estimate
  yields
  \begin{equation}
  \label{borne u L2}
    \int_\G (u')^2 \phi_{2,R}^2
    - 2\delta \biggl( \int_\G (u')^2 \phi_{2,R}^2 \biggr)^{1/2}
    +  \int_\G W(x) u^2 \phi_{2,R}^2
    \le \int_\G \rho(x) \abs{u}^p \phi_{2,R}^2
    \le \int_\G \rho(x) \abs{u}^p<\infty
  \end{equation}
  by Step 1.
  In particular, as $W\geq 0$,
  $$
  \int_{B_R} (u')^2 \le \int_\G (u')^2 \phi_{2,R}^2 \le D,
  $$
  and, letting $R \to +\infty$, we conclude that $u' \in L^2(\G)$.

  \medskip

  \noindent\textit{Step 3: If $e$ is a half-line and
    $\displaystyle\varliminf_{\substack{x\to\infty\\x\in e}} W(x)>0$
    then $u\in H^1(e)$ and
    $\displaystyle\lim_{\substack{x\to\infty\\x\in e}} u(x)=0$;}
  \medbreak

The fact that $u\in H^1(e)$ can be deduced from \eqref{borne u L2} and every function in $H^1(e)$ satisfies $\displaystyle\lim_{\substack{x\to\infty\\x\in e}} u(x)=0$.
  \medbreak

  \noindent{\it Step 4: If $e$ is a half-line and
    $\displaystyle\varliminf_{\substack{x\to\infty\\x\in e}} \rho(x)>0$
    then
    $u \in L^p(e)$ and
    $\displaystyle\lim_{\substack{x\to\infty\\x\in e}} u(x)=0$.}
  \medskip

  The first claim follows directly from $\rho \abs{u}^p \in
  L^1$ and the embedding of $H^1_\loc(\G)$ into $L^p_\loc(\G)$.
  For the second one, if it did not hold, observing that $u' \in
  L^2(e)$ implies that the function $u$ is uniformly
  continuous, there would exist $\epsilon > 0$,
  $\delta > 0$ and a sequence $(x_n)$  such
  that, for all $n$,  $\abs{u} > \epsilon$ on $[x_n - \delta, x_n + \delta]$.
  Moreover, all $[x_n - \delta, x_n + \delta]$ are disjoint.  This is a
  contradiction with $\rho \abs{u}^p \in L^1(\G)$.
\end{proof}

We now give a sufficient condition for the stability of a solution outside a compact set.

\begin{lemma}
    \label{Stable via Morse}
    Let $\G \in \GG$, $p>2$ and
    $W$, $\rho \in L^\infty(\G)$ be  nonnegative functions.
    Any solution $u\in H^1_\loc(\G)$ to \eqref{eq:edo2}
    with finite Morse index is stable outside a compact set.
\end{lemma}

\begin{proof}
If $\morse(u) =0$ there is nothing to prove.

Otherwise, there exist functions
  $\phi_1, \dotsc, \phi_{\morse(u)}$ in $ H^1(\G) \cap
  \Cc(\G)$ that form a basis of a space $X$ of maximal dimension $\morse(u)$ such that
  $$
  \forall \phi \in X \setminus\{0\},\quad Q_u(\phi; \G) < 0.
  $$
  From here we can deduce that $u$ is stable outside the compact set defined by $C\coloneq\bigcup_{j=1}^{\morse(u)} \supp \phi_j$. Indeed, suppose this were not the case. Then take $\psi \in H^1(\G) \cap
  \Cc(\G)$ such that $\supp \psi \subseteq \G \setminus C$ and $Q_u(\psi; \G) < 0$. The function $\psi$  has disjoint
  support with $\phi_1,\dotsc,\phi_{\morse(u)}$ and the space
  $X^* \coloneq \spanned(X \cup \{\psi\})$  would be a space of dimension $\morse(u)+1$
  such that $\forall \phi \in X^* \setminus\{0\},\ Q_u(\phi; \G) < 0$, contradicting the maximality of the Morse index.
\end{proof}

\begin{remark}
Observe that by Remark \ref{rem1} and Lemma \ref{Stable via Morse},
under the assumptions of Theorem \ref{main},
assuming $u_n \in H^1(\G)$
is equivalent to assuming
$u_n \in H^1_\loc(\G) \cap L^\infty(\G)$
for $n$ large.
\end{remark}

Next we present a result stating a sufficient condition to guarantee the positivity of the Morse index.

\begin{lemma}
  \label{Morse index>=1}
  Let $\G \in \GG$, $p > 2$,
  $W$, $\rho \in L^\infty(\G)$ with  $W > 0$
 and $u \in H^1(\G)\backslash \{0\}$ be a solution to \eqref{eq:edo2}.
  Then the Morse index of $u$ is at least~$1$.
\end{lemma}

\begin{proof}
  Using a truncation of $u$ as test function $\varphi$
  in~\eqref{eq:weak-formulation} and passing to the limit yields
  $$
  \int_\G \big(\abs{u'}^2 + W(x) u^2\big) \intd x = \int_\G \rho(x) |u|^p \intd x
  $$
  from which it follows that
  $$
  Q_u(u;\G) = - (p-2) \int_\G \big(\abs{u'}^2 + W(x) u^2\big) \intd x < 0.
  $$
  So if by density we take
  $\phi \in H^1(\G) \cap \Cc(\G)$ sufficiently close to $u$, we
  have $Q_u(\phi; \G) < 0$.
  \end{proof}

\begin{remark}
  \label{rem pos}
  In fact, the assumption $W>0$ may be weakened by assuming for
  example that, for all
$u\in H^1(\G)\setminus\{0\}$,
  $$
  \int_\G \big(\abs{u'}^2 + W(x) u^2\big) \intd x > 0 .
  $$
  If $\G$ is compact, this means that the first eigenvalue $\xi_1$ of
 \begin{equation*}
  \begin{cases}
    -u'' + W(x) u  =  \xi u
    &\text{on every edge } e \in \mathcal{E},\\
    u \text{ is continuous on }\G,
    \\[1\jot]
		 \displaystyle
    \sum_{\edge \incident \vv} \dd{u_{\edge}}{x}(\vv)=0
    &\text{at every vertex } \vv \in \V,
  \end{cases}
\end{equation*}
satisfies $\xi_1>0$.
\end{remark}

\begin{corollary}[Liouville theorem for $W\equiv 0$ on star graphs]
    \label{Liouville-lambda=0}
    Let $\G_k$ be a star graph with $k$ half-lines, $p>2$,
    $\rho \in L^{\infty}(\G_k)$ be a non-negative function
    such that, for some constant $a>0$,
    for every $i\in\{1,\ldots,k\}$, either, $\rho(x)\geq a$
    on $e_i$ or $\rho\equiv 0$ on $e_i$.
    Assume that
    $\rho\geq a$ on at least one half-line.

    If $u \in H^1_\loc(\G_k) \cap L^\infty(\G_k)$
    is a solution to~\eqref{eq:edo2} with $W \equiv 0$
    such that $\morse(u) < \infty$, then $u \equiv 0$.
\end{corollary}

\begin{proof} Recall that, by Lemma \ref{Stable via Morse} and Proposition
\ref{stable-outside-a-compact}, $u'\in L^2(\G)$, $\rho |u|^p \in L^1(\G)$ and, for all edges $e$ such that $\rho(x)\geq a$ on $e$, we have
$\displaystyle\lim_{\substack{x\to\infty\\x\in e}} u(x)=0$.
\medbreak

Let us consider the restriction of $u$ on an edge $e$ where $\rho(x)\geq a$ on $e$. We shall prove that $u\equiv 0$ on $e$.

If $u\not\equiv0$ on $e$, we have $\tilde x_1\in [0,+\infty)$ such that
$\abs{u(\tilde x_1)} = \max_{e} \abs{ u} >0$.
 Without loss of generality, assume that
$\abs{u(\tilde x_1)}=u(\tilde x_1)$.
Observe that there exists $\delta_1>0$ such that,
for $x\in [\tilde x_1, \tilde x_1+\delta_1]$, we have
$u(x) > 0$ and
$u''(x)<-a \, \big|\frac{u(\tilde x_1)}{2}|^{p-1}$.
By integration, as $u'(\tilde x_1)\leq 0$, we obtain
$u'(\tilde x_1+\delta_1)<-a \,\big|\frac{u(\tilde x_1)}{2}\big|^{p-1} \delta_1$.
Moreover, for every $x>\tilde x_1+\delta_1$ such that $u(x)>0$, we have $u''(x)<0$, hence
$u'(x)<u'(\tilde x_1+\delta_1)<-a \,\big|\frac{u(\tilde x_1)}{2}|^{p-1} \delta_1$.
This implies the existence of $x_1>\tilde x_1$ such that
$u(x_1)=0$ and $u(x)>0$ for $x\in [\tilde x_1, x_1)$ and $u'(x_1)<u'(\tilde x_1+\delta_1)<-a \,\big|\frac{u(\tilde x_1)}{2}\big|^{p-1} \delta_1$.

Let us prove that there exists $x_2>x_1$ such that $u(x_2)=0$ and $u(x)<0$ on $(x_1,x_2)$.

As $\displaystyle\lim_{\substack{x\to\infty\\x\in e}} u(x)=0$, we have $\tilde x_2$ such that $u'(x)<0$ on
$(x_1, \tilde x_2)$ and   $u'(\tilde x_2)=0$.
Moreover, there exists $\delta_2>0$ such that,
for $x\in [\tilde x_2, \tilde x_2+\delta_2]$ we have
$u(x) < 0$ and
$u''(x)>a \,\big|\frac{u(\tilde x_2)}{2}\big|^{p-1}$.
Hence, by integration and as $u'(\tilde x_2)=0$, this implies
$u'(\tilde x_2+\delta_2)>a \,\big|\frac{u(\tilde x_2)}{2}\big|^{p-1} \delta_2$.
Moreover, for every $x>\tilde x_2+\delta_2$ such that $u(x)<0$, we have $u''(x)>0$,
hence $u'(x)> u'(\tilde x_2+\delta_2)>a \,\big|\frac{u(\tilde x_2)}{2}|^{p-1} \delta_2$.
This implies the existence of $x_2>\tilde x_2$ such that $u(x_2)=0$ and $u(x)<0$ for
$x\in [\tilde x_2, x_2)$ and
$u'(x_2)>u'(\tilde x_2+\delta_2)>a \,\big|\frac{u(\tilde x_2)}{2}|^{p-1} \delta_2$.

By induction, we prove the existence of a sequence $(x_n)$ such that $u(x_n)=0$ and $|u(x)|>0$ on $(x_{n-1},x_n)$.
\medbreak

Now consider the functions $\varphi_n=u|_{[x_{n-1},x_n]}$ extended by $0$ on the rest of
the graph. Observe that for all $n$
  \begin{align*}
      0=\int_\G u'\,\varphi_n' \intd x
      - \int_\G \rho(x) |u|^{p-2} u \,\varphi_n \intd x
      &=
          \displaystyle
  	\int_{x_{n-1}}^{x_n} |u'|^2 \intd x
          - \int_{x_{n-1}}^{x_n} \rho(x) |u|^p \intd x
      \\
      &=
  	\int_\G |\varphi_n'|^2  \intd x
          - \int_\G \rho(x) |u|^{p-2} |\varphi_n|^2 \intd x
  \end{align*}
  from which it follows that
  $$
  Q_u(\varphi_n;\G)
  = - (p-2) \int_\G \rho(x) |u|^{p-2} |\varphi_n|^2 \intd x < 0.
  $$
This contradicts $\morse(u)  < \infty$ and proves that $u\equiv 0$ on $e$.

Hence $u\equiv 0$ on  every edge where $\rho\geq a$.
\medbreak

  On the half-lines $\e$ where $\rho \equiv 0$, as $u$ satisfies $u''=0$ and
  $u' \in L^2(\e)$, we deduce that $u$ must be constant.
  Since $u$ vanishes at least on one half-line and there is one common
  value to all the half-lines,
  we deduce that $u \equiv 0$ on~$\G_k$.
\end{proof}

The following result will be key to derive \eqref{eq:exp-bound}.

\begin{proposition}[Maximum Principle]
    \label{order-preserving-principle}
    Let $\G \in \GGfin$,
    $W \in L^{\infty}_\loc(\G)$ be a
    nonnegative function and $\phi \in \mathcal{C}(\G)$
    be a function such that,
    for every edge $e$, we have
    $\phi\in W^{2,\infty}_\loc(e)$
    (i.e.\ if $e$ is a bounded edge then
    $\phi\in W^{2,\infty}(e)$, while
    if $e$ is the half-line $[0,+\infty)$
    then for every $R>0$,  $\phi\in W^{2,\infty}(0,R)$).
    Assume that
    \begin{enumerate}[(i)]
    \item\label{non-neg} $-\phi'' + W(x)\phi \ge 0 $
        in every edge $e$ of $\G$;
    \item\label{non-pos-at-vertices} for every vertex $\vv \in \V$, \quad
        $\phi(\vv) < 0 \ \Rightarrow\
        \displaystyle
        \sum_{\edge \incident \vv}
        \dd{\phi_{\edge}}{x}(\vv)\leq 0$.
    \end{enumerate}
    Then one of the following three possibilities applies:
    \begin{enumerate}[(a)]
        \item $\phi \geq 0$ on $\G$;
        \item there exists $C<0$
        such that $\phi\equiv C$ and $W\equiv 0$ on $\G$;
        \item $\inf_{\G}\phi=-\infty$.
    \end{enumerate}
    In the third case, $\G$ contains at least
    one half-line along which
    $\displaystyle \lim_{x\to\infty} \phi=-\infty$.
\end{proposition}

\begin{remark}
    Observe that our regularity assumptions on $\phi$
    implies that $\phi$ is $\mathcal{C}^1$ on each edge
    of $\G$ up to its boundary.
\end{remark}

\begin{proof}
    Assume by contradiction that $\phi$ is not constant
    and $-\infty<\inf_{\G}\varphi<0$.
    \medbreak

    \noindent\textit{Step 1: For every half-line $e_0$,
    if $\inf_{e_0}\varphi<0$
    then $\inf_{e_0}\varphi=\varphi(0)$.}
    \medbreak

    Otherwise, there exists an half-line $e$ and $\bar x\in e$ such
    that $\varphi(\bar x) < \min\{0,\varphi(0)\}$ and $\varphi'(\bar
    x)<0$ (with the parameterization of $e$ starting at the vertex~$0$).

    Let $I \coloneq \{x>\bar x \mid \varphi(x)<0\}$.  By assumption~(i), we
    have that, for a.e.\ $x\in I$,
    $\varphi''(x)\leq 0$.
    This implies that $I=(\bar x, +\infty)$ and
    that for all $x\in I$, $\varphi'(x)\leq \varphi'(\bar x)<0$.
    Thus, along the edge $e$, we have $\displaystyle\lim_{x\to\infty} \varphi(x)=-\infty$ which contradicts $\inf_{\G}\varphi > -\infty$.
    \medbreak

    \noindent\textit{Step 2: For every bounded edge $e$, if $\inf_{e}\varphi<0$ and the vertices of $e$ are $\vv_1$, $\vv_2$ then  $\inf_{e}\varphi=\min(\varphi(\vv_1), \varphi(\vv_2))$.}
    \medskip

    Let $e$ be a bounded edge with $\inf_{e}\varphi<0$. As $e$ is compact, there exists $x_0\in e$ such that $\varphi(x_0)=\min_{e}\varphi<0$. Let us prove that $x_0$ is a vertex of $e$.

    Otherwise, if $x_0$ is in the interior of $e$, $\varphi'(x_0)=0$ and there exist $x_1<x_0<x_2$ such that, for all $x\in(x_1,x_2)$, $\varphi(x)<0$.
    This implies that, for all $x\in(x_1,x_2)$, 
    \begin{equation*}
      \varphi(x_0)
      \le \varphi(x)
      \le \varphi(x_0)
      + \int_{x_0}^x \int_{x_0}^t W(s)\varphi(s) \intd s\intd t
      \le \varphi(x_0)
    \end{equation*}
    where the second inequality results from assumption~\ref{non-neg}.
    We conclude that $\varphi$ is constant on $(x_1,x_2)$.
    Thus, we have $(x_1,x_2)=e$ and $W\equiv 0$ on $e$. In particular,
    $\min_e\varphi=\varphi(\vv_1)=\varphi(\vv_2)$.
    \medbreak

    \noindent\textit{Step 3: there exists $x_0\in \V$ such that
      $\inf_{\G}\varphi = \min_{\G}\varphi=\varphi(x_0)$.}
    \medskip

    Since $\inf_\G \varphi < 0$, it suffices to take the
    infimum on all the edges $e$ where $\inf_e \varphi < 0$.  By
    steps~1 and~2, the infimum on one of such edge $e$ is achieved
    at a vertex $\vv \in \V$.  Thus
    $\inf_\G \varphi = \inf_\V \varphi$ and it is achieved
    because $\V$ is finite.

    \medbreak

    \noindent\textit{Conclusion.}
    \medskip

    By Step 3, there exists $x_0\in \mathcal V$ such that
    $\varphi(x_0) = \min_{\G}\varphi$.  Thus, for all $e
    \incident x_0$,  $\dd{\varphi_{\edge}}{x}(x_0)\ge 0$.
    As $\phi(x_0) = \inf_\G \varphi < 0$,
    assumption~\ref{non-pos-at-vertices} implies that
    for all $e \incident x_0$,  $\dd{\varphi_{\edge}}{x}(x_0)= 0$.

    As in Step  2, this implies that $\varphi(x)\equiv \varphi(x_0)$ and $W\equiv 0$ on all the edges  $e \incident x_0$.

    Since $\G$ is connected, iterating the procedure, we conclude that
    $\varphi(x)\equiv \varphi(x_0)$ on $\G$.
    \medbreak

    It remains to prove the last point concerning the case
    $\inf_{\G} \varphi = -\infty$.
    Since there are finitely many edges, there must exist an half-line
    $e \in \mathcal E$, identified with $\intervalco{0, +\infty}$,
    such that $\inf_e \varphi = -\infty$.  If there exists
    $\bar x \in e$ such that $\varphi(\bar x) < 0$ and
    $\varphi'(\bar x) < 0$, then, arguing as in Step~1, we conclude
    that $\varphi(x) \to -\infty$ as $x \to \infty$, as desired.  If
    not, for every $x \in e$, one has
    $\varphi(x) < 0 \limplies \varphi'(x) \ge 0$.  We claim that
    this implies that, for every $x \in e$ such that
    $\varphi(x) < 0$, $\varphi$ is negative on $\intervalcc{0, x}$.
    Indeed, if there was a $x' \in \intervalco{0, x}$ such that
    $\varphi(x') \ge 0$, then there would exists a
    $x'' \in \intervalco{x', x}$ such that $\phi(x'') = 0$ and
    $\phi < 0$ on $\intervaloc{x'', x}$ and applying the Mean Value
    Theorem to $\intervalcc{x'', x}$ would yield a contradiction.
    As a consequence of the previous claim,
    $\inf_e \varphi = \varphi(0)$ which contradicts
    $\inf_e \varphi = -\infty$.
\end{proof}

A strong maximum principle also holds for every graph in $\GG$.
It will be used in Section~\ref{sec:A_Priori}.
\begin{proposition}[Strong Maximum Principle]
  \label{strong_maximum_principle}
  Let $\G \in \GG$,
  $W \in L^{\infty}_\loc(\G)$ a
  nonnegative function and $\phi \in \mathcal{C}(\G)$ a function such that,
  for every edge $e$, we have $\phi\in W^{2,\infty}_\loc(e)$.
  Assume that
  \begin{enumerate}[(i)]
  \item $-\phi'' + W(x)\phi\geq 0 $ in every edge $e$ of $\G$;
  \item for every vertex $\vv \in \V$, \quad
   $\phi(\vv) = 0 \ \Rightarrow\
   \displaystyle
     \sum_{\edge \incident \vv} \dd{\phi_{\edge}}{x}(\vv)\leq 0$;
  \item $\varphi \ge 0$ on $\G$.
  \end{enumerate}
  Then, either $\phi \equiv 0$ on $\G$
  or for all $x \in \G$, $\varphi(x) > 0$.
\end{proposition}

\begin{proof}
Assume by contradiction that $\phi\not\equiv 0$ and there exists $\bar x\in \G$ such that $\phi(\bar x)=\min \phi= 0$.
  As $\phi \in \mathcal{C}(\G)$ and $\G$ is connected,
  there exists an edge $\bar e$ and two points
  $x_0$ and $x_1$ of $\bar e$ such that
  $\varphi(x_0)=0$ and $\varphi(x_1)>0$.
  If $x_0$ belongs to the interior of $\bar e$ then,
  since $x_0$ is a local minimum, $\varphi'(x_0)=0$.
  If $x_0$ is a vertex of $\G$, then
  (using again that $x_0$ is a local minimum)
  for every $\edge \incident x_0$,
  $\dd{\varphi_{\edge}}{x}(x_0)\ge 0$ and by hypothesis~(ii),
  we get $\dd{\varphi_{\bar e}}{x}(x_0)= 0$.

  \medbreak
  Parameterizing the edge in the reverse direction if necessary,
  we may assume that $x_0 < x_1$.
  Then, the restriction of $\varphi$
  to the interval $\intervalcc{x_0, x_1}$
  satisfies $\varphi(x_0) = 0$, $\varphi'(x_0) = 0$
  and $\varphi(x_1) > 0$.
This contradicts the strong maximum principle on an interval
  (see e.g. \cite[Appendix, Theorem~5.1]{DeHab}).
\end{proof}

We end this section with a technical result which will be
used in the construction of the comparison functions in the proof of
\eqref{eq:exp-bound}.

\begin{lemma}
  \label{lem:exp-bound}
  Let $\ell, \phi_0, \phi_1$ be positive real numbers such that
  $\phi_0 \ge \phi_1$.  Let
  \begin{equation*}
    \alpha
    \coloneq \frac{1}{\ell} \cosh^{-1} \frac{\phi_0}{\phi_1}
    \ge 0.
  \end{equation*}
  There exists $\phi \in \C^\infty(\intervalcc{0,\ell})$ such that
  \begin{enumerate}[(i)]
  \item $\phi$ is non-increasing on $\intervalcc{0,\ell}$;
  \item $\phi(0) = \phi_0$, $\phi(\ell) = \phi_1$ and $\phi'(\ell) = 0$;
  \item $\phi'' = \alpha^2 \phi$;
  \item $0 < \phi(x) \le \phi_0 \e^{-\beta x}$ for all
    $x \in \intervalcc{0,\ell}$, where
    $\beta \coloneq \frac{1}{\ell} \ln(\phi_0/\phi_1) \ge 0$.
  \end{enumerate}
\end{lemma}

\begin{proof}
  It is easy to check that the function
  $\phi(x) \coloneq \phi_1 \cosh(\alpha(\ell - x))$ possesses the desired
  first three properties.   The last one is trivial if $\phi_0 =
  \phi_1$.  So we assume from now on that $\phi_0 > \phi_1$,
  thus $\alpha > 0$ and $\beta > 0$.
  Observe that the definition of
  $\beta$ is equivalent to $\e^{\beta \ell} = \phi_0/\phi_1$ and so
  equality holds at both ends of the interval:
  $\phi(0) = \phi_0 \e^{-\beta 0}$ and
  $\phi(\ell) = \phi_1 = \phi_0 \e^{-\beta \ell}$.  Using again
  $\e^{\beta \ell} = \phi_0/\phi_1$, expanding $\cosh$ and
  substituting $\ell-x$ with $t$, the desired inequality is equivalent
  to
  \begin{equation*}
    \forall t \in \intervalcc{0,\ell},\quad
    f(t) \coloneq
    \e^{2 \alpha t} - 2 \e^{(\alpha + \beta) t} + 1
    \le 0 .
  \end{equation*}
  The equality on the boundary of the interval translates to
  $f(0) = 0 = f(\ell)$.  Moreover, simple computations establish that
  $f$ has at most one critical point in $\intervalcc{0,\ell}$
  and $f'(0) = -2\beta < 0$.  These imply the claim.
\end{proof}

\section{Blow-up analysis and proof of Theorem \ref{main}}
\label{sec: blow-up}

\begin{lemma}
  \label{lower-bound-max|u|}
  Let $\G \in \GG$, $p>2$ and $W$, $\rho\in L^{\infty}(\G)$. Assume
  there exists $\alpha>0$, $\beta>0$ such that $W\geq \alpha$ and
  $\rho\leq \beta$ a.e.\ on $\G$.
  Let $u\in H^1_\loc(\G) \backslash \{0\}$
  be a solution of \eqref{eq:edo2}.
  Let $\bar x \in \G$ be a local
  maximum point of $\abs{u}$.
  \begin{enumerate}
  \item If $\abs{u(\bar x)}>0$ then
    \begin{equation*}
      \abs{u(\bar x)} \ge \bigl(\alpha/\beta\bigr)^{1/(p-2)}.
    \end{equation*}

  \item If $\abs{u(\bar x)}=0$ then $u$
    vanishes on all the edges containing $\bar x$.
  \end{enumerate}
\end{lemma}

\begin{proof} We distinguish two cases.
\medbreak

\noindent\textit{Case 1: $u(\bar x)=0$.}
\medbreak

In that case,  $u\equiv 0$ in a neighbourhood of $\bar x$ as $\bar x$ is a local maximum point of
$|u|$. By uniqueness of the solution of the Cauchy problem
of the NLS ODE, $u\equiv0$ on the edges containing
$\bar x$.
\medbreak

\noindent\textit{Case 2: $u(\bar x)\ne 0$.}
\medskip

Assume by contradiction that
$\abs{u(\bar x)} <(\alpha/\beta)^{1/(p-2)}$.
Let $\edge$ be an edge of $\G$,
identified with $\intervalcc{0, \ell_\edge}$
or $\intervalco{0, \infty}$,
such that $\bar x \in \edge$.
It is standard to show that
$u|_\edge \in W^{2,\infty}_\loc(\edge)$.

If $\bar x$ is in the interior of $\edge$, then,
there exists $x_2>\bar x$
such that $u\in \C^1([\bar x,x_2])$ and
\begin{equation}
  \label{maxloc}
  u'(\bar x)= 0\geq \sign(u(\bar x))u'(x_2).
\end{equation}
On the other hand,
taking $x_2$ closer to $\bar x$ is necessary,
the contradiction assumption implies
\begin{equation*}
  \text{for a.e. }x\in (\bar x,x_2), \qquad
  u(x)\ne 0 \quad \text{and}\quad
  \sign(u(\bar x)) u''(x)> 0.
\end{equation*}
Integrating this last inequality on $[\bar x,x_2]$
gives a contradiction with \eqref{maxloc}.
\medbreak

On the other side, if $\bar x$ lies
on the boundary of $\edge$, the derivatives
$\dd*{u_{\tilde e}}{x}(\bar x)$ must
all have the same sign for
$\tilde e \incident \bar x$ as $\bar x$
is a maximum point of $\abs{u}$ and the Kirchhoff condition
then implies
that $\dd*{u_{\edge}}{x}(\bar x)=0$.
Then we conclude by the same argument as in the first case.
\end{proof}

\begin{remark}
\label{Rem rho>0}
    Observe also that, if $\rho$ is such that there exists $a>0$ with, for every $e\in \mathcal{E}$, either,
    $\rho(x)\geq a$
    on $e$ or
    $\rho\equiv 0$ on $e$, then, with the notations of the proof of the previous lemma,
    one has $\rho\geq a>0$
    on every edge containing $\bar x$.
   Indeed, if $\rho\equiv 0$ on one of these edges, we also have
  \begin{equation*}
    \text{for a.e. } x\in \,(\bar x,x_2), \qquad
    u(x)\ne 0 \quad \text{and}\quad
    \sign(u(\bar x)) u''(x)> 0,
  \end{equation*}
  whatever the value of $|u(\bar x)|$ is.
\end{remark}

\begin{remark}
  As a consequence of Lemma \ref{lower-bound-max|u|},
  under assumption \ref{H},
  sequences of solutions
  $(u_n) \subseteq H^1(\G)\setminus\{0\}$
  to~\eqref{eq:edo} with $\lambda_n \to +\infty$
  (as in the assumptions of Theorem~\ref{main}),
  must necessarily blow-up in $L^\infty$-norm.
\end{remark}

Our next lemmas describe the behavior close to some points
of local maxima.
The situation is very different depending on whether
\[
\varlimsup_{n \to \infty}
\abs{u_n(x_n)}^{(p-2)/2} \dist(x_n, \V) = +\infty
\quad \text{or}\quad
\varlimsup_{n \to \infty}
\abs{u_n(x_n)}^{(p-2)/2}\dist(x_n, \V) < +\infty.
\]
In particular, in the first case, up to rescaling, the sequence
converges to the solution of a problem on $\mathbb R$ while in the
second case, the limit problem is set on a star graph. The next two
lemmas describe these situations.

In what follows, we denote
$B(x_0, r) \coloneq \{ x \in \G \mid \dist(x, x_0) < r\}$.

\begin{lemma}
  \label{concentration at a local maximum}
  Assume that $\G \in \GGfin$, $p>2$, and \ref{H} holds and let
  $(u_n) \subseteq H^1(\G) \backslash \{0 \}$ be a sequence of
  solutions to~\eqref{eq:edo}.  Suppose that
  \begin{equation*}
    \lambda_n \to +\infty \quad \text{ and } \quad
    \varlimsup_{n \to \infty}\morse(u_n) < \infty
  \end{equation*}
  where $\morse(u_n)$ denotes the Morse index of $u_n$ (see
  Definition~\ref{def:Morse index}).

  Let $(x_n) \subseteq \G$ be a sequence of local maxima of $|u_n|$
  such that, for some sequence $(\tilde R_n)$, with
  \begin{equation}
    \label{tilde R_n}
    \tilde R_n|u_n(x_n)|^{\frac{p-2}{2}}\to\infty,
  \end{equation}
  we have
  \begin{equation}
    \label{eq:loc-max-quantified1}
    \abs{u_n(x_n)} = \max_{B(x_n, \tilde R_n)} \abs{u_n}
    \qquad\text{and}\qquad
    u_n(x_n) \ne 0.
  \end{equation}
  Assume moreover that
  \begin{equation}
    \label{eq:dist->infty}
    \varlimsup_{n \to \infty}
    \abs{u_n(x_n)}^{\frac{p-2}{2}} \dist(x_n, \V) = +\infty.
  \end{equation}
  Then, passing if necessary to subsequences, one has
  \begin{enumerate}[(i)]
  \item\label{xn-in-e1}
    All the $x_n$ lie in the interior of the same edge $e_1$
    (identified with $\intervalcc{0,\ell_{e_1}}$ or with
    $\intervalco{0, +\infty}$) on which one has $\rho_n \ge a$.
  \item The scaled sequence
    \begin{equation}
      \label{scale-epsilon-n}
      \tilde u_n(y) \coloneq \frac{1}{|u_n(x_n)|}
      \, u_n(x_n + \tilde\epsilon_n\, y),
      \qquad
      \text{defined for }
      y \in \tilde e_n\coloneq \frac{e_1 - x_n}{\tilde \epsilon_n},
    \end{equation}
    where we
    denote $\tilde\epsilon_n \coloneq |u_n(x_n)|^{-\frac{p-2}{2}}$,
    converges in $\C^1_\loc(\IR)$ to a function
    $\tilde u \in H^1(\IR)$ which is a nontrivial solution to
    \begin{equation}
      \label{eq:edo-V}
      \begin{cases}
        -\tilde u'' + \tilde\lambda \tilde u
        = \tilde\rho(y) \, \abs{\tilde u}^{p-2} \tilde u, &\text{in } \IR,\\
        1 = |\tilde u(0)|=\max_{\mathbb R}|\tilde u|,
      \end{cases}
    \end{equation}
    where $\tilde\lambda \in (0, b]$ is such that
    \begin{equation}
      \label{eq:equiv-epsilontilde-epsilon}
      \frac{\lambda_n}{\abs{u_n(x_n)}^{p-2}}
      \xrightarrow[n \to \infty]{}
      \tilde\lambda,
    \end{equation}
    and $\tilde\rho(\cdot)$ is the $\sigma(L^{\infty}, L^1)$ weak*-limit of
    $\rho_n(x_n+\tilde\epsilon_n \cdot)$.
  \end{enumerate}
  Moreover
  \begin{enumerate}[(a)]
  \item\label{limit-has-bounded-Morse-index}
    $\displaystyle 1 \le \morse(\tilde u) \le \varlimsup_{n\to\infty}
    \morse(u_n)$.
  \item\label{localisation-fn-m>=1} there exists $R > 0$ and a sequence
    $(\phi_n) \subseteq H^1(\G) \cap \Cc(\G)$ with
    $\supp \phi_n \subseteq B(x_n, R \tilde\epsilon_n)$ such that, for
    $n$ large,
    \begin{equation}
      \label{eq:Morse-index-1concentration}
      Q_{u_n}(\phi_n; \G) \coloneq
      \int_\G \big(\abs{\phi'_n}^2 + W_n(x) \phi_n^2
      + \lambda_n \phi_n^2
      - (p-1) \rho_n(x) \abs{u_n}^{p-2} \phi_n^2\big) \intd x
      < 0.
    \end{equation}
  \item\label{Lq-lim-of-mass} for all $R > 0$ and $q \ge 1$, one has
    \begin{equation}
      \label{eq:mass-concentration-Lq}
      \lim_{n\to\infty} \lambda_n^{\frac{1}{2} - \frac{q}{p-2}}
      \int_{B(x_n, R\tilde\epsilon_n)} \abs{u_n}^q \intd x
      = \tilde\lambda^{\frac12-\frac{q}{p-2}}
      \int_{B(0, R)} \abs{\tilde u}^q \intd y.
    \end{equation}
  \end{enumerate}
\end{lemma}

\begin{remark}
  \label{Rem 3.5}
  Observe that, if $\K\subseteq \G$ and $\rho_n=\alpha_n \chi_{\K}$ with
  $\alpha_n\to\alpha$ and $\alpha>0$, then, by~\ref{xn-in-e1}
  and Remark~\ref{Rem
    rho>0}, all the $x_n$ belong to one edge of $\K$ and hence
  $\rho_n(x_n+\tilde\epsilon_n y)=\alpha_n$, for all
  $y \in\tilde e_n$. As
  $\rho_n(x_n+\tilde\epsilon_n \cdot)\to \tilde\rho(\cdot)$ in the
  weak*-topology $\sigma(L^{\infty}, L^1)$ means that, for every
  $u\in L^1(-r,r)$
  \begin{equation*}
    \int_{-r}^r \rho_n(x_n+\tilde\epsilon_n y)\, u(y)\intd y
    \to \int_{-r}^r \tilde\rho(y)\, u(y)\intd y \text{,}
  \end{equation*}
  we deduce that $\tilde\rho(y)=\alpha$.
\end{remark}

\begin{proof}
  First extract subsequences such that \eqref{eq:dist->infty} becomes
  valid with a ``$\lim$'' instead of a ``$\varlimsup$'' and, for all
  $n$, $\dist(x_n, \V) \ne 0$ i.e., $x_n \notin \V$.  Since $\G$ is
  made up of finitely many edges, further restricting the subsequence
  if necessary, one can assume that all $x_n$ belong to the interior
  of the same edge $e_1$.

  First of all, notice that by Lemma~\ref{lower-bound-max|u|}, using
  assumption \ref{H}~(i) and (ii), we have
  $|u_n(x_n)|^{p-2}\ge (\lambda_n-\bar W)/b$. Hence, using also Remark
  \ref{Rem rho>0}, we get in particular
  \[
    |u_n(x_n)|\to\infty, \qquad
    \rho_n\ge a \text{\ \ on\ \ } e_1
  \]
  and
  \[
    0 < \frac{\lambda_n}{\abs{u_n(x_n)}^{p-2}}
    \le b+\frac{\bar W}{\abs{u_n(x_n)}^{p-2}}.
  \]
  Thus, passing if necessary to subsequences, we can suppose the
  existence of the limits
  \begin{equation*}
    \lim_{n \to \infty}
    \frac{\lambda_n}{\abs{u_n(x_n)}^{p-2}},
    \hspace*{4em}
    \lim_{n \to \infty} \morse(u_n)
  \end{equation*}
  and we have
  \begin{equation*}
    \tilde\lambda
    \coloneq \lim_{n \to \infty}
    \frac{\lambda_n}{\abs{u_n(x_n)}^{p-2}}
    \in \intervalcc{0,b}.
  \end{equation*}

  \medbreak

  \noindent\textit{Step 1: Existence of $\tilde u$.}
  \medskip

  For all $n$, define $\tilde u_n$ by~\eqref{scale-epsilon-n}.
  By assumptions~\eqref{tilde R_n} and~\eqref{eq:loc-max-quantified1},
  denoting $R_n \coloneq \tilde R_n/\tilde\epsilon_n$, we have
  $R_n\to\infty$,
  \begin{equation*}
    1 = \abs{\tilde u_n(0)}
    = \max_{B(0, R_n)} \abs{\tilde u_n} ,
  \end{equation*}
  and  $\tilde u_n$ satisfies the equation
  \begin{equation*}
    -\tilde u_n'' + \frac{W_n(x_n + \tilde\epsilon_ny)
    + \lambda_n }{|u_n(x_n)|^{p-2}}  \, \tilde u_n
    = \rho_n(x_n + \tilde\epsilon_n y) \abs{\tilde u_n}^{p-2} \tilde u_n
    \quad
    \text{in } \tilde e_n.
  \end{equation*}
  Moreover, due to assumption \eqref{eq:dist->infty} and the construction,
  \begin{equation}
     \label{eq:0-away-boundary}
    \dist(0, \partial \tilde e_n)
    = \dist(x_n,\V)/\tilde\epsilon_n
    = \abs{u_n(x_n)}^{\frac{p-2}{2}} \dist(x_n, \V)
    \to +\infty,
  \end{equation}
  so that every interval $\intervalcc{-r,r}$ is included in
  $\tilde e_n$ for $n$ large enough.

  Observe that, for a.e. $y\in \tilde e_n$,
  \[
    \frac{\lambda_n-\bar W}{\abs{u_n(x_n)}^{p-2}}
    \le \frac{W_n(x_n+\tilde \epsilon_n y) + \lambda_n}{\abs{u_n(x_n)}^{p-2}}
    \le \frac{\lambda_n+\bar W}{\abs{u_n(x_n)}^{p-2}},
  \]
  therefore
  \[
    \frac{W_n(x_n+\tilde \epsilon_n y) + \lambda_n}{\abs{u_n(x_n)}^{p-2}}
    \to \tilde\lambda
    \quad\text{as }  n\to\infty \qquad \text{uniformly on } \tilde e_n.
  \]
  On the other hand, given assumption~\ref{H}~\ref{H:bounds-on-rhon},
  $(\rho_n)$ is bounded in $L^\infty$ and Alaoglu
  theorem implies that, passing if necessary to a subsequence,
  \begin{equation*}
    \rho_n(x_n+\tilde\epsilon_n \cdot) \to \tilde\rho(\cdot)
    \quad \text{in the weak*-topology } \sigma(L^{\infty}, L^1),
  \end{equation*}
  where $\rho_n(x_n+\tilde\epsilon_n \cdot)$ is extended to
  $\IR$ by~$0$.

  Therefore, there exists $C>0$ such that, for all $n\in \mathbb N$
  and all $y\in \tilde e_n$, $|\tilde u_n(y)| \le 1=|\tilde u_n(0)|$ and
  $|\tilde u_n''(y)| \le C$.   By the Ascoli-Arzela theorem, we
  can find $(\tilde u_n^1)$, a subsequence of $(\tilde u_n)$ that
  converges in $\C^1([-1,1])$. Proceeding by induction for any
  $k\in \IN$, we build $(\tilde u_n^k)$ which is a subsequence of
  $(\tilde u_n^{k-1})$ and converges in $\C^1([-k,k])$.  It follows
  that the diagonal sequence
  $(\tilde u_n^{n})$ converges pointwise to a
  function $\tilde u: \IR\to\IR$ and such that, for any
  $k\in \mathbb N_0$, the convergence takes place in $\C^1([-k,k])$.
  Hence, $\tilde u$ is a solution to
  \begin{equation*}
    -\tilde u'' + \tilde\lambda \tilde u
    = \tilde\rho(y) \abs{\tilde u}^{p-2} \tilde u
    \quad
    \text{in } \IR
    \qquad\text{and}\qquad
    1 = \abs{\tilde u(0)} = \max_{\IR} \abs{\tilde u}.
  \end{equation*}

  \noindent\textit{Step 2 :
    $\morse(\tilde u) \le \varlimsup \morse(u_n) < \infty$.}
  \medbreak

  By contradiction, if
  $m \coloneq \morse(\tilde u) > \lim \morse(u_n)$, there would exist
  linearly independent functions
  $\tilde\phi_1, \dotsc,\linebreak[2] \tilde\phi_m \in H^1(\IR) \cap
  \Cc(\IR)$ such that, for every
  $\varphi\in \mbox{span}\{\tilde\phi_1, \dotsc, \tilde\phi_m\}$,
  \begin{equation*}
    Q_{\tilde u}(\phi; \IR)
    \coloneq \int_\IR \Big(\abs{\phi'}^2 + \tilde\lambda\phi^2
    - (p-1) \tilde\rho(y) |\tilde u(y)|^{p-2}|\phi|^2\Big) \intd y
    < 0.
  \end{equation*}
  For
  $\tilde\varphi\in \mbox{span}\{\tilde\phi_1, \dotsc,
  \tilde\phi_m\}$, define
  \begin{equation*}
    \phi_{n}(x)
    \coloneq \tilde \epsilon_n^{1/2} \,
    \tilde\varphi\left(\frac{x-x_n}{\tilde \epsilon_n}\right).
  \end{equation*}
  Recalling that~\eqref{eq:0-away-boundary} implies that, for any $r > 0$,
  $\intervalcc{-r,r} \subseteq \tilde e_n$ and
  that each $\tilde\phi_i$ has compact support, the functions
  $\phi_{n}$ can be regarded as functions in $H^1(e_1)$ and thus
  also in $H^1(\G)$ by extending them by $0$
  on the rest of the graph.  Now observe that, by scaling,
  \begin{align*}
    Q_{u_n}(\phi_{n}; \G)
    &= Q_{u_n}(\phi_{n}; e_1)\\
    &= \int_{e_1} \Big(\abs{\phi_n'(x)}^2
      + (W_n(x)+ \lambda_n)|\phi_n(x)|^2
      - (p-1)\rho_n(x) |u_n(x)|^{p-2}|\phi_n(x)|^2\Big) \intd x\\
    &= \int_{\tilde e_n} \Big(\abs{\tilde\varphi'(y)}^2
      + \tilde\epsilon_n^2 \bigl(W_n(x_n+\tilde\epsilon_ny)+ \lambda_n\bigr)
      |\tilde\varphi(y)|^2\\
    &\hspace*{8em}
      - (p-1)\rho_n(x_n+\tilde\epsilon_n y)
      |\tilde u_n(y)|^{p-2}|\tilde\varphi(y)|^2\Big) \intd y.
  \end{align*}
  Recall that $\varphi$ has compact support and hence
  \begin{equation*}
    \tilde\epsilon_n^2
    \bigl(W_n(x_n+\tilde\epsilon_n \cdot)+ \lambda_n\bigr)
    = \frac{W_n(x_n+\tilde\epsilon_n \cdot)
      + \lambda_n}{|u_n(x_n)|^{p-2}} \to \tilde\lambda,
    \qquad \tilde u_n \to \tilde u
    \quad\text{uniformly on } \supp(\varphi).
  \end{equation*}
  As $\rho_n(x_n+\tilde\epsilon_n \cdot)\to \tilde\rho(\cdot)$ in the
  $\sigma(L^{\infty}, L^1)$ weak*-topology, we have
  \[
    Q_{u_n}(\phi_{n};  \G)
    \rightarrow  Q_{\tilde u}(\phi;  \IR)<0.
  \]
  Considering the rescaling of
  $\tilde\phi_1, \dotsc, \tilde\phi_m \in H^1(\IR) \cap \Cc(\IR)$, we
  obtain a family of $m$ linearly independent functions
  $\phi_{1,n},\dots,\phi_{m,n}$ in $H^1(\G)$ for which we then have
  $Q_{u_n}(\phi_{n}; \G) <0$ for all
  $\phi_n \in \spanned\{\phi_{1,n},\dots,\phi_{m,n}\}$ and all $n$
  large enough.  This implies that $\morse(u_n) \ge m$, a
  contradiction and we have proved Step 2.  \medbreak

  \noindent\textit{Step 3: $\tilde\lambda>0$.}
  \medskip

  As $|\tilde u|\leq 1$, we have
  $\tilde u \in H^1_\loc(\mathbb R)\cap L^{\infty}(\IR)$ and, by
  Corollary~\ref{Liouville-lambda=0} (with $k = 2$), as $\tilde\rho \ge a>0$
  and $\max |\tilde u|= 1$, we have $\tilde\lambda>0$.  \medbreak

  \noindent\textit{Step 4: $\tilde u \in H^1(\mathbb R)$
    and  $\morse(\tilde u) \ge 1$.}
  \medskip

  By Lemma \ref{Stable via Morse}, $\tilde{u}$ is stable outside a
  compact set and hence, by Remark \ref {rem1},
  $\tilde u \in H^1(\mathbb R)$ and the information on
  $\morse(\tilde u)$ can be obtained by Lemma~\ref{Morse index>=1}.
  \medbreak

  \noindent\textit{Step 5: Proof of~\ref{localisation-fn-m>=1}.}
  \medskip

  By Step 4, there exist $R > 0$ and a function
  $\Phi \in H^1(\IR) \cap \Cc(\IR)$ such that
  $\supp \Phi \subseteq (-R, R)$ and
  \begin{equation*}
    Q_{\tilde u}(\Phi;  \mathbb R)
    = \int_\IR \abs{\Phi'}^2 + \tilde\lambda |\Phi|^2
    - (p-1)\tilde\rho(y)  \abs{\tilde u(y)}^{p-2} |\Phi|^2 \intd y
    < 0.
  \end{equation*}
  Let
  $\phi_n(x) \coloneq \tilde\epsilon_n^{1/2} \Phi\bigl( (x - x_n)
  /\tilde\epsilon_n \bigr)$.  Clearly
  $\supp \phi_n \subseteq B(x_n, R\tilde\epsilon_n)$ so for $n$ large
  enough, $\phi_n$ can be seen as a function with compact support on
  $e_1$ and extended by $0$ to a function in $H^1(\G) \cap \Cc(\G)$.
  As in Step 2, we prove that
  \begin{equation*}
    Q_{u_n}(\phi_n;  \G) =
    \int_\G \Big(\abs{\phi_n'}^2
    + (W_n(x) + \lambda_n) |\phi_n|^2
    - (p-1) \rho_n(x) \abs{u_n}^{p-2} |\phi_n|^2\Big) \intd x
    \xrightarrow[n \to \infty]{} Q_{\tilde u}(\Phi;  \mathbb R) <0,
  \end{equation*}
  whence $ Q_{u_n}(\phi_n; \G) < 0$ for $n$ large enough, proving
  \eqref{eq:Morse-index-1concentration}.

  \medbreak

  \noindent\textit{Step 6: Proof of~\ref{Lq-lim-of-mass}.}
  \medskip

  Let $R > 0$ and $q \ge 1$.  Note again that, for $n$ large enough,
  $B(x_n, R\tilde\epsilon_n) \subseteq e_1$ so, defining $\tilde u_n$
  by \eqref{scale-epsilon-n}, we have
  \[
    \lambda_n^{\frac12-\frac{q}{p-2}}\int_{B(x_n,R\tilde\epsilon_n)}|u_n|^q \intd x
    =
    \left(\frac{\lambda_n}{|u_n(x_n)|^{p-2}}
    \right)^{\frac12-\frac{q}{p-2}} \int_{B(0,R)}|\tilde u_n|^q \intd y
    \to
    \tilde \lambda^{\frac12-\frac{q}{p-2}} \int_{B(0,R)}|\tilde u|^q\intd y
  \]
  as $\tilde u_n\to \tilde u$ uniformly in $\C([-R,R])$ and
  $\frac{\lambda_n}{|u_n(x_n)|^{p-2}}\to \tilde\lambda$.
\end{proof}

We now turn to the case where the maxima stay close to vertices and
prove the analogue of Lemma~\ref{concentration at a local maximum} in
this case.

\begin{lemma}
  \label{concentration at a local maximum v2}
  Assume that $\G \in \GGfin$, $p>2$, and \ref{H} holds and let
  $(u_n) \subseteq H^1(\G) \backslash \{0 \}$ be a sequence of
  solutions to~\eqref{eq:edo}.  Suppose that
  \begin{equation*}
    \lambda_n \to +\infty \quad \text{ and } \quad
    \varlimsup_{n \to \infty} \morse(u_n) < \infty.
  \end{equation*}
  Let $(x_n) \subseteq \G$ be a sequence of local maxima of $|u_n|$
  such that, for some sequence $\tilde R_n$, we have \eqref{tilde R_n}
  and \eqref{eq:loc-max-quantified1}.  Assume moreover that
  \begin{equation}
    \label{eq:dist<infty}
    \varlimsup_{n \to \infty}
    \abs{u_n(x_n)}^{\frac{p-2}{2}}\dist(x_n, \V) < +\infty.
  \end{equation}
  Then, passing if necessary to subsequences, the following holds:
  \begin{enumerate}[(i\/$'$\!)]
  \item All the $x_n$ lie in the same edge $e_1$
    on which one has $\rho_n \ge a$,
    and $x_n \to x^* \in \V$.
  \item Let $e_2, \dotsc, e_k$ be the other edges of $\G$ having $x^*$
    as a vertex. Each $e_i$, $1 \le i \le k$, is identified
    with the (possibly unbounded) interval having origin in $0$ and
    $x^*$ is identified with $0$.  Denote
    $\tilde \epsilon_n \coloneq |u_n(x_n)|^{-\frac{p-2}{2}}$.  At the
    limit $(e_i - x^*)/\tilde\epsilon_n$ will become a half-line that
    we denote by $e_i^*$.  These half-lines will be joined at their
    origin to form a star graph $\G_k$.  Consider the scaled sequences
    \begin{equation}
      \label{scale-epsilon-n2}
      \tilde u_n(y)
      \coloneq \frac{1}{|u_n(x_n)|} u_n(x^* + \tilde\epsilon_n  y)
      \qquad
      \text{defined for }
      y \in G_n\coloneq\bigcup_{i=1}^k \frac{e_i - x^*}{\tilde\epsilon_n} .
    \end{equation}
    The sequence $(\tilde u_n)$ converges in $\C^0_\loc(\G_k)$ and in
    $\C^1_\loc(e_i^*)$, for each $i = 1,\dotsc, k$, to a function
    $\tilde u \in H^1(\G_k)$ which is a nontrivial solution to
    \begin{equation}
      \label{eq:edo-V2}
      \begin{cases}
        -\tilde u'' + \tilde \lambda \tilde u
        = \tilde\rho(y) \abs{\tilde u}^{p-2} \tilde u, &\text{in } \G_k,\\
        \tilde u \text{ is continuous on } \G_k,  \\
        \displaystyle
        \sum_{i=1}^k \dd{\tilde u_{\edge_i}}{x}(0)=0,\\
      \end{cases}
    \end{equation}
    where $\tilde\lambda \in \intervaloc{0, b}$ is defined by
    \eqref{eq:equiv-epsilontilde-epsilon} and $\tilde\rho(\cdot)$ is the
    $\sigma(L^{\infty}, L^1)$ weak*-limit of
    $\rho_n(x^*+\tilde\epsilon_n \cdot)$.
  \end{enumerate}
  Moreover properties~\ref{limit-has-bounded-Morse-index} and
  \ref{localisation-fn-m>=1}
  from Lemma~\ref{concentration at a local maximum} hold and
  \begin{enumerate}[(c\/$'$\!)]
  \item for all $R > 0$ and $q \ge 1$, one has
    \begin{equation}
      \label{eq:mass-concentration-Lq2}
      \lim_{n\to\infty} \lambda_n^{\frac{1}{2} - \frac{q}{p-2}}
      \int_{B(x_n, R\tilde\epsilon_n)} \abs{u_n}^q \intd x
      = \tilde\lambda^{\frac12-\frac{q}{p-2}}
      \int_{B(\hat x, R)} \abs{\tilde u}^q \intd y.
    \end{equation}
    where $\hat x$ is a global maximum point of $\tilde u $ located on
    $e_1^*$ whose coordinate is given by
    \begin{equation*}
      \hat x \coloneq
      \lim_{n\to \infty} \frac{\dist(x_n, x^*)}{\tilde \epsilon_n}
      \in \intervalco{0, +\infty}.
    \end{equation*}
  \end{enumerate}
\end{lemma}

\begin{remark}
  Observe that, if $\K\subseteq \G$ and $\rho_n=\alpha_n \chi_{\K}$ with
  $\alpha_n\to\alpha$ and $\alpha>0$, as in Remark \ref{Rem 3.5}, we
  show that $\tilde\rho = \alpha \chi_{\K}$.
\end{remark}

\begin{proof}
  As in the proof of Lemma \ref{concentration at a local maximum}, we
  can suppose that all $x_n$ belong to the same edge $e_1$ and we have
  \[
    |u_n(x_n)|\to\infty, \qquad \rho_n\geq a \,\,\mbox{ on }\,\,e_1.
  \]
  Passing to a subsequence, we can also assume  that
  $\lim_{n\to\infty} \morse(x_n)$ exists and
  \begin{equation*}
    \tilde\lambda
    \coloneq \lim_{n \to \infty}
    \frac{\lambda_n}{\abs{u_n(x_n)}^{p-2}}
    \in \intervalcc{0,b}.
  \end{equation*}
  \sloppy %
  By \eqref{eq:dist<infty}, we have that $(x_n)$ converges to a point
  $x^* \in e_1 \cap \V$.  Since
  $\dist(0, \partial G_n) = \min_{1\le i\le k} \ell_{e_i} /
  \tilde\epsilon_n \to +\infty$, any compact set $C \subset \G_k$ is
  included in $G_n$ for $n$ large enough.  Moreover, thanks
  to~\eqref{eq:dist<infty}, passing if necessary to a subsequence, one
  can assume that the sequence
  $\bigl(\dist(x_n, x^*) / \tilde\epsilon_n\bigr) =
  \bigl(|u_n(x_n)|^{\frac{p-2}{2}} \dist(x_n, \V) \bigr)$ converges and so
  \begin{equation*}
    \tilde x_n \coloneq \frac{x_n - x^*}{\tilde\epsilon_n}
    \xrightarrow[n \to\infty]{} \tilde x \in e^*_1.
  \end{equation*}
  Set
  $\overline{R}_n \coloneq \min\bigl\{ \dist(0, \partial
  G_n),\frac{\tilde R_n - \dist(x_n, x^*)}{\tilde\epsilon_n}\bigr\}
  \to +\infty$.  Note that, if $y \in B(0, \overline R_n)$, then
  $y \in G_n$, so $\tilde u_n(y)$ is well defined, and
  $$
  \dist(x^* + \tilde \epsilon_n y, \, x_n) \leq \tilde R_n.
  $$
  Given that $(\tilde x_n)$ converges, for $n$ large enough,
  $\tilde x_n \in B(0, \overline{R}_n)$ so that,
  assumption~\eqref{eq:loc-max-quantified1} implies
  \begin{equation}
    \label{eq:max-tilde-u2}
    1 = \abs{\tilde u_n(\tilde x_n)}
    = \max_{B(0, \overline{R}_n)} \abs{\tilde u_n}.
  \end{equation}
  Finally, $\tilde u_n$ satisfies the following equation
  \begin{equation*}
    -\tilde u_n'' + \frac{
      W_n(x^*+ \tilde\epsilon_n y) + \lambda_n}{|u_n(x_n)|^{p-2}}
    \tilde u_n
    = \rho_n(x^* + \tilde \varepsilon_n y) \abs{\tilde u_n}^{p-2} \tilde u_n
    \quad
    \text{in } G_n.
  \end{equation*}
  with the continuity and the Kirchhoff condition at $0$. Denote by
  $\tilde u_{n,i}$ the restriction of $\tilde u_n$ to
  $(e_i - x^*)/\tilde\epsilon_n$.  Therefore, as in the proof
  of Lemma \ref{concentration at a local maximum}, by Ascoli-Arzela
  theorem and passing if necessary to a subsequence, we have the
  existence of a function $\tilde u$ defined on $\G_k$ such that, for
  all $r$ and all $i = 1,\dotsc,k$, $(\tilde u_{n,i})$ converges to
  $\tilde u_i$ in $\C^1([0,r])$.  Given also that the value at $0$ of
  $\tilde u_{n,i}$ is the same on all edges, the same is true for the
  limit $\tilde u$ which is therefore well defined and continuous on
  $\G_k$.  Since this convergence takes place in $\C^1$ up to the
  origin, the Kirchhoff condition is preserved at the limit.  In all,
  the limit $\tilde u$ satisfies
  \begin{equation*}
    \begin{cases}
      -\tilde u'' + \tilde\lambda \tilde u
      = \tilde\rho(x) \abs{\tilde u}^{p-2} \tilde u
      \quad
      \text{in } \G_k, \\
      \tilde u \mbox{ continuous on } \G_k, \\
      \displaystyle
      \sum_{i=1}^k \dd{\tilde u_{\edge_i}}{x}(0)=0,\\
    \end{cases}
    \qquad\text{and}\qquad
    1 = \abs{\tilde u(\tilde x)} = \max_{\G_k} \abs{\tilde u},
  \end{equation*}
  where $\tilde\lambda \in (0, b]$ is defined by
  \eqref{eq:equiv-epsilontilde-epsilon}, $\tilde\rho(\cdot)$ is the
  $\sigma(L^{\infty}, L^1)$ weak*-limit of
  $\rho_n(x^*+\tilde\epsilon_n \cdot)$ is such that
  $\tilde\rho_{\mid e^*_1}(\cdot)\ge a$.

  We conclude following the same lines as in the proof of
  Lemma~\ref{concentration at a local maximum}.
\end{proof}

\begin{proof}[Proof of Theorem \ref{main}]
  Recall that, by Lemma \ref{lower-bound-max|u|}, for every local
  maximum $x_n^i$ of $u_n$ with $u_n(x_n^i)\not=0$, we have
  $|u_n(x_n^i)|\to\infty$.

  The proof is divided into three parts.

  \medskip\noindent %
  \textbf{Part 1. Passing if necessary to subsequences,
    there exist $m \in \{1,\dotsc,m^*\}$ and sequences
    $(x^1_n),\dotsc, (x^m_n)$,
    $(R^1_n),\linebreak[2] \dotsc,\linebreak[2] (R^m_n)$
    satisfying~\eqref{eq:dist-points}--\eqref{eq:loc-max} such that
    \begin{equation}
      \label{eq:small-residual}
      \lim_{R \to +\infty} \varlimsup_{n \to \infty}
      \Bigl( \lambda_n^{-1/(p-2)} \max_{d_n(x) \ge R \lambda_n^{-1/2}}
      \abs{u_n(x)} \Bigr) = 0,
    \end{equation}
    where $d_n(x) \coloneq \min_{1\le i \le m} \dist(x, x^i_n)$.
    \pagebreak[2]}
  \medbreak

  {\itshape Step 1: There exists $(x_n^1)$ and $(\tilde R^1_n)$ with
    $\tilde R_n^1 \,|u_n(x_n^1)|^{\frac{p-2}{2}} \to +\infty$ such
    that \eqref{eq:loc-max-quantified1} holds and the scaled sequence
    $(\tilde u^1_n)$ defined by~\eqref{scale-epsilon-n} or
    \eqref{scale-epsilon-n2}, (with $x_n = x^1_n$ and
    $\tilde\epsilon_n^1 \coloneq |u_n(x_n^1)|^{-\frac{p-2}{2}}$), according to
    the case we are in, converges in
    $\C^0_\loc(\G_{k^1}) \cap \C^1_\loc(\G_{k^1}\setminus\{0\})$ to a
    nontrivial solution $\tilde u^1\in H^1(\G_{k^1})$
    of~\eqref{eq:edo-V}, respectively of~\eqref{eq:edo-V2} (with
    $\G_{k^1} \coloneq \IR$ if \eqref{eq:dist->infty} is satisfied or
    $\G_{k^1}$ is the star graph defined as in Lemma
    \ref{concentration at a local maximum v2} if \eqref{eq:dist<infty}
    holds) with
    \begin{equation}
      \label{eq:lambda1}
      \tilde\lambda^1
      \coloneq \lim_{n \to \infty} \frac{\lambda_n}{\abs{u_n(x_n^1)}^{p-2}}
      = \lim_{n \to \infty} \lambda_n \, (\tilde\varepsilon^1_n)^2
      \in \intervaloc{0, b}
    \end{equation}
    and $\tilde\rho^1$
    defined accordingly}.  \medbreak

  As $u_n$ is in $H^1(\G)$, we know that $u_n(x) \to 0$ as
  $x \to \infty$ and the global maximum of $\abs{u_n}$ on $\G$ is
  achieved at a point $x^1_n$. Since the solution is nontrivial,
  $u_n(x^1_n) \ne 0$.  Therefore \eqref{eq:loc-max-quantified1} is
  satisfied for any sequence $\tilde R^1_n$ with
  $\tilde R_n^1 \,|u_n(x_n^1)|^{\frac{p-2}{2}} \to +\infty$.  The rest
  of the step follows from Lemmas \ref{concentration at a local
    maximum} and \ref{concentration at a local maximum v2}.  \medbreak

  \textit{Step 2: If \eqref{eq:small-residual} is not satisfied for
    $m = 1$, there exists $(x_n^2)$ and $(\tilde R^2_n)$ with
    $\tilde R_n^2 \,|u_n(x_n^2)|^{\frac{p-2}{2}} \to +\infty$ such
    that $\lambda_n^{1/2} \dist(x^1_n, x^2_n) \to +\infty$,
    $\abs{u_n(x_n^2)} = \max_{B(x_n^2, \tilde R_n^2)}
    \abs{u_n}$ and the scaled sequence $(\tilde u^2_n)$ defined
    by~\eqref{scale-epsilon-n} or \eqref{scale-epsilon-n2}, (with
    $x_n = x^2_n$ and
    $\tilde\epsilon_n^2 \coloneq |u_n(x_n^2)|^{-\frac{p-2}{2}}$), according to
    the case we are in, converges in
    $\C^0_\loc(\G_{k^2}) \cap \C^1_\loc(\G_{k^2}\setminus\{0\})$ to a
    nontrivial solution $\tilde u^2\in H^1(\G_{k^2})$
    of~\eqref{eq:edo-V}, respectively of~\eqref{eq:edo-V2} (with
    $\G_{k^2} \coloneq \IR$ if \eqref{eq:dist->infty} is satisfied or
    $\G_{k^2}$ is the star graph defined as in Lemma
    \ref{concentration at a local maximum v2} if \eqref{eq:dist<infty}
    holds) with
    $\displaystyle\tilde\lambda^2 \coloneq \lim_{n \to \infty}
    \frac{\lambda_n}{\abs{u_n(x_n^2)}^{p-2}}
    = \lim_{n \to \infty} \lambda_n (\tilde\varepsilon^2_n)^2
    \in \intervaloc{0,b}$ and $\tilde\rho^2$
    defined accordingly}.  \medbreak

  If \eqref{eq:small-residual} is not satisfied for $m = 1$, the
  following holds:
  \begin{equation*}
    \varlimsup_{R \to +\infty} \varlimsup_{n \to \infty}
    \Bigl( \lambda_n^{-1/(p-2)} \max_{\dist(x, x^1_n) \ge R \lambda_n^{-1/2}}
    \abs{u_n(x)} \Bigr) =: 4\delta > 0.
  \end{equation*}
  Thus, there exists $R$ as large as we want such that, passing if
  necessary to a subsequence, we have
  \begin{equation}
    \label{eq:remaining-mass1}
    \forall n,\qquad
    \lambda_n^{-1/(p-2)} \max_{\dist(x, x^1_n) \ge R \lambda_n^{-1/2}}\abs{u_n(x)}
    \ge 2\delta.
  \end{equation}
  Taking a larger $R$ if necessary, we may also assume that
  \begin{equation}
    \label{eq:small-V1}
    \forall y \in \G_{k^1},\qquad
    \dist(y, \hat x^1) \ge \frac{R}{(\tilde \lambda^1)^{1/2}} \ \limplies\
    \abs{\tilde u^1(y)}
    \le \delta \, \bigl(\tilde\lambda^1\bigr)^{1/(p-2)},
  \end{equation}
  where $\hat x^1 \coloneq 0$ if we applied Lemma \ref{concentration at
    a local maximum} or $\hat x^1$ is a maximum point of $\tilde u^1$
  if we applied Lemma \ref{concentration at a local maximum v2}.  Again
  because $u_n$ vanishes at infinity, for all $n\in \mathbb N$, there
  exists $x^2_n$ such that
  $\dist(x^2_n, x^1_n) \ge R \lambda_n^{-1/2}$ and
  \begin{equation}
    \label{eq:max-x2n}
    \abs{u_n(x^2_n)}
    = \max_{\dist(x, x^1_n) \ge R  \lambda_n^{-1/2}} \abs{u_n(x)}
    > 0.
  \end{equation}

  \noindent\textit{Claim :
    $\dist(x^2_n, x^1_n) \, \lambda_n^{1/2} \to +\infty$.}

  Indeed, if up to a subsequence
  $\dist(x^2_n, x^1_n)  \lambda_n^{1/2} \to R' \in [R,\infty)$,
  then, using~\eqref{eq:lambda1},
  \[
    \frac{\dist(x^2_n, x^1_n)}{\tilde\varepsilon^1_n}
    =
    \dist(x^2_n, x^1_n) \, \lambda_n^{1/2}\,
    \frac{1}{\lambda_n^{1/2} \, \tilde\varepsilon^1_n}
    \to \frac{R'}{(\tilde \lambda^1)^{1/2}},
  \]
  thus $\dist(x^2_n, x^1_n) \to 0$ and either $x^2_n$ would lie on the same
  edge as $x^1_n$ when \eqref{eq:dist->infty} holds, or on an edge
  adjacent to $x^{*,1} \coloneq \lim x^1_n$ when \eqref{eq:dist<infty}
  holds.
  Observe also that, given~\eqref{eq:lambda1},
  \[
    \varlimsup_{n \to \infty}
    \lambda_n^{1/2} \dist(x_n^1, \V) = +\infty
    \qquad
    \text{ if and only if }
    \qquad
    \varlimsup_{n \to \infty}
    \abs{u_n(x_n^1)}^{(p-2)/2} \dist(x_n^1, \V) = +\infty.
  \]
  Therefore, we can define the sequence $(\hat x^2_n)$ by
  \begin{numcases}{\hat x^2_n \coloneq}
    \label{eq:x2n-1}
    \displaystyle
    \frac{x^2_n - x^1_n}{\tilde\varepsilon^1_n} 
    &$\displaystyle\text{if } \varlimsup_{n \to \infty}
    \lambda_n^{1/2} \dist(x_n^1, \V) = +\infty$,\\
    \label{eq:x2n-2}
    \displaystyle
    \frac{x^2_n - x^{*,1}}{\tilde\varepsilon^1_n}
    &$\displaystyle\text{if } \varlimsup_{n \to \infty}
    \lambda_n^{1/2} \dist(x_n^1, \V) < +\infty$.
  \end{numcases}
  Recall that in the second case
  $\lim_{n\to\infty} \dist(x_n^1,x^{*,1}) / \tilde\varepsilon^1_n
  = \hat x^1\in \intervalco{0,+\infty}$ on $e_1^*$.

  Recall that
  $\dist(x^2_n, x^1_n) \lambda_n^{1/2} \to R' \in [R,\infty)$ and
  hence, up to a subsequence, $\hat x^2_n \to \hat x^2 \in \G_{k^1}$
  such that
  $\dist(\hat x^2, \hat x^1) \ge R / (\tilde\lambda^1)^{1/2}$, but then
  \eqref{eq:remaining-mass1}--\eqref{eq:small-V1} yield the
  following contradiction:

  In case~\eqref{eq:x2n-1} holds:
  \begin{equation*}
    2\delta \le \lambda_n^{-1/(p-2)} \abs{u_n(x^2_n)}
    =\lambda_n^{-1/(p-2)} \abs{u_n(x^1_n)}\,
    \tilde u_n(\hat x^2_n)
    \xrightarrow[n \to \infty]{}
    \bigl(\tilde\lambda^1\bigr)^{-\frac{1}{p-2}}
    \abs{\tilde u^1(\hat x^2)}
    \le \delta.
  \end{equation*}
  
  In case~\eqref{eq:x2n-2} holds:
  \begin{equation*}
    2\delta \le \lambda_n^{-1/(p-2)} \abs{u_n(x^2_n)}
    = \lambda_n^{-1/(p-2)} \abs{u_n(x^1_n)}\,\,\,
    \tilde u_n(\hat x^2_n)
    \xrightarrow[n \to \infty]{}
    \bigl(\tilde\lambda^1\bigr)^{-\frac{1}{p-2}}
    \abs{\tilde u^1(\hat x^2)}
    \le \delta.
  \end{equation*}
  This proves the claim.
  \medbreak

  Let $\tilde R^2_n \coloneq \tfrac{1}{2} \dist(x^2_n, x^1_n)$.
  Clearly, for all $x$ such that $\dist(x, x^2_n) < \tilde R^2_n$,
  we have
  \begin{equation*}
    \dist(x, x^1_n)
    \ge \dist(x^1_n, x^2_n) - \dist(x, x^2_n)
    > \dist(x^1_n, x^2_n) - \tilde R^2_n
    = \tilde R_n^2.
  \end{equation*}
  The above claim reads $\tilde R^2_n \lambda_n^{1/2} \to +\infty$.
  Thus, for $n$ large enough, $\tilde R_n^2 \ge R \lambda_n^{-1/2}$ and
  $B(x^2_n, \tilde R^2_n) \subseteq \{ x \mid \dist(x, x^1_n)
  \ge R\lambda_n^{-1/2} \}$ and, thanks to~\eqref{eq:max-x2n},
  assumption~\eqref{eq:loc-max-quantified1} holds with
  $x_n = x^2_n$.  Moreover, using~\eqref{eq:lambda1},
  assumption~\eqref{tilde R_n} is also satisfied.
  Therefore, the scaled sequence $(\tilde u^2_n)$ (using
  either~\eqref{scale-epsilon-n} or~\eqref{scale-epsilon-n2} depending
  on the behavior of $(x^2_n)$) converges in
  $\C^0_\loc(\G_{k^2}) \cap \C^1_\loc(\G_{k^2} \setminus\{0\})$ to a
  function $\tilde u^2 \in H^1(\G_{k^2})$ which satisfies
  either~\eqref{eq:edo-V} or~\eqref{eq:edo-V2} with the corresponding
  parameters.
  \medbreak
  
  \textit{Step 3: Suppose we have $(x^1_n), \dotsc, (x^s_n)$
    satisfying \eqref{eq:dist-points}--\eqref{eq:loc-max} and such
    that \eqref{eq:small-residual} does not hold with $m = s$.  Then,
    there exists $(x_n^{s+1})$ and $(\tilde R^{s+1}_n)$ with
    $\tilde R_n^{s+1} \,|u_n(x_n^{s+1})|^{\frac{p-2}{2}} \to +\infty$
    such that $\lambda_n^{1/2} \dist(x^i_n, x^{s+1}_n) \to +\infty$,
    for all $i\in\{1,\dotsc,s\}$ ,
    $\abs{u_n(x_n^{s+1})} = \max_{B(x_n^{s+1}, \tilde
      R_n^{s+1})} \abs{u_n}$ and the scaled sequence
    $(\tilde u^{s+1}_n)$ defined by~\eqref{scale-epsilon-n} or
    \eqref{scale-epsilon-n2} (with $x_n = x^{s+1}_n$ and and
    $\tilde\epsilon_n^{s+1} \coloneq |u_n(x_n^{s+1})|^{-\frac{p-2}{2}}$),
    according to the case we are in, converges in
    $\C^0_\loc(\G_{k^{s+1}}) \cap
    \C^1_\loc(\G_{k^{s+1}}\setminus\{0\})$ to a nontrivial solution
    $\tilde u^{s+1}\in H^1(\G_{k^{s+1}})$ of~\eqref{eq:edo-V},
    respectively of~\eqref{eq:edo-V2} (with $\G_{k^{s+1}} = \IR$
    if \eqref{eq:dist->infty} is satisfied or the star graph defined
    as in Lemma \ref{concentration at a local maximum v2} if
    \eqref{eq:dist<infty} holds) and the corresponding parameters}.
  \medbreak

  If we have $(x^1_n), \dotsc, (x^s_n)$ such that
  \eqref{eq:small-residual} does not hold with $m = s$, as previously,
  there exist values of $R$ as large as we like such that, passing if
  necessary to subsequences,
  \begin{equation*}
    \forall n,\qquad
    \lambda_n^{-1/(p-2)} \max_{d^s_n(x) \ge R \lambda_n^{-1/2}}\abs{u_n(x)}
    \ge 2\delta > 0.
  \end{equation*}
  where $d^s_n(x) \coloneq \min_{1\le i \le s} \dist(x, x^i_n)$.  Taking $R$
  larger if necessary, we may also assume that
  \begin{equation*}
    \forall i = 1,\dotsc, s,\  \forall y \in \G_{k^i},\qquad
    \dist(y, \hat x^i) \ge \frac{R}{(\tilde\lambda^i)^{1/2}}
    \ \limplies\
    \abs{\tilde u^i(y)} \le \delta \bigl(\tilde\lambda^i\bigr)^{1/(p-2)}.
  \end{equation*}
  Consider $(x^{s+1}_n)$ a sequence such that
  \begin{equation*}
    \forall n,\qquad
    \abs{u_n(x^{s+1}_n)}
    = \max_{d^s_n(x) \ge R \lambda_n^{-1/2}} \abs{u_n(x)} > 0.
  \end{equation*}
  For all $i = 1,\dotsc,s$, we establish as before that
  $\dist(x^{s+1}_n, x^i_n)\lambda_n^{1/2} \to +\infty$. In this way,
  \eqref{eq:dist-points} holds for $(x^1_n), \dotsc, (x^{s+1}_n)$.
  Set
  $\tilde R^{s+1}_n \coloneq \tfrac{1}{2} d^s_n(x^{s+1}_n)$.
  We then have $R^{s+1}_n \coloneq \tilde R^{s+1}_n \lambda_n^{1/2}
  \to +\infty$.  For all $i = 1,\dotsc, s$ and
  for $n$ large, the same argument as above shows that, for all
  $x \in \G$,\linebreak[2]
  $\dist(x, x^{s+1}_n) < \tilde R^{s+1}_n
  \limplies \dist(x,x^i_n) > \tilde R^{s+1}_n$.  Therefore,
  for $n$ large enough, $\tilde R^{s+1}_n \ge R \lambda_n^{-1/2}$ and
  $B(x^{s+1}_n, \tilde R^{s+1}_n) \subseteq \{ x \mid
  d^s_n(x) \ge R \lambda_n^{-1/2} \}$
  and assumption~\eqref{eq:loc-max-quantified1} of
  Lemma~\ref{concentration at a local maximum} is satisfied with
  $x_n \coloneq x^{s+1}_n$, i.e.\ \eqref{eq:loc-max} holds for
  $i = s+1$. As a result, the
  scaled sequence $(\tilde u^{s+1}_n)$ (using
  either~\eqref{scale-epsilon-n} or~\eqref{scale-epsilon-n2} depending
  on the behavior of $(x^{s+1}_n)$) converges in
  $\C_\loc^0(\G_{k^{s+1}}) \cap \C_\loc^1(\G_{k^{s+1}}
  \setminus\{0\})$ to a function
  $\tilde u^{s+1} \in H^1(\G_{k^{s+1}}) \setminus \{0\}$ that
  satisfies either~\eqref{eq:edo-V} or~\eqref{eq:edo-V2} with the
  corresponding parameters.  \medbreak

  \textit{Step 4: This construction cannot be pushed to an $m > m^*$.}
  \medbreak

  In order to do so, remember that, for all
  $i = 1,\dotsc, m$, the consequence~\ref{localisation-fn-m>=1}
  of Lemmas~\ref{concentration
    at a local maximum} or~\ref{concentration at a local maximum v2}
  yields a sequence of functions $(\phi^i_n)$ such that (recalling
  \eqref{eq:equiv-epsilontilde-epsilon})
  \begin{equation*}
    \forall n,\quad
    \supp \phi^i_n \subseteq B(x^i_n, R\lambda_n^{-1/2})
  \end{equation*}
  for some $R$ (taking $R$ sufficiently large so it is independent of
  $i$) and~\eqref{eq:Morse-index-1concentration} holds. Due
  to~\eqref{eq:dist-points}, the supports of the functions $\phi^i_n$,
  $i = 1, \dotsc, m$ are disjoint for $n$ large enough.  As a
  consequence, \eqref{eq:Morse-index-1concentration} implies that, for
  $n$ large,
  \begin{equation*}
    \forall \phi \in \spanned\{\phi^1_n,\dotsc, \phi^m_n\}
    \setminus\{0\},\quad
    \int_\G \abs{\phi'}^2 + \bigl(\lambda_n+ W_n(x)\bigr) \phi^2
    - (p-1) \rho_n(x) \abs{u_n}^{p-2} \phi^2 \intd x
    < 0,
  \end{equation*}
  a contradiction with the assumption that $\morse(u_n) \le m^*$.

  \medbreak
  \noindent %
  \textbf{Part 2.  Proof of the bound~\eqref{eq:exp-bound}.}

  Given $R > 0$ and $n \in \mathbb{N}$, let
  $\displaystyle A_{R,n} \coloneq \bigl\{x \in \G \bigm| d_n(x) \ge R
    \lambda_n^{-1/2} \bigr\} = \bigcap_{i=1}^m \complement
  B\bigl(x^i_n, R\lambda_n^{-1/2}\bigr)$.

  \medbreak
  \textit{Step~1: Let us prove that}
  $\forall \epsilon > 0,\
    \exists R_\epsilon,\ \forall R \ge R_\epsilon,\
    \exists n_R,\ \forall n \ge n_R,\ \forall x \in \G,$
  \begin{equation}
    \label{eq:exp-bound-explicit}
    x \in A_{R,n} \quad \limplies \quad  \abs{u_n(x)}
    \le \epsilon \e^{R} \lambda_n^{1/(p-2)}
    \exp\bigl( -\tfrac{1}{2} \lambda_n^{1/2} d_n(x) \bigr).
  \end{equation}
  
  Let $\epsilon > 0$.  Taking $n$ sufficiently large, we may assume
  without loss of generality that
  \begin{equation}
    \label{eq:condition-epsilon}
    \epsilon^{p-2} \le   \frac{2^{p-4}}{b}
    \quad \text{and} \quad
    4 \bar W\leq \lambda_n
  \end{equation}
  where $\bar W$ and $b$ are given by assumption \ref{H}.
  By~\eqref{eq:small-residual}, there exists $R_\epsilon\geq 8$ such
  that, for all $R \ge R_\epsilon$, there is an $n_R$ for which
  \begin{equation}
    \label{eq:u-small-An}
    \forall n \ge n_R,\ \forall x \in \G,\qquad
    x \in A_{R,n}\ \limplies \
    \abs{u_n(x)} \le \tfrac{1}{2} \epsilon \lambda_n^{1/(p-2)}.
  \end{equation}
  Moreover, we can take $R_\epsilon$ larger if necessary so that, for
  all $i = 1,\dotsc, m,$
  \begin{equation}
    \label{eq:R-large-take-vertices}
    d^i_\infty
    \coloneq \varlimsup_{n \to \infty} \lambda_n^{1/2} \dist(x^i_n, \V)
    < \infty
    \ \limplies\
    R_\epsilon \ge 2 d^i_\infty.
  \end{equation}

  In the aim to establish \eqref{eq:exp-bound-explicit}, we first need
  to describe more precisely the set $A_{R,n}$. First note that,
  passing if necessary to sub-se\-quen\-ces, we can assume that, if
  two sequences of points $(x^i_n)$ and $(x^j_n)$, $i \ne j$, lie in
  the same edge of $\G$, then either $\forall n,\ x^i_n < x^j_n$ or
  vice versa.  For each edge $e$ of $\G$, observe that
  $e \setminus B(x^i_n, R\lambda_n^{-1/2})$ may remove at most two
  intervals from $e$ (two intervals may be removed only if
  $e$ is
  a loop, possibly choosing $n_R$ larger).  Thus
  $e \cap A_{R,n}= e \setminus \bigcup_{i=1}^m B(x^i_n,
  R\lambda_n^{-1/2})$ consists of finitely many closed intervals or
  closed half-lines, the number of which is bounded independently of
  $R$ and
  $n$. Adding (if not already present) vertices at the boundary of
  these intervals makes $A_{R,n}$ a (possibly disconnected) graph.


  Let $e_n$ be any bounded edge of $A_{R,n}$ of length $\ell_{e_n}$. We
  will now show that, by possibly taking $n_R$ larger, one has
  \begin{equation}
    \label{eq:length-lower-bound}
    \forall n \geq n_R,\quad
    \lambda_n^{1/2} \ell_{e_n} \geq 4.
  \end{equation}
  Indeed, let $e=\intervalcc{\vv_{0}, \vv_{1}}$ be an edge of $\G$
  delimited by the vertices $\vv_0$ and $\vv_1$ and
  $e_n = \intervalcc{\vv_{0,n}, \vv_{1,n}}$ be a connected component
  of $e \cap A_{R,n}$ (i.e.\ an edge of $A_{R,n}$) of finite length
  $\ell_{e_n}$.  Let us distinguish several cases.
  \begin{itemize}
  \item If
    $\abs{\vv_0 - \vv_{0,n}} \le R \lambda_n^{-1/2}$ and
    $\abs{\vv_1 - \vv_{1,n}} \le R \lambda_n^{-1/2}$,
    then $\ell_{e_n} \ge \ell_e - 2R\lambda_n^{-1/2}$ and so
    $\lambda_n^{1/2} \ell_{e_n} \to +\infty$ as $n \to \infty$. This
    is in particular true if $e_n = e$.
  \item If
    $e_n = \intervalcc{x^i_n + R\lambda_n^{-1/2}, \, x^j_n -
      R\lambda_n^{-1/2}}$ where $x^i_n \in e$ and $x^j_n \in e$
    are such that $\forall n,\ x^i_n < x^j_n$, then
    $\ell_{e_n} = \abs{x^i_n - x^j_n} - 2R\lambda_n^{-1/2} \ge
    \dist(x^i_n, x^j_n) - 2R\lambda_n^{-1/2}$ and so, thanks
    to~\eqref{eq:dist-points}, one again has that
    $\lambda_n^{1/2} \ell_{e_n} \to +\infty$.
  \item The remaining case is
    $e_n = \intervalcc{\vv_{0,n}, x^i_n - R\lambda_n^{-1/2}}$ where
    $\abs{\vv_0 - \vv_{0,n}} \le R\lambda_n^{-1/2}$ and $x^i_n \in e$
    (or the similar case near $\vv_1$).  Then
    \begin{equation}
      \label{eq:length-near-bd}
      \ell_{e_n} = \abs{\vv_0 - x^i_n}
      - \abs{\vv_0 - \vv_{0,n}} - R\lambda_n^{-1/2}.
    \end{equation}
    We claim that necessarily
    $\lambda_n^{1/2} \abs{\vv_0 - x^i_n} \to +\infty$.  If not,
    passing to a subsequence if necessary, $\abs{\vv_0 - x^i_n} \to 0$
    and so $\dist(x^i_n, \V) = \abs{\vv_0 - x^i_n}$.  Indeed,
    $\dist(x^i_n, \V) = \abs{\vv_0 - x^i_n}$ holds as soon as
    $\abs{\vv_0 - x^i_n} \le \underline{\ell}/2$
    where $\underline{\ell} > 0$ is the minimal length
    of an edge.  Hence we have
    $\displaystyle \varlimsup_{n \to \infty} \lambda_n^{1/2}
    \dist(x^i_n, \V)<\infty$.  Now,
    using~\eqref{eq:length-near-bd}
    and~\eqref{eq:R-large-take-vertices}, one deduces that
    \begin{math}
      \ell_{e_n} \le \dist(x^i_n, \V) - R\lambda_n^{-1/2}
      \le \bigl(d^i_\infty + o(1)\bigr) \lambda_n^{-1/2}
      - R \lambda_n^{-1/2}
      \le \bigl(d^i_\infty + o(1) - R_\varepsilon \bigr) \lambda_n^{-1/2}
      < 0
    \end{math}
    for $n$ large, a contradiction.  Therefore
    $\lambda_n^{1/2} \ell_{e_n} \ge \lambda_n^{1/2} \abs{\vv_0 -
      x^i_n} - 2R \to +\infty$.
  \end{itemize}
  Since the number of edges of $A_{R,n}$ is bounded independently of
  $n$, we can take $n_R$ large enough so that
  \eqref{eq:length-lower-bound} holds in all possible cases.

  Next, we will further split some edges of $A_{R,n}$ to obtain a
  simple relationship between local coordinates on each edge and the
  global distance function $d_n$ defined in~\eqref{eq:exp-bound}.
  Let $e_n$ be an edge of $A_{R,n}$.  We need to distinguish two
  cases.
  \begin{itemize}
  \item The edge $e_n = \intervalco{\vv_{0,n},+\infty}$ is unbounded.
    We identify it with $\intervalco{0, +\infty}$.  In this case none
    of the $x^i_n$,
      $i = 1,\dotsc, m$, can belong to $e_n$ (but may belong to
      $e\setminus e_n$ with $e$ being the infinite edge of $\G$
      containing $e_n$). Thus, all paths joining a $x \in e_n$ to a
      $x_n^i$ must go through the sole node $\vv_{0,n}$ of $e_n$,
      identified with 0. In all, for all $x \in e_n$,
      $d_n(x) = d_n(0) + x$.
    \item The edge $e_n=[\vv_{0,n},\vv_{1,n}]$ is bounded.  We
      identify it with $\intervalcc{0, \ell_{e_n}}$ so that
      $\vv_{0,n}$ has coordinate $0$ and $\vv_{1,n}$ has coordinate
      $\ell_{\e_n}$.  The paths from any $x \in e_n$ to a $x^i_n$,
      $i = 1,\dotsc,m$, must exit $e_n$ by one of its vertices and
      thus
    \begin{equation}
      \label{eq:varying distance}
      d_n(x) = \min\{ d_n(\vv_{0,n}) + x,\,
      d_n(\vv_{1,n}) + \ell_{e_n} - x \}.
    \end{equation}
    If regardless of the $x \in e_n$ chosen the quantity $d_n(x)$ is
    always given by $d_n(\vv_{0,n})+x$ we leave $e_n$ unchanged. If
    $d_n(x)$ is always given by $d_n(\vv_{1,n}) + \ell_{e_n} - x$ we
    parametrize $e_n$ so that $\vv_{1,n}$ has coordinate $0$ and
    $\vv_{0,n}$ has coordinate $\ell_{e_n}$.  Now, if $d_n(x)$ changes
    its expression then there exists a unique coordinate $\tilde x_n$
    such that $d_n(x) = d_n(\vv_{0,n}) + x$ for
    $x \in \intervalcc{0, \tilde x_n}$ and
    $d_n(x) = d_n(\vv_{1,n}) + \ell_{e_n} - x$ for
    $x \in \intervalcc{\tilde x_n, \ell_{e_n}}$.  If the length of
    both $\intervalcc{0, \tilde x_n}$ and
    $\intervalcc{\tilde x_n, \ell_{e_n}}$ is greater or equal to
    $4\lambda_n^{-1/2}$, we add a new node $\vv_{2,n}$ of degree $2$
    at the point of coordinate $\tilde x_n$ splitting the edge $e_n$
    into $\intervalcc{\vv_{0,n}, \vv_{2,n}}$ and
    $\intervalcc{\vv_{1,n}, \vv_{2,n}}$ (the latter being parametrized
    so that $\vv_{1,n}$ has coordinate~$0$).  On these two
    (reparametrized) new edges, the distance function $d_n$ takes the
    form:
    \begin{equation}
      \label{eq:appropriate-local-coordinates}
      d_n(x) = d_n(0) + x.
    \end{equation}
    If one of the lengths is less than $4\lambda_n^{-1/2}$, we leave
    $e_n$ unchanged. Note that for this case we will still have to
    deal with \eqref{eq:varying distance}.
  \end{itemize}

  The above procedure ensures that \eqref{eq:length-lower-bound}
  remains valid for the new edges.  We still denote $A_{R,n}$ the graph
  with these additional degree $2$ nodes. \medskip

  We may finally turn to proving the inequality
  \eqref{eq:exp-bound-explicit}. Define the linear operator
  $L_n(u) \coloneq - u'' + \tilde W_n(x) u$ where
  $$
  \tilde W_n(x) \coloneq W_n(x) + \lambda_n -\rho_n(x)
  \abs{u_n(x)}^{p-2}.
  $$
  The fact that $u_n$ is a solution to~\eqref{eq:edo} can be written
  as $L_n(u_n) = 0$. For all $R \ge R_\epsilon$ observe that
  \eqref{eq:condition-epsilon} and property~\eqref{eq:u-small-An}
  imply that
  \begin{equation}
    \label{eq:an>=1/2}
    \forall n \ge n_R,\ \forall x \in A_{R,n}, \quad
    \tilde W_n(x) \ge \lambda_n
    \Bigl(1 - \frac{1}{4}- \rho_n(x) \frac{1}{2^{p-2}} \epsilon^{p-2} \Bigr)
    \ge \frac{\lambda_n}{2}.
  \end{equation}

  Fix $n$ large enough and let $\A_{R,n}$ be a connected component of
  $A_{R,n}$.
  If $\A_{R,n}$ is a single point, then $\A_{R,n}$ is an isolated
  point $x$ of $\partial A_{R,n}$ and so
  $d_n(x) = R \lambda_n^{-1/2}$.  In view of \eqref{eq:u-small-An} we
  have
  \begin{equation*}
    \abs{u_n(x)}
    \le \tfrac{1}{2} \epsilon \lambda_n^{1/(p-2)}
    \le \epsilon \e^{R/2} \lambda_n^{1/(p-2)}
    \exp\bigl( -\tfrac{1}{2} \lambda_n^{1/2} d_n(x) \bigr)
  \end{equation*}
  and \eqref{eq:exp-bound-explicit} holds.  Now assume $\A_{R,n}$
  possesses at least one edge.  We will construct below a positive
  function
  $\phi \in \C(\A_{R,n}) \cap L^\infty(\A_{R,n}) \cap L^2(\A_{R,n})$
  of class $\C^ \infty$ on each edge of $\A_{R,n}$ (remember that we
  have added degree $2$ vertices) such that
  \begin{enumerate}[(i)]
  \item\label{L>=0}
    $L_n(\phi) \ge 0$ on each edge of $\A_{R,n}$;
  \item\label{neg-at-vertices}
    $\sum_{e \incident \vv}
    \dd{\varphi_{\edge}}{x}(\vv)
    \leqslant 0$ for each vertex
    $\vv$ of $\A_{R,n}$;
  \item\label{phi-lower-bd}
    $\phi(x) \ge \e^{-R/2}$ for all
    $x \in \partial\A_{R,n}$;
  \item\label{phi-upper-bd}
    $\phi(x) \le \e^{R/2}
    \exp\bigl( -\tfrac{1}{2} \lambda_n^{1/2} d_n(x) \bigr)$
    for all $x \in \A_{R,n}$.
  \end{enumerate}
  Assuming we have such a function $\phi$, let
  $\psi_{\pm} \coloneq \epsilon \lambda_n^{1/(p-2)} \e^{R/2} \phi \pm
  u_n$.  Thanks to~\ref{L>=0},
  $$
    L_n(\psi_{\pm}) = \epsilon \lambda_n^{1/(p-2)} \e^{R/2} L_n(\phi) \ge 0.
  $$
  Since $\A_{R,n} \ne \G$ and $\G$ is connected,
  we have that  $\partial \A_{R,n} \ne \emptyset$.
  For any $\vv \in \partial\A_{R,n}$ (a vertex of
  the graph $\A_{R,n}$), \eqref{eq:u-small-An} and \ref{phi-lower-bd}
  imply that
  \begin{equation*}
    \psi_{\pm}(\vv)
    = \epsilon \lambda_n^{1/(p-2)} \e^{R/2} \phi(\vv) \pm u_n(\vv)
    \ge \epsilon \lambda_n^{1/(p-2)} \pm u_n(\vv)
    \ge \tfrac{1}{2}\epsilon \lambda_n^{1/(p-2)} > 0.
  \end{equation*}
  For all other vertices $\vv$ of $\A_{R,n}$, including the added
  ones of degree~$2$, $u_n{}_{\mid\A_{R,n}}$ satisfies the Kirchhoff
  condition, whence,
  thanks to \ref{neg-at-vertices},
  $\sum_{e \incident \vv} \dd*{\psi_{\pm,e}}{x}(\vv) \le 0$.
  Moreover,
  $\psi_{\pm} \in L^2(\A_{R,n})$.
  In view of Proposition~\ref{order-preserving-principle}, as
  $\tilde W_n > 0$ by \eqref{eq:an>=1/2} and
  $\psi_{\pm}\in L^{\infty}(\A_{R,n})$, we have $\psi_{\pm} \ge 0$ on
  $\A_{R,n}$, i.e.
  \begin{equation*}
    \abs{u_n} \le \epsilon \lambda_n^{1/(p-2)} \e^{R/2} \phi
    \quad \text{on } \A_{R,n}.
  \end{equation*}
  Since one can perform this argument on each connected component
  $\A_{R,n}$ of $A_{R,n}$, \eqref{eq:exp-bound-explicit} is a
  consequence of \ref{phi-upper-bd}.  \medbreak

  To conclude step~1, there only remains to
  construct the function $\phi : \A_{R,n} \to \IR$.  It is defined by
  the following rules:
  \begin{itemize}
  \item for each vertex $\vv$ of $\A_{R,n}$, set
    $\phi(\vv) \coloneq \exp\bigl( -\frac{1}{2} \lambda_n^{1/2}
    d_n(\vv) \bigr)$;
  \item on each bounded edge $e_n$ of $\A_{R,n}$ parametrized by
    $\intervalcc{0, \ell_{e_n}}$, $\phi$ is the function given by
    Lemma~\ref{lem:exp-bound} with $\ell = \ell_{e_n}$,
    $\phi_0 = \phi(0)$ and $\phi_1 = \phi(\ell_{e_n})$ if
    $\phi(0) \ge \phi(\ell_{e_n})$ or vice-versa if
    $\phi(0) < \phi(\ell_{e_n})$;
  \item for each unbounded edge $e_n$ of $\A_{R,n}$ identified with
    $\intervalco{0, \infty}$,
    \begin{equation}
      \label{eq:bound-half-line}
      \phi(x)
      \coloneq \exp\bigl(-\tfrac{1}{2} \lambda_n^{1/2} (d_n(0) + x) \bigr)
      = \exp\bigl(-\tfrac{1}{2} \lambda_n^{1/2} d_n(x) \bigr).
    \end{equation}
  \end{itemize}
  It is clear that $\phi$ is continuous and bounded on $\A_{R,n}$,
  $\phi\in L^2(\A_{R,n})$, and $\phi$ is of class $\C^\infty$ on each
  edge of $\A_{R,n}$.  At each vertex $\vv$ of an edge $e_n$, whether
  bounded or not, the above choices imply that
  $ \dd{\phi_{\edge_n}}{x}(\vv) \le 0$ and so
  property~\ref{neg-at-vertices} holds.
  Given that, for all
  $x \in \partial\A_{R,n} \subseteq \partial A_{R,n}$ we have
  $d_n(x) = R \lambda_n^{-1/2}$, the first rule implies
  property~\ref{phi-lower-bd}.

  Now, on a bounded edge $e_n$, $\phi'' = \alpha^2_{e_n} \phi$ with,
  for $n$ large enough,
  \begin{align*}
    \alpha_{e_n}
    &\coloneq \frac{1}{\ell_{e_n}} \cosh^{-1} \frac{\phi(\vv_0)}{\phi(\vv_1)}
      = \frac{1}{\ell_{e_n}} \cosh^{-1} \exp\bigl(\tfrac{1}{2} \lambda_n^{1/2}
    (d_n(\vv_1) - d_n(\vv_0)) \bigr)
    \notag \\
    & \le \frac{1}{\ell_{e_n}} \cosh^{-1} \exp\bigl(\tfrac{1}{2} \lambda_n^{1/2}
      \ell_{e_n} \bigr)
      \le 0.7\, \lambda_n^{1/2}.
  \end{align*}
  where the inequalities use the fact that
  $d_n(\vv_1) - d_n(\vv_0) \le \ell_{e_n}$, that  
  $\xi \mapsto \cosh^{-1}(\exp(\xi))$ is increasing, and finally
  the lower bound provided by~\eqref{eq:length-lower-bound}
  together with
  $\forall \xi \in \intervalco{2, +\infty},\
  \cosh^{-1}(\exp(\xi)) \le 1.4\, \xi$.
  This last inequality is easily shown as follows: for $\xi \ge 2$,
  \begin{equation*}
    \e^\xi
    \le \frac{\e^{0.8}}{2} \e^\xi
    \le \frac{\e^{0.4\xi}}{2} \e^\xi
    = \frac{\e^{1.4\xi}}{2}
    \le \cosh(1.4\xi),
  \end{equation*}
  since $\exp(0.8) > 1 + 0.8 + 0.8^2/2 = 2.12 > 2$.

  On an unbounded edge $e_n$, $\phi'' = \alpha^2_{e_n} \phi$ holds
  with $\alpha_{e_n} \coloneq \tfrac{1}{2} \lambda_n^{1/2}$.  Thus, on
  each bounded or unbounded edge $e_n$ of $\A_{R,n}$ and for all $n$
  large enough, one has, thanks to \eqref{eq:an>=1/2}, that
  \begin{equation*}
    L_n(\phi)
    = (\tilde W_n(x) - \alpha_{e_n}^2) \phi
    \ge \lambda_n \bigl(\tfrac{1}{2} - 0.7^2 \bigr) \phi
    \ge 0
  \end{equation*}
  and property~\ref{L>=0} is verified.  Finally,
  property~\ref{phi-upper-bd} is plainly
  verified at each vertex $\vv$ of $\A_{R,n}$. It also holds on each
  unbounded edge $e_n$ of $\A_{R,n}$ because
  of~\eqref{eq:bound-half-line}.
  On a bounded edge $e_n$ of $\A_{R,n}$ identified with
  $\intervalcc{0, \ell_{e_n}}$, in view of Lemma~\ref{lem:exp-bound},
  property~\ref{phi-upper-bd} is satisfied if, for all
  $x \in \intervalcc{0, \ell_{e_n}}$,
  \begin{equation*}
    \tfrac{1}{2} \lambda_n^{1/2} d_n(0) + \beta_{e_n} x
    \ge \tfrac{1}{2} \lambda_n^{1/2} d_n(x)  -\tfrac{R}{2}
    \qquad\text{where }
    \beta_{e_n} \coloneq \tfrac{1}{2} \lambda_n^{1/2}  \,
    \frac{d_n(\ell_{e_n}) - d_n(0)}{\ell_{e_n}}
  \end{equation*}
  or, equivalently, if
  \begin{equation}
    \label{eq:exp-bound-bounded-edge}
    \forall x \in \intervalcc{0, \ell_{e_n}}, \qquad
    d_n(0) + \frac{d_n(\ell_{e_n}) - d_n(0)}{\ell_{e_n}} x
    \ge d_n(x) - R \lambda_n^{-1/2} .
  \end{equation}
  Note that the inequality is plainly satisfied for $x = 0$ and
  $x = \ell_{e_n}$.  Thus if \eqref{eq:appropriate-local-coordinates}
  holds on the edge $e_n$, \eqref{eq:exp-bound-bounded-edge} is true
  as both hand sides are affine functions.
  When \eqref{eq:appropriate-local-coordinates} does not hold (because
  $e_n$ could not be split) there exists a
  $\tilde x_n \in (0, \ell_{e_n})$ such that the right hand side is
  affine both on $\intervalcc{0, \tilde x_n}$ and on
  $\intervalcc{\tilde x_n, \ell_{e_n}}$, and the length of one interval,
  say for instance $\intervalcc{\tilde x_n, \ell_{e_n}}$, is smaller
  than $4\lambda_n^{-1/2}$.  In this case,
  \eqref{eq:exp-bound-bounded-edge} is true if and only if it holds
  for $x = \tilde x_n$, which in turn is equivalent to
  \begin{equation}
    \label{eq:exp-bound-nonsplit-edge}
    \bigl(\ell_{e_n} - d_n(\ell_{e_n}) + d_n(0) \bigr)
    \frac{\tilde x_n}{\ell_{e_n}}
    \le R \lambda_n^{-1/2}.
  \end{equation}
  Recalling that the two expressions in the min
  of~\eqref{eq:varying distance} are equal at $\tilde x_n$,
  the fact that the length of the second interval is small means
  $\ell_{e_n} - \tilde x_n = \tfrac{1}{2} \bigl(\ell_{e_n} -
  d_n(\ell_{e_n}) + d_n(0) \bigr) < 4 \lambda_n^{-1/2}$.  Using this,
  $\tilde x_n \le \ell_{e_n}$ and $R \ge R_\epsilon \ge 8$, one easily
  deduces~\eqref{eq:exp-bound-nonsplit-edge}.

  \medbreak

  \noindent\textit{Step 2: Proof of \eqref{eq:exp-bound}.}
  \medskip

  To do so we will distinguish whether $x\in \G$ belongs to the set
  $A_{R,n}$ or to its complement.  We apply
  \eqref{eq:exp-bound-explicit} with $\epsilon = b^{- 1 /(p-2)}$ and
  $R \coloneq R_\epsilon$. Also, since the statement allows the use
  of subsequences, one can assume without loss of generality that
  $n_R = 0$. As such, we obtain \eqref{eq:exp-bound} for
  $x \in A_{R,n}$ if $C \ge \e^{R} b^{- 1 /(p-2)}$. We now turn to the
  case $x \notin A_{R,n}$.  On one hand, for any $x \in \G$ and any
  $n$, the definition of $x^1_n$ and
  \eqref{eq:equiv-epsilontilde-epsilon} of Lemma~\ref{concentration at
    a local maximum} yield
  \begin{equation*}
    \abs{u_n(x)}
    \le \smash{\max_{\G} \abs{u_n}}
    = \abs{u_n(x^1_n)}
    \le D \lambda_n^{1/(p-2)}
  \end{equation*}
  for some positive constant $D$.  On the other hand, for
  $x \notin A_{R,n}$, $d_n(x) < R \lambda_n^{-1/2}$  and so whenever
  $C \ge D \e^{R/2}$, we have
  \begin{equation*}
    |u_n(x)|
    \le D\lambda_n^{\frac{1}{p-2}}
    \le D \e^{R/2}\lambda_n^{\frac{1}{p-2}}
    \exp\bigl( -\tfrac{1}{2} \lambda_n^{1/2} d_n(x) \bigr)
    \le C\lambda_n^{\frac{1}{p-2}}
    \exp\bigl( -\tfrac{1}{2} \lambda_n^{1/2} d_n(x) \bigr),
    \quad x\notin A_{R,n}.
  \end{equation*}
  Thus \eqref{eq:exp-bound} holds if we choose
  $C \coloneq \max\{D, b^{-1/(p-2)} \e^{R/2}\} \e^{R/2}$.

  \bigbreak

  \noindent\textbf{Part 3: Proof of \eqref{eq:mass-convergence}.}
  \medskip

  Let $q \ge 1$.  We will establish the stronger property
  \begin{equation*}
    \lim_{n \to \infty} \lambda_n^{\frac{1}{2} - \frac{q}{p-2}}
    \int_\G \abs{u_n}^q \intd x
    = \sum_{i=1}^m (\tilde \lambda^i)^{\frac12-\frac{q}{p-2}}
      \int_{\G_{k^i}} \abs{\tilde u^i}^q \intd y.
  \end{equation*}

  Using~\eqref{eq:mass-concentration-Lq} or
  \eqref{eq:mass-concentration-Lq2} according to the case we are in,
  we easily establish that, for all $i$,
  \begin{equation*}
    \lim_{R \to \infty}\varlimsup_{n \to \infty} \,\,\,
    \biggabs{
      \lambda_n^{\frac{1}{2} - \frac{q}{p-2}}
      \int_{B(x^i_n, R\tilde\epsilon_n^i)} \abs{u_n}^q \intd x
      - (\tilde \lambda^i)^{\frac12-\frac{q}{p-2}}
      \int_{\G_{k^i}} \abs{\tilde u^i}^q \intd y
    } = 0.
  \end{equation*}
  Thus, it suffices to prove that
  \begin{equation*}
    \lim_{R \to \infty}  \varlimsup_{n \to \infty} \,\,\,
    \lambda_n^{\frac{1}{2} - \frac{q}{p-2}}
    \biggabs{
      \int_\G \abs{u_n}^q \intd x
      - \sum_{i=1}^m \int_{B(x^i_n, R\tilde\epsilon_n^i)} \abs{u_n}^q \intd x
    } = 0
  \end{equation*}
  or, equivalently, that
  \begin{multline}
    \label{eq:lim-masses}
    \forall \eta > 0,\
    \exists \tilde R_\eta,\ \forall R \ge \tilde R_\eta,\
    \exists \tilde n_R,\ \forall n \ge \tilde n_R,\\
    \lambda_n^{\frac{1}{2} - \frac{q}{p-2}} \biggabs{
      \int_\G \abs{u_n}^q \intd x
      - \sum_{i=1}^m \int_{B(x^i_n, R\lambda_n^{-1/2})} \abs{u_n}^q \intd x
    } \le \eta.
  \end{multline}
  Let $\eta > 0$ and take $\tilde R_\eta := 2R_\epsilon$
  where $R_\epsilon$ is obtained by~\eqref{eq:exp-bound-explicit}
  for a value $\epsilon$ to be chosen later.  Let $n_R$ be the value
  provided by~\eqref{eq:exp-bound-explicit} for a given $R \ge
  R_\epsilon$.  Let $R \ge \tilde R_\eta$.
  Take $\tilde n_R \ge n_{R/2}$ large enough so that,
  for all $n \ge \tilde n_R$,
  \begin{equation*}
    \forall i = 1,\dotsc, m,\quad
    R^i_n \ge R
  \end{equation*}
  (remembering that $R^i_n \to \infty$ as $n \to \infty$)
  and the balls $B(x^i_n, R\lambda_n^{-1/2})$ are disjoint
  for $i = 1,\dotsc,m$
  (this is possible thanks to~\eqref{eq:dist-points}). As a result,
  \begin{equation*}
    \int_\G \abs{u_n}^q \intd x
    - \sum_{i=1}^m \int_{B(x^i_n, R\lambda_n^{-1/2})} \abs{u_n}^q \intd x
    = \int_{A_{R,n}} \abs{u_n}^q \intd x.
  \end{equation*}
  Observe now that if $e_n$, identified with
  $\intervalcc{0, \ell_{e_n}}$ or $\intervalco{0, \infty}$,
  is an edge of $A_{R,n}$
  where \eqref{eq:appropriate-local-coordinates} holds then we have
  $d_n(x)\geq R\lambda_n^{-1/2}+x$.  Thus since
  $A_{R,n} \subseteq A_{R/2,n}$ and using
  \eqref{eq:exp-bound-explicit} with $R$ replaced by
  $R/2 \ge R_\epsilon$ (which holds because $n \ge n_{R/2}$),
  we obtain,
  \begin{align}
    \lambda_n^{\frac{1}{2} - \frac{q}{p-2}} \int_{e_n } \abs{u_n}^q \intd x
    & \le
      \lambda_n^{1/2} \, \varepsilon^q \, \e^{qR/2}
      \int_{e_n} \e^{-\frac{1}{2}q \, \lambda_n^{1/2} d_n(x)} \intd x
      \nonumber\\
    &\le \lambda_n^{1/2} \, \varepsilon^q \,  \e^{qR/2} \,
      \int_{0}^{\ell_{e_n}} \exp\Bigl(-\frac{q}{2} \,
      \lambda_n^{1/2} (R\lambda_n^{-1/2} + x) \Bigr) \intd x
      \nonumber\\
    &\le \lambda_n^{1/2} \, \varepsilon^q \, \e^{qR/2}\,
      \int_{0}^{+\infty}  \exp\Bigl(-\frac{q}{2} \,
      \lambda_n^{1/2}(R\lambda_n^{-1/2}+x) \Bigr) \intd x
      \nonumber\\
    &=  \frac{2}{q} \, \varepsilon^q.
      \label{eq:bound-Lq-edge}
  \end{align}
  If \eqref{eq:appropriate-local-coordinates} does not hold (because
  we cannot cut the edge),
  the interval splits into $\intervalcc{0, \tilde x_n}$ and
  $\intervalcc{\tilde x_n, \ell_{e_n}}$ on each of which
  \eqref{eq:appropriate-local-coordinates} holds
  (see the explanation before~\eqref{eq:appropriate-local-coordinates}).
  On each of these sub-intervals, the bound~\eqref{eq:bound-Lq-edge}
  holds and so it suffices to
  count twice this edge.  Finally, taking
  into account that the number of edges of $A_{R,n}$ is
  bounded independently of $R$ and $n$,
  we deduce that, for all $q \geq 1$,
  $$\lambda_n^{\frac{1}{2} - \frac{q}{p-2}} \int_{A_{R,n}} \abs{u_n}^q \intd x$$
  can be made arbitrarily small by taking $\varepsilon >0$ small
  enough. At this point, \eqref{eq:lim-masses} follows ending the
  proof of the theorem.
\end{proof}

\begin{remark}
  \label{Rem Strict}
  Let us observe that, contrary to what happens for positive radial
  solutions on annuli (see e.g.~\cite[Corollary 3.2]{Espetal}), here
  the number of selected blow-up points can be
  strictly less than the number of local maxima
  of the absolute value of solutions. The situation is more complex
  due to the fact that we have two possible limit profiles.  For
  example, consider the four-star graph $\G_4$ with four half-lines
  $e_1, \dotsc, e_4$ connected at their origin $0$.

  Consider also $\varphi$ the unique solution to
  \[
    \begin{cases}
      -\varphi''+\varphi=|\varphi|^{p-2}\varphi,
      \quad \text{on } \IR,
      \\
      \varphi\in H^1(\IR),
      \\
      \varphi>0 \text{ on }\IR, \qquad
      \varphi'(0)=0,
    \end{cases}
  \]
  namely,
  $\varphi(x) = (\frac{p}{2})^{1/(p-2)} \,
  \bigl(\cosh(\frac{p-2}{2}x)\bigr)^{-2/(p-2)}$ is the soliton
  with maximum in $0$.

  Observe that, defining $u_n$ on the edge $e_i$ by
  \[
    u^i_n(x)
    = \lambda_n^{1/(p-2)} \varphi\bigl(\lambda_n^{1/2}x+(-1)^i \bigr),
    \quad \text{for } x \ge 0 \text{ and } i=1,2,3,4,
  \]
  we obtain a family of solutions in $H^1(\G_4)$ of
  \begin{equation}
    \label{G4}
    \begin{cases}
      - u'' + \lambda_n  u =  \abs{ u}^{p-2}  u, &\text{in } \G_4,
      \\
      \displaystyle
      \sum_{i=1}^4 \dd{u_{e_i}}{x}(0)=0.
    \end{cases}
  \end{equation}
  Their maxima are located on the odd edges at coordinate
  $x_{n}=\lambda_n^{-1/2}$ and values
  $u_n(x_n)=(\frac{p\lambda_n}{2})^{\frac{1}{p-2}}$.  Hence,
  \begin{equation*}
    |u_n(x_n)|^{\frac{p-2}{2}} \dist(x_n, \V)
    = \sqrt{\frac{p}{2}}
  \end{equation*}
  and we are in the situation of Lemma \ref{concentration at a local
    maximum v2}. Considering the sequence of maximum points $x_{n}^1$
  in $e_1$, the limit problem satisfied by the renormalized
  sequence~\eqref{scale-epsilon-n2}
  in Lemma \ref{concentration at a local maximum v2} is given by
  \[
    \begin{cases}
      -\tilde u'' + \frac2p \tilde u
      = \abs{\tilde u}^{p-2} \tilde u, &\text{in } \G_4,
      \\
      \displaystyle
      \sum_{i=1}^k \dd{\tilde u_{\edge_i}}{x}(0)=0,
    \end{cases}
  \]
  with maximum on $e_1$ at coordinate $\hat x=
  \sqrt{\frac{p}{2}}$.  Actually, a direct computation shows
  \begin{equation*}
    \tilde u^{i}_n(x)
    = \tilde u^{i}(x)
    = \left(\frac2p\right)^{\frac{1}{p-2}}
    \varphi\biggr(\sqrt{\frac2p}x+(-1)^i\biggr),
    \quad \text{for } x \ge 0 \text{ and }  i=1,2,3,4.    
  \end{equation*}
  Observe that $\tilde u$ has two local maxima (on edges $1$ and $3$)
  but no local minima.  This does not happen for radial solutions on
  radially symmetric domains, see \cite[Section 3]{Espetal}.

  Observe also that
  \[
    \lim_{R \to +\infty} \varlimsup_{n \to \infty}
    \Bigl( \lambda_n^{-1/(p-2)} \max_{d_n(x) \ge R \lambda_n^{-1/2}}
    \abs{u_n(x)} \Bigr) = 0,
  \]
  where $d_n(x) \coloneq \dist(x, x^1_n)$. Hence the construction of
  the proof of Theorem \ref{main} stops with
  $m=1$ while the functions $u_n$ have all two local maxima. These two
  local maxima are no longer visible on the number $m$ of sequences
  considered but on the limit profile.

  Consider now the sequence $(u_n)\subset H^1(\G_4)$ of solutions
  to~\eqref{G4} given, on every edge $e_i$, by
  \[
    u^{i}_n(x)
    = \lambda_n^{1/(p-2)} \varphi\bigl(\lambda_n^{1/2}(x+(-1)^i)\bigr),
    \quad \text{for } x \ge 0 \text{ and } i=1,2,3,4.
  \]
  Their maxima are on the odd edges at coordinates $x_n=1$ and values
  $u_n(x_n)=(\frac{p\lambda_n}{2})^{\frac{1}{p-2}}$.  Hence, in that
  case
  $|u_n(x_n)|^{\frac{p-2}{2}} \dist(x_n, \V)
  = \sqrt{\frac{p\lambda_n}{2}}\to\infty$ and we are in the situation
  of Lemma \ref{concentration at a local maximum}. Considering the
  sequence of maximum points $x_{n}^1$ in $e_1$, the limit solution
  of the renormalized sequence~\eqref{scale-epsilon-n}
  in Lemma \ref{concentration at a local
    maximum} is given by
  \[
    \tilde u(x)
    = \left(\frac2p\right)^{\frac{1}{p-2}}
    \varphi\biggl(\sqrt{\frac2p}x\biggr),
    \quad \text{ for } x\in \IR.
  \]
  Moreover, in this case
  \[
    \lim_{R \to +\infty} \varlimsup_{n \to \infty}
    \Bigl( \lambda_n^{-1/(p-2)} \max_{d_n(x) \ge R \lambda_n^{-1/2}}
    \abs{u_n(x)} \Bigr)
    \ge \left(\frac{p}{2} \right)^{\frac{1}{p-2}},
  \]
  where $d_n(x) \coloneq \dist(x, x^1_n)$. Hence the construction of
  the proof of Theorem \ref{main} goes on until $m=2$ and the result
  ``counts'' the number of local maxima of $u_n$, as for domains.
\end{remark}

\begin{proof}[Proof of Corollary \ref{info-norms}]
  Take $q=2$ in \eqref{eq:mass-convergence}. Assume by contradiction
  that $\displaystyle\varlimsup_{n \to \infty} \lambda_n=\infty$. Then
  we have
  \begin{equation*}
    \displaystyle    \varlimsup_{n \to \infty}
    \lambda_n^{\frac{1}{2} - \frac{q}{p-2}}
    =
    \begin{cases}
      0, & \text{if } p \in \intervaloo{2, 6}, \\
      +\infty,& \text{if } p>6.      
    \end{cases}
  \end{equation*}
  This is a contradiction with \eqref{eq:mass-convergence}.
\end{proof}

\section{$L^{\infty}$ and $L^2$ bounds of solutions
  with bounded Morse index}\label{sec:A_Priori}

In this section, we consider pairs $(\lambda,u)\in \IR\times H^1(\G)$
solutions to
\begin{equation}
  \label{eq:edo_fixed}
  \begin{cases}
    -u'' + W(x) u + \lambda u = \rho(x) \abs{u}^{p-2} u
    &\text{on every edge } e \in \mathcal{E},\\
    u \text{ is continuous on } \G,\\[1\jot]
    \displaystyle
    \sum_{\edge \incident \vv} \dd{u_{\edge}}{x}(\vv)=0
    &\text{at every vertex } \vv \in \V.
  \end{cases}
\end{equation}
on a given graph $\G \in \GG$.
We assume the counterpart of condition~\ref{H} for this problem,
namely
\begin{equation}
  \label{hyp_W_rho}
  \tag{$\mathrm{H}_{W, \rho}$}
  W \in L^\infty(\G); \quad
  \rho \in L^\infty(\G) \setminus \{ 0 \}; \quad
  \exists a >0,\ \forall e \in \mathcal{E}, \text{ either
    $\rho\geq a$ on $e$ or $\rho\equiv 0$ on $e$.}
\end{equation}

\subsection{Link between the number of nodal zones
and the Morse index}

The following proposition is an adaptation to our setting of a classic
result, see e.g.\,\cite[Proposition 1]{BaLi} for a similar statement
in the context of domains in $\IR^N$.

\begin{proposition}
  \label{Morse_nodal_zones}
  Let $\G \in \GG$,
  $p > 2$, $W$ and $\rho$ be functions satisfying
  \eqref{hyp_W_rho}.  Let $(\lambda, u) \in \IR \times H^1(\G)$ be a
  solution to~\eqref{eq:edo_fixed}.  Then, $\morse(u)$, the Morse index of
  $u$, is greater or equal to the number of nodal zones\footnote{By
    definition, a nodal zone of $u$ is a connected component of
    $\{ x \in \G \mid u(x) \ne 0\}$.}  of $u$ that are not included in
  the set $\G_0 \coloneq \{ x \in \G \mid \rho(x) = 0\}$.
\end{proposition}

\begin{proof}
  Let $\C_1, \dotsc, \C_k$ be the $k$ nodal zones of $u$ which are not
  included in $\G_0$ (with possibly $k=\infty$) and let
  $\varphi_1, \dotsc, \varphi_k$ be the restrictions of $u$ to
  $\C_1, \dotsc, \C_k$.  For a given $1 \le i \le k$, since
  $(\lambda,\varphi_i)$ is a solution to~\eqref{eq:edo_fixed}
  on $\C_i$ with either Kirchhoff or Dirichlet
  vertex conditions (at the boundary of $\C_i$), one has
  \begin{align*}
    Q_u(\varphi_i;\G)
    &\coloneq \int_\G \Big(\abs{\varphi_i'}^2
    + (W(x) + \lambda)\abs{\varphi_i}^2
    - (p-1)\rho(x) |u(x)|^{p-2}|\varphi_i|^2\Big) \intd x\\
    &= - (p-2) \int_{\C_i} \rho(x) |u(x)|^{p} \intd x,
  \end{align*}
  which is negative since $\C_i$ is not a subset of $\G_0$.

  Now, we consider a sequence of cut-off functions
  $(\chi_n)_{n \ge 1} \subseteq H^1(\G) \cap \Cc(\G)$ such that
  $\chi_n v \xrightarrow{H^1} v$ as $n \to \infty$ for all $v \in H^1(\G)$
  (see e.g. \cite[Remark 3.2]{DeDoGaSeTr} for details about the
  existence of such a sequence of cut-off functions for metric graphs
  in $\G$).  Taking $n$ large enough, we deduce that there exists a
  suitable cut-off function $\chi \in H^1(\G) \cap \Cc(\G)$ so that
  $\tilde{\varphi}_i \coloneq \chi \varphi_i$ satisfies
  $Q_u(\tilde{\varphi}_i;\G) < 0$ for all $1 \le i \le k$.  Then,
  since the interiors of the supports of
  $\tilde{\varphi}_1, \dotsc, \tilde{\varphi}_k$ are disjoint,
  $Q_u(\cdot;\G)$ is negative definite on the linear span of the
  $\tilde{\varphi}_i$, $1 \le i \le k$, which has dimension $k$.
\end{proof}

\begin{remark}
  In Proposition \ref{Morse_nodal_zones}, the Morse index of $u$ may
  be equal to the number of nodal zones of $u$ as for example if $u$
  is an action ground state or a nodal action ground state (see
  e.g.\,\cite[Theorem 3.7]{DaPa}).  However,
  strict inequality may also hold.  Indeed, on
  compact metric graphs, taking $W \equiv 0$ and $\rho \equiv 1$,
  given $\lambda \ge 0$ the equation
  \eqref{eq:edo_fixed} admits the constant solution
  $c \coloneq \lambda^{\frac{1}{p-2}}$.  Its
  associated quadratic form is given by
  \begin{equation*}
    Q_c(\varphi;\G)
    = \int_\G \Big(\abs{\varphi'}^2
    + \lambda \abs{\varphi}^2
    - (p-1) \lambda |\varphi|^2\Big) \intd x
    = \int_\G \abs{\varphi'}^2 \intd x
    - (p-2) \lambda \int_\G |\varphi|^2 \intd x.
  \end{equation*}
  By taking $\lambda$ large enough, the Morse index of the constant
  solution can thus be made arbitrarily large, while it has only one
  nodal zone.
\end{remark}

The following example shows that the number of nodal
zones may be greater than the Morse index, if we
also count those included in $\G_0$.
We will come back to this example several times
in this section.

\begin{example}[``Purely linear solution on the 3-bridge'']
  \label{three_bridge}
  Consider the 3-bridge metric graph $\G_{\text{b}}$ made of two vertices $v_L$
  and $v_R$ joined by $3$ edges $e_1$, $e_2$, $e_3$ of length $1$
  (see Figure~\ref{fig:3-bridge}).
  \begin{figure}[ht]
    \centering
    \begin{tikzpicture}
      \node at (-2,0) [nodo] (L) {};
      \node at (2,0) [nodo] (R) {};
      \draw (0, 0) ellipse (2 and 0.6);
      \draw (L) -- (R);
      \node at (-2.3, -0.03) {$v_L$};
      \node at (2.3, -0.03) {$v_R$};
      \node at (0, 0.8) {$e_1$};
      \node at (0, 0.2) {$e_2$};
      \node at (0, -0.4) {$e_3$};
    \end{tikzpicture}
    \caption{The triple-bridge $\G_{\text{b}}$}
    \label{fig:3-bridge}
  \end{figure}
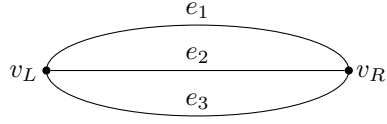

  \noindent
  Consider
  the eigenvalue problem on $\G_{\text{b}}$:
  \begin{equation*}
    \begin{cases}
      -u'' = \mu^2 u
      &\text{on every } e \in \{e_1, e_2, e_3\},\\
      u \text{ is continuous on } \G,\\[1\jot]
      \displaystyle
      \sum_{e \incident \vv} \dd{u_{e}}{x}(\vv) = 0
      & \text{for all } \vv \in \{v_L,v_R\}.
    \end{cases}
  \end{equation*}
  We use the notation $\lambda=\mu^2$ when
  referring to eigenvalues.  The first
  eigenvalue is $\lambda_1 = 0$ and its eigenspace is spanned by
  constant functions.  From now on, we thus consider $\mu > 0$.
  Denoting by $u_i$, $i=1,2,3$, the restriction
  of $u$ to $e_i$, which we identify to a function from
  $\intervalcc{0, 1}$ to $\IR$, we obtain
  \begin{equation*}
    u_i(x) = a_i \cos(\mu x) + b_i \sin(\mu x)
  \end{equation*}
  for some coefficients $a_i$ and $b_i$.  The continuity and Kirchhoff
  conditions in $v_L$ and $v_R$ read
  \begin{equation*}
    \begin{cases}
      a_1 = a_2 = a_3\\
      a_1 \cos(\mu) + b_1 \sin(\mu)
      = a_2 \cos(\mu) +b_2 \sin(\mu)
      = a_3 \cos(\mu) +b_3 \sin(\mu)\\
      \mu(b_1 + b_2 + b_3) = 0\\
      \mu\bigl(-(a_1 + a_2 + a_3)\sin(\mu)
      + (b_1 + b_2 + b_3)\cos(\mu) \bigr) = 0.
    \end{cases}
  \end{equation*}
  Simple computations imply that the system above has nonzero
  solutions if and only if $\sin(\mu) = 0$. For $\mu = k\pi$ with $k$
  a positive integer, we obtain that
  $\lambda_{3k-1} = \lambda_{3k} = \lambda_{3k+1} = (k\pi)^2$. The
  cor\-res\-pon\-ding eigenspace has dimension 3 and is given by
  \begin{equation*}
    \bigl\{ \varphi^k_{(a, b_1, b_2, b_3)} \bigm|
    (a, b_1, b_2, b_3) \in \IR^4,\
    b_1 + b_2 + b_3 = 0 \bigr\},
  \end{equation*}
  where the restriction of $\varphi^k_{(a, b_1, b_2, b_3)}$ to $e_i$
  is given by $x \mapsto a\cos(k\pi x) + b_i\sin(k\pi x)$.

  Let us consider $\varphi^k \coloneq \varphi^k_{(0, 1, -1, 0)}$.
  Then, $\varphi^k$ is an eigenfunction associated to $\lambda_{3k-1}$
  which vanishes identically on $e_3$.  We take
  $\rho \in L^\infty(\G_{\text{b}})$ defined by
  \begin{equation}
    \rho(x)
    \coloneq \begin{cases}
      0   &\text{if } x \in e_1 \cup e_2,\\
      1   &\text{if } x \in e_3.
    \end{cases}
    \label{three_bridge_rho}
  \end{equation}
  Then, the couple $(-\lambda_{3k-1},\varphi^k)$
  solves \eqref{eq:edo_fixed} with $W \equiv 0$ and $\rho$ as in
  \eqref{three_bridge_rho} since $\rho \cdot \varphi^k \equiv 0$.
  The quadratic form associated to $\varphi^k$ as a solution to
  \eqref{eq:edo_fixed} is given by (using again that
  $\rho \cdot \varphi^k \equiv 0$)
  \begin{equation*}
    Q_{\varphi^k}(\psi;\G_{\text{b}})
    \coloneq \int_{\G_{\text{b}}} |\psi'|^2 \intd x
    - \lambda_{3k-1}\int_{\G_{\text{b}}} |\psi|^2 \intd x,
  \end{equation*}
  so that the Morse index of $\varphi^k$ is $3k-2$.  Moreover, the
  number of nodal zones of $\varphi^k$ is $2k$.

  Therefore, when $k = 1$, we obtain a solution with $2$ nodal zones
  and having Morse index $1$.  This does not contradict
  Proposition~\ref{Morse_nodal_zones} since both those zones are
  included in $\G_0$.
\end{example}
\bigbreak

The following seemingly standard proposition
holds in our setting but its proof
requires some care since $\rho$ may vanish.

\begin{proposition}
  \label{Morse_at_least_one}
  Let $\G \in \GG$, $p > 2$, $W$ and $\rho$ be
  functions satisfying \eqref{hyp_W_rho}.  Let
  $(\lambda, u) \in \IR \times (H^1(\G) \setminus \{ 0 \})$ be a
  solution to~\eqref{eq:edo_fixed}.  Then,
  $\morse(u) \ge 1$.
\end{proposition}

\begin{remark}
  We have already encountered a similar statement in Lemma~\ref{Morse
    index>=1}, where the assumption on $\rho$ of \eqref{hyp_W_rho} was
  replaced by an assumption on $W+\lambda$.
\end{remark}

\begin{proof}
  Let $(\lambda, u) \in \IR \times (H^1(\G) \setminus \{ 0 \})$
  be a solution to~\eqref{eq:edo_fixed}.
  According to Proposition~\ref{Morse_nodal_zones}, the claim is true
  if $\rho \cdot u$ does not vanish identically.  Thus, let us now
  assume that $\rho \cdot u \equiv 0$.

  Since $\rho$ does not vanish identically, $u$ must change sign,
  since otherwise (up to replacing $u$ by $-u$)
  one would have $u \ge 0$, hence $u > 0$ on $\G$ by the strong
  maximum principle (Proposition \ref{strong_maximum_principle}),
  contradicting the fact that $u \cdot \rho \equiv 0$.

  Since $u \cdot \rho \equiv 0$, $u$ solves
  \begin{equation}
    \label{u_eigv}
    \begin{cases}
      -u'' + W(x) u + \lambda u = 0
      &\text{on } \G,\\[1\jot]
      \displaystyle
      \sum_{e \incident \vv} u'_{e}(\vv) = 0
      & \text{for all } \vv \in \V,
    \end{cases}
  \end{equation}
  and is thus an eigenfunction of $v \mapsto -v'' + W(x)v$ with
  eigenvalue $-\lambda$.

  The quadratic form associated to $u$ is
  \begin{equation*}
    Q_u(\varphi;\G)
    \coloneq \int_\G \abs{\varphi'}^2
    + \bigl(W(x) + \lambda\bigr)\abs{\varphi}^2
    \intd x.
  \end{equation*}
  It remains to show that one does not have
  $Q_u(\varphi;\G) \ge 0$ for all
  $\varphi \in H^1(\G) \cap \Cc(\G)$.
  If this was the case, it would mean that
  \begin{equation*}
    \inf_{\varphi \in H^1(\G) \setminus \{ 0 \}}
    \frac{
      \int_\G \abs{\varphi'}^2 + W(x)\abs{\varphi}^2 \intd x}{
      \int_\G \abs{\varphi}^2 \intd x}
    \ge -\lambda.
  \end{equation*}
  Since $u$ solves \eqref{u_eigv} and thus
  satisfies $Q_u(u;\G) = 0$, we deduce that $u$ would be a minimizer
  of the Rayleigh quotient and so a first eigenfunction. This
  contradicts the fact that $u$ changes sign, ending the proof.
\end{proof}

\subsection{Lower bounds on $\lambda$ for a given Morse index}

In this subsection, we are interested in
\begin{equation}
  \label{def_Lambda}
  \Lambda_{m^*}
  \coloneq \inf \bigl\{ \lambda \in \IR \bigm|
  \exists u \in H^1(\G) \setminus \{ 0 \},\
  (\lambda, u) \text{ solves } \eqref{eq:edo_fixed},
  \morse(u) \le m^* \bigr\}
\end{equation}
for a given positive integer $m^*$, taking
$\Lambda_{m^*} \coloneq +\infty$ if \eqref{eq:edo_fixed} has only the
trivial solution for all $\lambda\in\mathbb R$.

\subsubsection{Compact case}
\begin{proposition}
  \label{lower_bound_lambda_compact}
  Let $\G$ be a compact metric graph, $p > 2$, $m^*$ be a positive
  integer, and $W$, $\rho$ be functions satisfying
  \eqref{hyp_W_rho}.  Then, one has
  \begin{equation}
    \label{lower_bound_Lambda}
    \Lambda_{m^*} \ge -\lambda_{m^*+1}\bigl(-v'' + W(x)v\bigr)
  \end{equation}
  where $\lambda_{m+1}\bigl(-v'' + W(x)v\bigr)$ denotes
  the $(m+1)$-th eigenvalue of
  the operator $v \mapsto -v'' + W(x)v$ on the graph $\G$ with
  Kirchhoff conditions.  Moreover, if $m^* = 1$ and
  $\inf_{\G} \rho > 0$, one has
  \begin{equation}
    \label{Lambda_1}
    \Lambda_{1} = -\lambda_{1}\bigl(-v'' + W(x)v\bigr)
  \end{equation}
  and the infimum defining $\Lambda_{1}$ in \eqref{def_Lambda} is not
  attained.
\end{proposition}

\begin{proof}
    Let $(\lambda, u)
    \in \IR \times (H^1(\G) \setminus \{ 0 \})$
    be a nontrivial solution to~\eqref{eq:edo_fixed}
    with $\morse(u) \le m^*$.

    Let $\lambda_1 \le \lambda_2 \le \cdots \le \lambda_{m^*+1}$ be the $m^*+1$
    smallest eigenvalues of $v \mapsto -v'' + W(x)v$
    counted with multiplicity
    and $\varphi_1, \varphi_2, \dotsc, \varphi_{m^*+1}$
    denote the corresponding eigenfunctions,
    which we assume to be $L^2$-orthogonal.
    Given $t_1, t_2, \dotsc, t_{m^*+1} \in \IR$, let us define
    \begin{equation*}
        \varphi \coloneq \sum_{1 \le i \le m^*+1} t_i \varphi_i.
    \end{equation*}
    Assume by contradiction that
    $\lambda < -\lambda_{m^*+1}\bigl(-v'' + W(x)v\bigr)$. Then,
    \begin{align*}
        \int_\G \abs{\phi'}^2
        + \bigl( W(x) + \lambda \bigr) \phi^2
        - (p-1) \rho(x) |u|^{p-2} \phi^2(x) \intd x
        &\le \sum_{1 \le i \le m^*+1} t_i^2
        (\lambda + \lambda_i) \|\phi_i\|_{L^2(\G)}^2\\
        &\le (\lambda + \lambda_{m^*+1})
        \sum_{1 \le i \le m^*+1} t_i^2 \|\phi_i\|_{L^2(\G)}^2,
    \end{align*}
    is negative if at least one of the $t_i$ is nonzero,
    proving that $m^* + 1 \le \morse(u)$, contradicting
    the assumption $\morse(u) \le m^*$.

    Let us now prove \eqref{Lambda_1}.
    For $\lambda > -\lambda_1\bigl(-v'' + W(x)v\bigr)$,
    since $\rho(x) \ge a$ for almost every $x \in \G$,
    it is standard to show that solutions
    to \eqref{eq:edo_fixed} exist, for instance
    obtained by minimizing the action functional
    on the Nehari manifold associated to the problem
    (see e.g.\,\cite{DeDoGaSeTr} for such existence
    results on metric graphs).  Moreover, it is well known that these action ground states
    have Morse index $1$ (see e.g. \cite[Theorem 3.7]{DaPa}).

    Now, let us show that
    for $\lambda \le -\lambda_1\bigl(-v'' + W(x)v\bigr)$,
    problem \eqref{eq:edo_fixed} has no solution
    with Morse index one.
    First, we remark that such a solution would
    have a constant sign, since otherwise it would have
    two nodal zones (at least one for $u(x) > 0$
    and one for $u(x) < 0$), which is impossible
    using Proposition~\ref{Morse_nodal_zones}
    since $\rho \ge a$ almost everywhere.
    Up to replacing $u$
    by $-u$, we may assume $u \ge 0$ on $\G$.
    Thus, $u > 0$ using the strong maximum principle
    (Proposition \ref{strong_maximum_principle}).
    Considering $\varphi_1$ the first
    eigenfunction associated to $-v'' + W(x)v$
    (with $\varphi_1 > 0$)  gives
     \begin{equation*}
        (\lambda + \lambda_1) \int_{\G} u\,\varphi_1 \intd x
        = \int_{\G} \bigl(-u'' + W(x)u+\lambda u\bigr) \varphi_1 \intd x
        = \int_{\G} \rho(x) u^{p-1} \varphi_1 \intd x,
    \end{equation*}%
    which is a contradiction since the first term
    is nonpositive and the last is positive
    since $u$, $\rho$ and $\varphi_1$ are positive
    almost everywhere in $\G$.
\end{proof}

\begin{remark}
    \label{inf_lambda_attained}
    Example \ref{three_bridge} shows that
    equality in \eqref{lower_bound_Lambda} may hold,
    and also that the infimum \eqref{def_Lambda} may be attained.
    Indeed, for $k \ge 1$, $\varphi^k$
    is a solution to~\eqref{eq:edo_fixed}
    with $\lambda = -\lambda_{3k-1}$
    having Morse index $3k-2$, so that
    $\Lambda_{3k-2} = -\lambda_{3k-1}$.
    In particular, for $k = 1$,
    we have $\Lambda_1 = -\lambda_2 < -\lambda_1$,
    showing that \eqref{Lambda_1}
    may not hold when $\rho$ vanishes on some edge.
\end{remark}

\subsubsection{Metric graphs with half-lines}

\begin{proposition}
    \label{lower_bound_lambda_noncompact}
    Let $\G \in \GG$ be a metric graph with
    at least one half-line, $p > 2$,
    $W$ and $\rho$ be functions
    satisfying \eqref{hyp_W_rho}.
    Then, if $(\lambda, u) \in
    \IR \times H^1(\G)$
    is a finite Morse index solution to~\eqref{eq:edo_fixed},
    one has
    \begin{equation}
        \label{lower_bound_lambda_half-lines}
        \lambda
        \ge - \inf_{h \text{\,\upshape half-line}}
        \varlimsup_{\substack{x \to +\infty \\ x \in h}} W(x) ,
    \end{equation}
    where $h$ is identified
    with $\intervalco{0, +\infty}$.
\end{proposition}

\begin{proof}
    By contradiction, assume that there exists
    a half-line $h = \intervalco{0, +\infty}$
    on which one has
    \begin{equation*}
        \lambda + \varlimsup_{x \to +\infty} W(x) < 0.
    \end{equation*}
    Thus, there exists $R > 0$ and $\delta > 0$
    such that one has
    $\lambda + W(x) \le - 2\delta^2$
    for almost every $x \ge R$.
    We now define
  \begin{equation*}
        \psi(x)
        \coloneq \begin{cases}
            \sin(\delta x)
            &\text{if $0 \le x \le \pi/\delta$,}\\
            0
            &\text{otherwise}.
        \end{cases}
    \end{equation*}
    Then, $\psi \in H^1(\IR) \cap \Cc(\IR)$
    and, taking
  $\varphi_i(x) \coloneq \psi(x - R - 2i\pi/\delta)$,
    we obtain a infinite sequence $(\varphi_i)_{i \ge 1}$
    of functions with disjoint supports such that, if
    $\varphi \coloneq \sum_{1 \le i \le k} a_i \varphi_i \ne 0$,
    we have
    \begin{align*}
        \int_\G \abs{\phi'}^2 + \bigl( W(x) + \lambda \bigr) \phi^2
        - (p-1) \rho(x) |u|^{p-2} \phi^2 \intd x
        &\le -\delta \frac{\pi}{2}\sum_{1 \le i \le k} a_i^2
        < 0.
    \end{align*}
    Therefore the Morse index of $u$ is infinite
    since $k$ is arbitrary.
\end{proof}

\begin{remark}
    \label{finite_morse_index_lambda_zero}
    We will see in Example~\ref{compactly_supported_solutions}
    that equality may occur
    in \eqref{lower_bound_lambda_half-lines}.
\end{remark}

\subsection{$L^{\infty}$ bounds on solutions
when $\lambda$ stays bounded}

\subsubsection{$L^\infty$ bounds from bounds
on the Morse index}

We begin this section by stating
a property of superlinear ODEs.
\begin{lemma}
    \label{L_infty_ODE}
    Let $k$ be a positive integer,
    $p > 2$, $\ell > 0$,
    $W \in L^{\infty}\bigl((0, \ell)\bigr)$
    and $\rho \in L^{\infty}\bigl((0, \ell)\bigr)$.
    Assume that
    $\rho_{\text{\upshape m}}$, $\rho_{\text{\upshape M}}$,
    $W_{\text{\upshape M}} \in \intervaloo{0, +\infty}$
    are such that $|W(x)| \le W_{\text{\upshape M}}$
    and $\rho_{\text{\upshape m}} \le \rho(x) \le \rho_{\text{\upshape M}}$
    for almost every $x \in \intervaloo{0, \ell}$.
    Then, there exists
    $\Delta \bigl(k, p, \ell, \rho_{\text{\upshape m}},
    \rho_{\text{\upshape M}}, W_{\text{\upshape M}} \bigr) > 0$
    such that, if $u: \intervalcc{0, \ell} \rightarrow \IR$
    is a solution to the Cauchy problem
    \begin{equation*}
        \begin{cases}
            -u'' + W(x) u = \rho(x) |u|^{p-2} u,
            &x \in \intervaloo{0, \ell},\\
            u(0) = u_0,\ u'(0) = u_0',
        \end{cases}
    \end{equation*}
    that exists on $\intervalcc{0, \ell}$
    and so that $u_0^2 + u_0'^2 \ge \Delta$,
    then $u$ has at least $k$ zeros in $\intervalcc{0, \ell}$.
\end{lemma}

\begin{proof}
    This can be proved as in \cite[Lemma 2.1]{Ha},
    working with $H^1$ solutions instead
    of classical solutions.
\end{proof}

An important consequence of Lemma~\ref{L_infty_ODE}
is the following
\begin{proposition}
    \label{L_infty_G}
    Let $\G \in \GG$,
    $p > 2$,
    $W$ and $\rho$
    satisfy \eqref{hyp_W_rho},
    $k$ be a positive integer
    and $\Lambda > 0$.
    We define $\G_0 \coloneq \{ x \in \G \mid \rho(x) = 0\}$
    and $\underline{\ell} \coloneq \inf_{e \in \mathcal{E}} |e| > 0$.

    Then, there exists
    $C(k, p, \underline{\ell}, a,
    \|\rho\|_{L^\infty(\G)}, \|W\|_{L^\infty(\G)},
    \Lambda) > 0$
    so that, if
    $(\lambda, u)
    \in \intervalcc{-\Lambda, \Lambda} \times H^1(\G)$
    is a solution to~\eqref{eq:edo_fixed}
    such that on each edge, $u$ either vanishes identically
    or has at most $k$ roots, then
    \begin{equation}
        \label{L_infty_eq}
        \| u \|_{L^{\infty}(\G \setminus \G_0)} \le C
        \quad
        \text{and}
        \quad
        \| u' \|_{L^{\infty}(\G \setminus \G_0)} \le C.
    \end{equation}
\end{proposition}

\begin{proof}
  We apply Lemma~\ref{L_infty_ODE} edge by edge, taking
  $W_{\text{\upshape M}} \coloneq \|W\|_{L^\infty(\G)} + \Lambda$,
  $\rho_{\text{\upshape m}} \coloneq a$
  (where $a$ is given by assumption~\ref{hyp_W_rho}) and
  $\rho_{\text{\upshape M}} \coloneq \|\rho\|_{L^\infty(\G)}$ and
  $\ell \coloneq \underline{\ell}/2$.
\end{proof}

\begin{remark}
    \label{roots_morse_index}
    If a solution $u$ has at least
    $k$ simple roots inside some edge
    in which $\rho \ge a$,
    then the Morse index of $u$ is at least $k-1$
    using Proposition~\ref{Morse_nodal_zones}.
    Therefore, one can use Proposition~\ref{L_infty_G} to derive
    $L^{\infty}$ bounds on $u$ and on $u'$ knowing the Morse index of
    a solution.
\end{remark}

\begin{remark}
    In the ``localized nonlinearity'' setting,
    $\G$ has finitely many edges and
    $W$ and $\rho$ are identically equal to zero
    on the half-lines.
    Then, $\G \setminus \G_0$ is equal to  the
    compact core of $\G$ (made of all bounded edges of $\G$),
    but one nevertheless has a $L^{\infty}$ bound on the full graph $\G$,
    see point~\ref{Linfty_locnonlin}
    of Theorem~\ref{thm:locnonlin}.
\end{remark}

\begin{remark}
    Considering the solution of
    Example~\ref{three_bridge},
    we see that all functions of the form
    $t \varphi^k$, $t \in \IR$ solve \eqref{eq:edo_fixed}
    (with $\lambda = -\lambda_{3k-1}$,
    $W \equiv 0$ and $\rho$ given by
    \eqref{three_bridge_rho})
    and have the same Morse index. Taking $t \to +\infty$
    shows that one does not have $L^\infty$
    bounds on $\G_0$.
    We also observe that
    $t \varphi^k$ vanishes identically
    on $\G \setminus \G_0$
    so that \eqref{L_infty_eq} holds.
\end{remark}

\subsubsection{Bounds on the Morse index
from $L^\infty$ bounds}

On compact graphs, a bound on the Morse index
follows from a $L^\infty$ bound.

\begin{proposition}
    \label{Linfty_to_morse_compact}
    Let $\G$ be a compact metric graph,
    $p > 2$, $W \in L^{\infty}(\G)$ and $\rho \in L^{\infty}(\G)$.
    Then, for every $C > 0$, there exists
    a positive integer $m$ such that,
    for any $(\lambda, u) \in \IR \times H^1(\G)$
    solution to~\eqref{eq:edo_fixed}
    such that $\lambda \geq -C$
    and $\| u \|_{L^\infty(\G)} \le C$,
    $\morse(u) \le m$. In particular,
    all $H^1(\G)$ solutions to~\eqref{eq:edo_fixed}
    (for a given $\lambda \in \IR$)
    have a finite Morse index.
\end{proposition}

\begin{proof}
    If $Y \subseteq H^1(\G) \cap \Cc(\G)$
    is a subspace such that
    \begin{equation*}
        \forall \varphi \in Y \setminus \{ 0 \}, \quad
        \int_\G \Big(\abs{\phi'}^2
        + (W(x) + \lambda) \abs{\phi}^2
        - (p-1)\rho(x) |u(x)|^{p-2}|\phi|^2\Big) \intd x
        < 0,
    \end{equation*}
    then for some $D > 0$ independent of $u$, one has
    \begin{equation*}
        \forall \varphi \in Y \setminus \{ 0 \}, \quad
        \int_\G \abs{\phi'}^2 \intd x
        < D \int_\G \abs{\phi}^2 \intd x.
    \end{equation*}
    Since there are only finitely many eigenvalues
    of the Laplacian on $\G$ less or equal than $D$,
    this implies that the dimension of $Y$ is bounded.
\end{proof}

\begin{corollary}
    \label{corollary_bounds_compact}
    Let $\G$ be a compact metric graph,
    $p > 2$, $W \in L^{\infty}(\G)$,
    $\rho \in L^{\infty}(\G)$ be
    such that $\inf_{\G} \rho > 0$.
    Let $\Lambda \in \intervaloo{0, +\infty}$
    and let $\mathcal{U} \subseteq
    \intervalcc{-\Lambda, \Lambda} \times H^1(\G)$
    be a set of solutions to~\eqref{eq:edo_fixed}.
    Then, all solutions in $\mathcal{U}$
    have a finite Morse index, and one has that
    \begin{equation*}
        \sup_{(\lambda, u) \in \mathcal{U}}
        \| u \|_{L^\infty(\G)} < +\infty
        \quad\iff\quad
        \sup_{(\lambda, u) \in \mathcal{U}} \,
        \sup_{e \in \mathcal{E}}
        \| u \|_{C^1(e)} < +\infty
        \quad\iff\quad
        \sup_{(\lambda, u) \in \mathcal{U}}
        \morse(u) < +\infty.
    \end{equation*}
\end{corollary}

\begin{proof}
    This follows by combining
    Proposition~\ref{Linfty_to_morse_compact}, Proposition~\ref{L_infty_G}
    and Proposition \ref{Morse_nodal_zones}.
\end{proof}

When the graph is non-compact, one may obtain
an equivalent of Proposition~\ref{Linfty_to_morse_compact}
for problems with a localized nonlinearity
and solutions such that $W(x) + \lambda$ is
\emph{nonnegative}.

\begin{proposition}
    \label{Linfty_to_morse_localized}
    Let $\G \in \GG$,
    $p > 2$,
    $W \in L^{\infty}(\G)$, $\rho \in L^{\infty}(\G)$.
    Let $\K \subseteq \G$ be a compact set.
    Then, for every $C > 0$, there exists
    a positive integer $m$ such that,
    if $(\lambda, u) \in \IR \times H^1(\G)$
    is a solution to~\eqref{eq:edo_fixed}
    such that $\lambda + W(x) \ge 0$,
    $\| u \|_{L^\infty(\K)} \le C$
    and $u \cdot \rho \equiv 0$
    on $\G \setminus \K$, then $\morse(u) \le m$.
    In particular, all solutions to~\eqref{eq:edo_fixed}
    with compact support and $W(x) + \lambda \ge 0$
    have a finite Morse index.
\end{proposition}

\begin{proof}
    If $Y \subseteq H^1(\G) \cap \Cc(\G)$
    is a subspace such that
    \begin{equation}
        \label{morse_loc_prop}
        \forall \varphi \in Y \setminus \{ 0 \}, \quad
        \int_\G \Bigl(\abs{\phi'}^2
        + (W(x) +\lambda) \abs{\phi}^2
        - (p-1)\rho(x) |u(x)|^{p-2}|\phi|^2\Bigr) \intd x
        < 0,
    \end{equation}
    then, remarking that
    \begin{equation*}
      \forall \varphi\in H^1(\G), \quad
      \int_{\K} \abs{\phi'}^2 \le \int_\G \Big(\abs{\phi'}^2
      + (W(x) +\lambda) \abs{\phi}^2 \Bigr) \intd x,
    \end{equation*}
    since $W(x) +\lambda \ge 0$, we deduce that
    there exists $D > 0$ independent of $u$ such that
    \begin{equation*}
        \forall \varphi \in Y \setminus \{ 0 \}, \quad
        \int_{\K} \abs{\phi'}^2 \intd x
        < D \int_{\K} \abs{\phi}^2 \intd x.
    \end{equation*}
    As in Proposition~\ref{Linfty_to_morse_compact},
    we deduce that the dimension of
    $Y_{\K} \coloneq \{ \varphi_{\mid \K} \mid \varphi \in Y \}
    \subseteq H^1(\K)$
    is bounded by a number independent of~$u$.
    To conclude, we show that $\dim Y_{\K} = \dim Y$.
    To do so, let us prove that the linear map
    $Y \to Y_{\K}: \varphi \mapsto \varphi_{\mid \K}$
    is a bijection. It is clearly surjective.
    Moreover, assuming it is not injective,
    there exists $\varphi \in Y \setminus \{ 0 \}$
    which is identically equal to zero on $\K$.
    However, for this $\varphi$,
    \eqref{morse_loc_prop} is not satisfied, which ends
    the proof.
\end{proof}

The next example shows that the sign
of $W(x) +\lambda$
is important in the previous proposition.
\begin{example}[Compactly supported $H^1$ solutions]
    \label{compactly_supported_solutions}
    Let $\lambda \in \IR$ and choose $\ell > 0$ such that
    there exists a periodic solution of \eqref{eq:edo_fixed}
    having two roots in a circle of length $\ell$
    (see Proposition~\ref{periodic_sol}).
    Consider the tadpole graph $\G_{\text{t}}$
    with a loop of length $\ell$ depicted
    in Figure~\ref{fig:tadpole}.
    We take $W \equiv 0$ and $\rho \equiv 1$
    on the loop (the values of $\rho$ on the half-line
    are irrelevant).

    \begin{figure}[ht]
        \centering
        \begin{tikzpicture}[scale=1.2]
            \node at (0.5,0) [nodo] (0) {};
            \node at (5.25, 0) (1) {};
            \draw (0) -- (1);
            \draw (0, 0) ellipse (0.5 and 0.5);
            \node at (5.5, -0.03) {$\cdots$};
        \end{tikzpicture}
        \caption{A tadpole graph $\G_{\text{t}}$.}
        \label{fig:tadpole}
    \end{figure}
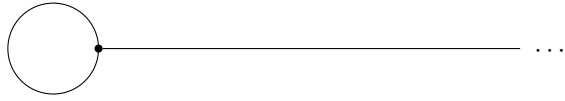
    Putting the periodic solution on the loop
    in such a way that one of its roots
    is at the vertex at which the half-line
    is attached and extending it by zero on
    the half-line, one obtains a solution $u \in H^1(\G)$
    of \eqref{eq:edo_fixed} with compact support.
    Now:
    \begin{itemize}
        \item if $\lambda < 0$,
        Proposition~\ref{lower_bound_lambda_noncompact}
        implies that $u$ is a $H^1$ solution
        with infinite Morse index
        (in contrast to the compact case,
        for which this is impossible, as shown
        by Proposition~\ref{Linfty_to_morse_compact});

        \item if $\lambda = 0$,
        Proposition~\ref{Linfty_to_morse_localized}
        shows that $u$ has a finite Morse index,
        showing that equality may occur
        in \eqref{lower_bound_lambda_half-lines} as announced
        in Remark~\ref{finite_morse_index_lambda_zero}.
    \end{itemize}
\end{example}

The next example shows that on non-compact graphs,
there may exist
sequences of solutions which are bounded in $L^\infty$
but their Morse indices are not.
\begin{example}
    Consider the $\mathbb{Z}$-periodic metric graph
    depicted in Figure~\ref{fig:necklace},
    where all edges have length one.
    We take $W \equiv 0$ and $\rho \equiv 1$.

    \begin{figure}[ht]
        \centering
        \begin{tikzpicture}[scale=1.2]
            \node at (-5.25, 0) (0) {};
            \node at (-4.5,0) [nodo] (1) {};
            \node at (-3.5,0) [nodo] (2) {};
            \node at (-2.5,0) [nodo] (3) {};
            \node at (-1.5,0) [nodo] (4) {};
            \node at (-0.5,0) [nodo] (5) {};
            \node at (0.5,0) [nodo] (6) {};
            \node at (1.5,0) [nodo] (7) {};
            \node at (2.5,0) [nodo] (8) {};
            \node at (3.5,0) [nodo] (9) {};
            \node at (4.5,0) [nodo] (10) {};
            \node at (5.25, 0) (11) {};
            \draw (0) -- (1);
            \draw (2) -- (3);
            \draw (4) -- (5);
            \draw (6) -- (7);
            \draw (8) -- (9);
            \draw (10) -- (11);
            \draw (-4, 0) ellipse (0.5 and 0.5);
            \draw (-2, 0) ellipse (0.5 and 0.5);
            \draw (0, 0) ellipse (0.5 and 0.5);
            \draw (2, 0) ellipse (0.5 and 0.5);
            \draw (4, 0) ellipse (0.5 and 0.5);
            \node at (-5.5, -0.03) {$\cdots$};
            \node at (5.5, -0.03) {$\cdots$};
            \node at (-4, 0.75) {$\mathcal{L}_{-2}$};
            \node at (-2, 0.75) {$\mathcal{L}_{-1}$};
            \node at (0, 0.75) {$\mathcal{L}_0$};
            \node at (2, 0.75) {$\mathcal{L}_1$};
            \node at (4, 0.75) {$\mathcal{L}_2$};
            \node at (-4.75, 0.2) {$v_{-4}$};
            \node at (-3.3, 0.2) {$v_{-3}$};
            \node at (-2.7, 0.2) {$v_{-2}$};
            \node at (-1.3, 0.2) {$v_{-1}$};
            \node at (-0.65, 0.2) {$v_0$};
            \node at (0.65, 0.2) {$v_1$};
            \node at (1.35, 0.2) {$v_2$};
            \node at (2.65, 0.2) {$v_3$};
            \node at (3.35, 0.2) {$v_4$};
            \node at (4.65, 0.2) {$v_5$};
        \end{tikzpicture}
        \caption{A $\mathbb{Z}$-periodic graph.}
        \label{fig:necklace}
    \end{figure}
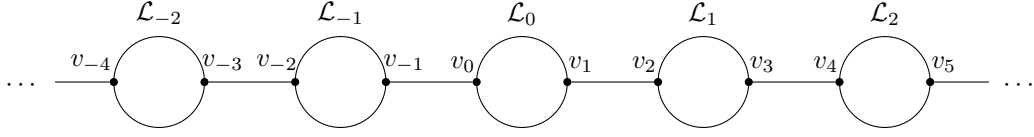

    Let $\lambda \ge 0$.
    On each loop $\mathcal{L}_i$,
    one put a periodic solution
    having two roots in the loop, in such a way
    that the roots coincide with the vertices
    $v_i$ and $v_{i+1}$ (see Proposition~\ref{periodic_sol}).
    In this way, for any finite subset
    $I \subseteq \mathbb{Z}$,
    one obtains a solution $u_I \in H^1(\G)$
    given by a periodic solution on every loop
    $\mathcal{L}_i$ with $i \in I$ and zero on all other
    edges. By construction, all those solutions have
    the same $L^\infty$ norm and solve the same equation.
    Since $\lambda \ge 0$ and those solutions
    have a compact support,
    Proposition~\ref{Linfty_to_morse_localized}
    implies that all solutions $u_I$
    have a finite Morse index.\footnote{Remark that
        the bound on the Morse index given by
        Proposition~\ref{Linfty_to_morse_localized}
        depends on the compact set $\K$,
        whence on $I$, so that we do not have
        a uniform Morse index bound
        on the solutions $u_I$.}
    However, since the number of nodal zones of those
    solutions is arbitrarily high when $|I| \to \infty$,
    their Morse indices are also unbounded
    according to Proposition~\ref{Morse_nodal_zones}.
\end{example}

\subsection{$L^2$ bounds on solutions
for large values of $\lambda$}

The following proposition is a direct consequence
of point \eqref{eq:mass-convergence}
from Theorem~\ref{main}, the main result
of the first sections of the paper.

\begin{proposition}
    \label{L2_large_lambda}
    Let $\G \in \GGfin$,
    $p > 2$, $m$ be a positive integer,
    $W$ and $\rho$
    satisfy \eqref{hyp_W_rho}.
    Then, there exists
    $\Lambda > 0$, $c > 0$, and $C > 0$
    so that for every $\lambda \ge \Lambda$,
    if $(\lambda, u) \in \IR \times (H^1(\G) \setminus \{ 0 \})$
    is a solution to~\eqref{eq:edo_fixed}
    with $\morse(u) \le m$, then
    \begin{equation*}
        c \lambda^{\frac{6-p}{2(p-2)}}
        \le \int_\G \abs{u}^2 \intd x
        \le C \lambda^{\frac{6-p}{2(p-2)}}.
    \end{equation*}
\end{proposition}

\begin{proof}
    Assume that one cannot find the existence
    of a couple $(\Lambda_1, c)$ such that the first
    inequality holds. If so, there exists a sequence
    of solutions $(\lambda_n, u_n)
    \subseteq \IR \times (H^1(\G) \setminus \{ 0 \})$
    with $\morse(u_n) \le m$  and $\lambda_n\to\infty$ so that
    \begin{equation*}
        \lim_{n \to \infty}
        \lambda_n^{\frac{p-6}{2(p-2)}}
        \int_\G \abs{u_n}^2 \intd x = 0.
    \end{equation*}
    This contradicts \eqref{eq:mass-convergence}
    (with $q = 2$).
    Similarly, if one does not have the existence
    of a couple $(\Lambda_2, C)$ such that the second
    inequality holds, there exists a sequence
    of solutions $(\lambda_n, u_n)
    \subseteq \IR \times (H^1(\G) \setminus \{ 0 \})$
    with $\morse(u_n) \le m$ and $\lambda_n\to\infty$  so that
    \begin{equation*}
        \lim_{n \to \infty}
        \lambda_n^{\frac{p-6}{2(p-2)}}
        \int_\G \abs{u_n}^2 \intd x = +\infty,
    \end{equation*}
    again contradicting \eqref{eq:mass-convergence}
    (with $q = 2$).
    We conclude taking $\Lambda \coloneq \max(\Lambda_1, \Lambda_2)$.
\end{proof}

\subsection{Behavior of the $L^2$ norms in terms
of $\lambda$ in some usual settings}

In all this subsection, we define
\begin{equation}
    q_{m^*}(\lambda)
    \coloneq \sup \bigl\{ \| u \|_{L^2(\G)} \mid
    u \in H^1(\G) \setminus \{ 0 \},\
    (\lambda, u) \text{ solves } \eqref{eq:edo_fixed},\
    \morse(u) \le m^* \bigr\},
\end{equation}
with the convention that
$q_{m^*}(\lambda) = 0$
if no such $u \in H^1(\G) \setminus \{ 0 \}$ exists.
We want to describe the behavior of $q_{m^*}(\lambda)$
in terms of $\lambda$ in different settings.
More precisely, in three common settings, we show that,
even if the behavior is the same for $\lambda$ large,
the global behavior  is different.

\subsubsection{Compact setting}
In the compact case, we may sometimes deduce that
$L^2$-norms converge to $0$ when $\lambda$
converges to the threshold value $\Lambda_{m^*}$
defined by~\eqref{def_Lambda}.

\begin{proposition}
    \label{limit_lambda_small}
    Let $\G$ be a compact metric graph, $p > 2$,
    $m^*$ be a positive integer,
    $W \in L^{\infty}(\G)$
    and $\rho \in L^{\infty}(\G)$
    be such that $\inf_{\G} \rho > 0$.
    Then, if the infimum defining $\Lambda_{m^*}$
    in \eqref{def_Lambda} is not attained
    (in particular, in view of
    Proposition~\ref{lower_bound_lambda_compact},
    if $m^* = 1$), one has
    \begin{equation}
        \label{small_mass_near_threshold}
        \lim_{\lambda \to \Lambda_{m^*}^+}
        q_{m^*}(\lambda)
        = 0.
    \end{equation}
\end{proposition}

\begin{proof}
    By contradiction, if
    \eqref{small_mass_near_threshold} does not hold,
    there exists $\epsilon > 0$ and
    a sequence of solutions $(\lambda_n, u_n)
    \subseteq \intervaloo{\Lambda_{m^*}, +\infty}
    \times (H^1(\G) \setminus \{ 0 \})$ such that
    \begin{equation*}
        \lambda_n \to \Lambda_{m^*}, \qquad
        \morse(u_n) \le m^*, \qquad
        \| u_n \|_{L^2(\G)} \ge \epsilon.
    \end{equation*}
    Since $\inf_{\G} \rho > 0$,
    Corollary~\ref{corollary_bounds_compact} implies that
    $(u_n)$ is bounded in $\C^1$-norm
    on every edge, hence in $H^1(\G)$
    (since $\G$ is compact).
    Taking a subsequence, there exists $u \in H^1(\G)$
    such that $u_n \rightharpoonup u$.
    Then, since weak $H^1$ convergence
    implies strong $L^\infty$ and in particular $L^2$
    convergence (as $\G$ is compact),
    $(\Lambda_{m^*}, u)$ is a solution of \eqref{eq:edo_fixed}
    and
    $\| u \|_{L^2(\G)} \ge \epsilon$.
    We then remark that $m(u) \le m^*$.
    Indeed, otherwise there would exist
    a space $Y \subseteq H^1(\G) \cap \Cc(\G)$
    of dimension $m^*+1$ on which
    $Q_u(\varphi;\G)$ is negative,
    but the $L^\infty$ convergence of $u_n$
    to $u$ implies that $Q_{u_n}(\cdot;\G)$ is also negative
    on $Y$ for $n$ large, contradicting $\morse(u_n) \le m^*$.
    Considering $u$ shows that $\Lambda_{m^*}$ is attained,
    a contradiction.
\end{proof}

In the following theorem, we adapt to the graph setting
a result of Pierotti and Verzini about solutions
of nonlinear Schrödinger equations on $\C^1$ bounded domains
(see \cite[Theorem 1.2]{PieVer}).
\begin{theorem}
    \label{thm:masses_compact}
    Let $\G$ be a compact metric graph, $p > 2$,
    $m^*$ be a positive integer,
    $W \in L^{\infty}(\G)$,
    and $\rho \in L^{\infty}(\G)$
    be such that $\inf_\G \rho > 0$.
    Let us define $\Lambda_{m^*}$ as in \eqref{def_Lambda}.
    Then, for all $\tilde\Lambda > 0$, there exists
    $C(\G, p, m^*, W, \rho, \tilde\Lambda) > 0$ such that,
    for all $\lambda \in \IR$,
    \begin{enumerate}
        \item\label{comp1} if $\lambda < \Lambda_{m^*}$,
        then $q_{m^*}(\lambda) = 0$;

        \item\label{comp2} $q_{m^*}(\Lambda_{m^*}) > 0$
        if and only if $\Lambda_{m^*}$ is attained;

        \item\label{comp3} if $\Lambda_{m^*}$ is not attained,
        then \smash{$q_{m^*}(\lambda)
        \xrightarrow[\lambda \to \Lambda_{m^*}]{} 0$};

        \item\label{comp4} if
        $\Lambda_{m^*} \le \lambda \le \tilde\Lambda$,
        then $q_{m^*}(\lambda) \le C$;

        \item\label{comp5} if
        $\lambda \ge \tilde\Lambda$, then
        $q_{m^*}(\lambda) \le C \lambda^{\frac{6-p}{4(p-2)}}$.
    \end{enumerate}
    In particular, if $p \ge 6$, $q_{m^*}$ is bounded on $\IR$
    so that there is a bound on the $L^2$ norms obtained
    for a bounded Morse index.
    When $p > 6$, we furthermore have
    $q_{m^*}(\lambda)
    \xrightarrow[\lambda \to +\infty]{} 0$.
\end{theorem}

\begin{proof}
  Let $\tilde\Lambda > 0$.
  Recall that by Proposition \ref{lower_bound_lambda_compact} we know that $\Lambda_{m^*}>-\infty$.
    The points~\ref{comp1} and~\ref{comp2}
    follow from the definition \eqref{def_Lambda}
    of $\Lambda_{m^*}$
    and point~\ref{comp3} from
    Proposition~\ref{limit_lambda_small}.
    Point~\ref{comp4} follows from
    Corollary~\ref{corollary_bounds_compact}.
    Finally, Proposition~\ref{L2_large_lambda} provides $\Lambda$
    such that the inequality in point~\ref{comp5} holds for
    $\lambda \ge \Lambda$.  Taking $C$ larger if necessary, we can
    ensure that the inequality of point~\ref{comp5} also holds for
    $\lambda \in \intervalcc{\tilde\Lambda, \Lambda}$.
\end{proof}

\subsubsection{Localized nonlinearity setting}
\label{sec:locnonlin}
\begin{theorem}
    \label{thm:locnonlin}
    Let $\G \in \GGfin$ be a metric graph with
    at least one half-line and at least one bounded edge,
    $p > 2$ and $m^*$ be a positive integer.
    Let $W \in L^\infty(\G)$ and
    $\rho \in L^\infty(\G)$ satisfy
    $\inf_{\K} \rho > 0$ where $\K$
    is the compact core of $\G$ (i.e.\ the metric subgraph of $\G$ consisting
of all the bounded edges of $\G$).
    Assume that $W \equiv \rho \equiv 0$
    on the half-lines of $\G$.
    Then, for all $\tilde\Lambda > 0$, there exists
    $C(\G, p, m^*, W, \rho, \tilde\Lambda) > 0$ such that,
    for all $\lambda \in \IR$,
    \begin{enumerate}
        \item\label{loc1} if $\lambda < 0$, then
        $q_{m^*}(\lambda) = 0$;
        \item\label{Linfty_locnonlin} every solution
        $(\lambda, u) \in \intervalcc{0, \tilde\Lambda}
        \times H^1(\G)$
        with $m(u) \le m^*$
        satisfies $\| u \|_{L^\infty(\G)} \le C$;
        \item\label{loc3} $q_{m^*}(0) \le C$
        and all solutions
        $u \in H^1(\G)$ with $\lambda = 0$
        and $m(u) \le m^*$ vanish
        on every half-line of $\G$;
        \item\label{loc4} if $0 < \lambda \le \tilde\Lambda$,
        then $q_{m^*}(\lambda)
        \le C \lambda^{-1/4}$;
        \item\label{loc5} if $\lambda \ge \tilde\Lambda$,
        then $q_{m^*}(\lambda)
        \le C \lambda^{\frac{6-p}{4(p-2)}}$.
    \end{enumerate}
\end{theorem}

\begin{proof}
  Let $\tilde\Lambda > 0$.
    Point~\ref{loc1} follows from
    Proposition~\ref{lower_bound_lambda_noncompact}.
    Then, we remark that
    on the half-lines, the ODE is $-u'' + \lambda u = 0$.
    If $\lambda = 0$, the only solution
    of this ODE converging to $0$
    as $x \rightarrow +\infty$ is $0$.
    If $\lambda > 0$, the solutions converging
    to $0$ have the form $u(x) = a \e^{-\sqrt{\lambda}x}$.
    Using Proposition~\ref{L_infty_G}
    (and recalling Remark~\ref{roots_morse_index}
    and that $\inf_{\K} \rho > 0$),
    the explicit form of the solutions
    on the half-lines implies point~\ref{Linfty_locnonlin},
    from which we also deduce point~\ref{loc3}
    since nonzero solutions for $\lambda = 0$
    have their support included in $\K$.

    \medbreak
    On the half-lines, the $L^2$ norms are given by
    \begin{equation*}
        \int_0^{+\infty} \bigl( a \e^{-\sqrt{\lambda}x} \bigr)^2 \intd x
        = \frac{a^2}{2\sqrt{\lambda}},
    \end{equation*}
    so that
    \begin{equation*}
        \| u \|_{L^2(\G)}^2
        \le \| u \|_{L^\infty(\G)}^2
        \Bigl( |\K| + \frac{H}{2\sqrt{\lambda}} \Bigr),
    \end{equation*}
    where
    $H$ is the number of half-lines of $\G$.
    Using point~\ref{Linfty_locnonlin},
    we obtain point \ref{loc4}.
    Proposition~\ref{L2_large_lambda} provides $\Lambda$ such that
    the inequality in point~\ref{loc5} holds for $\lambda \ge
    \Lambda$.  Possibly taking $C$ larger,
    point~\ref{loc4} ensures that this inequality is also valid for
    $\lambda \in \intervalcc{\tilde\Lambda, \Lambda}$.
\end{proof}

\begin{example}
    Let us show that one may have
    a blow-up of the $L^2$ norms as $\lambda \to 0^+$ with
    \begin{equation}
        \label{blowup_norm_tadpole}
        \varliminf_{\lambda \to 0^+}
        q_{m^*}(\lambda) \lambda^{1/4} > 0,
    \end{equation}
    so that one cannot improve the rate
    in point~\ref{loc4}.

    We consider a tadpole graph $\G_{\text{\upshape t}}$
    as in Figure~\ref{fig:tadpole},
    with a loop of length $2$
    and we take $W \equiv 0$ on $\G_{\text{\upshape t}}$,
    $\rho \equiv 1$ on the loop
    and $\rho \equiv 0$ on the half-line.

    Let us identify the loop with the interval
    $\intervalcc{-1, 1}$ in such a way that
    the middle of the loop corresponds to $0$
    and that the vertex at which the half-line
    is attached corresponds to $-1$ and to~$1$.

    Consider $u_0: \intervalcc{-1, 1} \to \IR$
    a $2$-periodic solution
    of $-u'' = |u|^{p-2}u$
    with two roots inside $\intervalcc{-1, 1}$
    which attains its (negative) minimum at $0$
    and its maximum at $1$ (see
    Proposition~\ref{periodic_sol}).
    Define $M_0 \coloneq -u_0(0)$.

    Finally, define a function $U_0: \G_{\text{\upshape t}} \to \IR$ by
    \begin{equation*}
        U_0(x) \coloneq
        \begin{cases}
            u_0(x)  &\text{if $x$ belongs to the loop,}\\
            u_0(1)  &\text{if $x$ belongs to the half-line}.
        \end{cases}
    \end{equation*}
    Then, $U_0$ solves \eqref{eq:edo_fixed}
    in a pointwise sense but is not in $H^1$
    since it does not converge to $0$ along the half-line.

    We will show that, for $\lambda > 0$
    small enough, there exist solutions $U_{\lambda}$, having a
    uniform bound on their Morse index,
    which converge pointwise to $U_0$
    as $\lambda \to 0^+$ and whose $L^2$-norms
    converge to $+\infty$ in such a way that
    \eqref{blowup_norm_tadpole} holds.

    Given $M \in \IR$ and $\lambda \in \IR$, we consider
    $\psi_{M, \lambda}(x)$ the solution of the Cauchy problem
    (which is well-defined for all $x \in \IR$ according
    to Proposition~\ref{no_blow_up})
    \begin{equation*}
        \begin{cases}
            -\psi'' + \lambda \psi
            = |\psi|^{p-2}\psi\\
            \psi(0) = -M,\ \psi'(0) = 0.
        \end{cases}
    \end{equation*}
    We want to show that for $\lambda \ge 0$ small,
    there exists $M(\lambda)$ so that one may put
    $\psi_{M, \lambda}$ on the loop
    and extend it to a solution on $\G_{\text{\upshape t}}$.
    On the half-line (identified with
    $\intervalco{0, +\infty}$), the ODE in \eqref{eq:edo_fixed} is simply
    $-u'' + \lambda u = 0$ so that one must have
    $u(x) = \psi_{M, \lambda}(1) \e^{-\sqrt{\lambda} x}$
    to obtain a solution decaying at infinity
    satisfying the continuity condition.
    It remains to satisfy Kirchhoff's condition
    at the node, which reads
    \begin{equation*}
        \psi_{M, \lambda}'(1)
        - \psi_{M, \lambda}'(-1)
        + \sqrt{\lambda} \psi_{M, \lambda}(1)
        = 0,
    \end{equation*}
    namely $F(M, \lambda) = 0$ with
    \begin{equation*}
        F(M, \lambda)
        \coloneq 2\psi_{M, \lambda}'(1)
        + \sqrt{\lambda} \psi_{M, \lambda}(1)
    \end{equation*}
    since $\psi_{M, \lambda}(x)$ is even in $x$.
    We have that $F(M_0, 0) = 0$.
    Moreover,
    $F \in \C^0(\IR \times \intervalco{0, +\infty})$
    by the continuous dependence of solutions
    of the Cauchy problem on initial values and parameters.

    Given $\epsilon > 0$ small enough, we have
    $F(M_0 - \epsilon, 0) = 2\psi_{M_0 - \epsilon, 0}'(1) > 0$
    and $F(M_0 + \epsilon, 0) = 2\psi_{M_0 + \epsilon , 0}'(1) < 0$
    (since the period of $\psi_{M, 0}$ is a decreasing function of $M$
    (see e.g.~\cite[Lemma~3.4]{CGJT-2025}) and $F(M_0, 0)=0$).
     By continuity of $F$, there exists $\delta > 0$
    such that for all $\lambda \in \intervalcc{0, \delta}$,
    one has $F(M_0 - \epsilon, \lambda) > 0$
    and $F(M_0 + \epsilon, \lambda) < 0$.
    Therefore, for all $\lambda \in \intervalcc{0, \delta}$,
    there exists
    $M(\lambda) \in \intervalcc{M_0 - \epsilon, M_0 + \epsilon}$
    with $F(M(\lambda), \lambda) = 0$.
    Moreover, we can choose $M(\lambda)$ so that
    $M(\lambda) \to M_0$ as $\lambda \to 0^+$.

    Therefore, for all $\lambda \in \intervalcc{0, \delta}$,
    the function $U_{\lambda}: \G_{\text{\upshape t}} \to \IR$ defined by
    \begin{equation*}
        U_{\lambda}(x) \coloneq
        \begin{cases}
            \psi_{M(\lambda), \lambda}(x)
            &\text{if $x$ belongs to the loop,}\\
            \psi_{M(\lambda), \lambda}(1)  \e^{-\sqrt{\lambda} x}
            &\text{if $x$ belongs to the half-line,}
        \end{cases}
    \end{equation*}
    is a solution of \eqref{eq:edo_fixed}.
    Since $M(\lambda) \to M_0$ as $\lambda \to 0^+$,
    we deduce that for all $x \in \G_{\text{\upshape t}}$,
    we have $U_{\lambda}(x) \to U_0(x)$ as $\lambda \to 0^+$.
    For all $\lambda \in \intervalcc{0, \delta}$,
    we have
    \begin{equation*}
        \| U_{\lambda} \|_{L^\infty(\G)}
        = \| \psi_{M(\lambda), \lambda} \|_{L^\infty(-1, 1)},
    \end{equation*}
    which is uniformly bounded with respect to
    $\lambda$. Thus, using Proposition~\ref{Linfty_to_morse_localized},
    there exists a positive integer $m^*$
    so that one has $\morse(U_\lambda) \le m^*$
    for all $\lambda \in \intervalcc{0, \delta}$.

    Finally, we have that
    \begin{equation*}
        \| U_{\lambda} \|_{L^2(\G)}
        \ge \psi_{M(\lambda), \lambda}(1)
        \left(
            \int_{0}^{+\infty} \e^{-2\sqrt{\lambda} x}\intd x
        \right)^{1/2}
        = \psi_{M(\lambda), \lambda}(1) \, 2^{-1/2} \,
        \lambda^{-1/4}.
    \end{equation*}
    Therefore,
    \begin{equation*}
        q_{m^*}(\lambda) \lambda^{1/4}
        \ge \| U_{\lambda} \|_{L^2(\G)} \lambda^{1/4}
        \ge \psi_{M(\lambda), \lambda}(1) \, 2^{-1/2}
        \xrightarrow[\lambda \to 0^+]{} u_0(1) \, 2^{-1/2},
    \end{equation*}
    so that \eqref{blowup_norm_tadpole} holds.
\end{example}

\begin{remark}
    On the other hand, Example~\ref{compactly_supported_solutions}
    shows that there may exist nonzero solutions
    with finite Morse index when $\lambda = 0$.
    Considering the same type of example for $\lambda > 0$
    shows that there exist families of solutions
    whose Morse index and $L^2$-norms stay bounded
    as $\lambda \to 0^+$. This shows that the $L^2$-norm of the solutions for $\lambda\to0$
    does not always behave like $\lambda^{-1/4}$.
    Hence, the $\lambda^{-1/4}$ rate for $\lambda \to 0^+$ is not universal,
    contrary to the $\lambda^{\frac{6-p}{2(p-2)}}$ rate
    for $\lambda\to +\infty$ (see Proposition \ref{L2_large_lambda}).
\end{remark}

\subsubsection{Usual NLS equation on the half-lines}
\begin{theorem}
    \label{thm:nls}
    Let $\G \in \GGfin$ be a metric graph with
    at least one half-line,
    $p > 2$ and $m^*$ be a positive integer.
    Let $W \in L^\infty(\G)$ and
    $\rho \in L^\infty(\G)$ be such that
    $\inf_{\K} \rho > 0$
    where $\K$ is the compact core of $\G$.
    Assume that $W \equiv 0$
    and that $\rho \equiv 1$
    on the half-lines of $\G$.
    Then, for all $\tilde\Lambda > 0$, there exists
    $C(\G, p, m^*, W, \rho, \tilde\Lambda) > 0$
    and $D(\G, p, m^*, W, \rho, \tilde\Lambda) > 0$ such that,
    for all $\lambda \in \IR$,
    \begin{enumerate}
        \item\label{nls1} if $\lambda < 0$, then
        $q_{m^*}(\lambda) = 0$;
        \item\label{nls2} every solution
        $(\lambda, u) \in \intervalcc{0, \tilde\Lambda}
        \times H^1(\G)$
        with $\morse(u) \le m^*$
        satisfies $\| u \|_{L^\infty(\G)} \le C$;
        \item\label{nls3} $q_{m^*}(0) \le C$
        and all solutions
        $u \in H^1(\G)$ with $\lambda = 0$
        and $\morse(u) \le m^*$ vanish
        on every half-line of $\G$;
        \item\label{nls4} if $0 < \lambda \le \tilde\Lambda$,
        then $q_{m^*}(\lambda)
        \le D + C \lambda^{\frac{6-p}{4(p-2)}}$;
        \item\label{nls5} if $\lambda \ge \tilde\Lambda$,
        then $q_{m^*}(\lambda)
        \le C \lambda^{\frac{6-p}{4(p-2)}}$.
    \end{enumerate}
\end{theorem}

\begin{proof}
    Point~\ref{nls1} follows from
    Proposition~\ref{lower_bound_lambda_noncompact} and
    point~\ref{nls2} from Proposition~\ref{L_infty_G}.
    On the half-lines, the ODE is
    \begin{equation}
        \label{ODE_half-line_NLS}
        -u'' + \lambda u = |u|^{p-2}u.
    \end{equation}
    If $\lambda \le 0$, a phase plane analysis
    shows that the only solution
    of this ODE converging to $0$
    as $x \rightarrow +\infty$ is $0$,
    from which point~\ref{nls3} follows from point~\ref{nls2}
    since the solutions for $\lambda = 0$
    have their support included in $\K$.

    If $\lambda > 0$, the solutions converging to $0$ of \eqref{ODE_half-line_NLS}
    are portions of the soliton $\phi_{\lambda}$,
    the unique positive solution
    of \eqref{ODE_half-line_NLS} on $\IR$ attaining
    its maximum at $x = 0$.
    A computation implies that for every $\lambda > 0$,
    \begin{equation*}
        \phi_{\lambda}(x)
        = \lambda^{\frac{1}{p-2}} \phi_1(\lambda^{\frac12}x)
    \end{equation*}
    so that
    \begin{equation*}
        \| \phi_{\lambda} \|_{L^2(\IR)}
        = \lambda^{\frac{6-p}{4(p-2)}} \| \phi_1 \|_{L^2(\IR)}.
    \end{equation*}
    Thus, if $(\lambda, u) \in \IR \times H^1(\G)$
    is a solution with $\morse(u) \le m^*$, we have
    \begin{equation}
        \| u \|_{L^2(\G)}^2
        \le \| u \|_{L^\infty(\G)}^2 |\K| + H \lambda^{\frac{6-p}{2(p-2)}} \| \phi_1 \|_{L^2(\IR)}^2
        \label{bd_2_u_nls}
    \end{equation}
    where $H$ is the number of half-lines of $\G$.
    Point~\ref{nls4}
    follows from \eqref{bd_2_u_nls} and point~\ref{nls2}.
    Point~\ref{nls5} results from Proposition~\ref{L2_large_lambda}
    and point~\ref{nls4} in the same way as in the proof of
    Theorem~\ref{thm:locnonlin}.
\end{proof}

\begin{remark}
    Example~\ref{compactly_supported_solutions}
    shows that there exists a family
    $(u_{\lambda})_{\lambda \in \intervalcc{0, \delta}}$
    of nonzero solutions with bounded Morse index
    and satisfying
    $0 < A \le \| u _{\lambda} \|_{L^2(\G)} \le B$
    for some constants $A, B > 0$ and
    for all $\lambda \in \intervalcc{0, \delta}$
    since one may consider a continuous family
    of periodic solutions supported in the loop
    (whose $L^2$ norms vary continuously as $\lambda$
    varies).
    Moreover, simply taking $\G = \IR$
    and $W \equiv 0$, $\rho \equiv 1$
    shows that the term
    $C \lambda^{\frac{6-p}{4(p-2)}}$
    is important in the upper bounds
    in \ref{nls4} and \ref{nls5}, by considering
    the soliton as a solution on the real line.
\end{remark}

\appendix
\section{The NLS ODE}
In this appendix, for the reader's convenience,
we briefly recall a few properties
of the NLS ODE
\begin{equation}
    \label{eq:NLS_ODE}
    -u'' + \lambda u = |u|^{p-2}u
\end{equation}
for given $p > 2$ and $\lambda \in \IR$.

\begin{proposition}
    \label{no_blow_up}
    Given $p > 2$ and $\lambda \in \IR$,
    all solutions of \eqref{eq:NLS_ODE}
    exist globally in $\IR$ and the ODE energy
    \begin{equation*}
        H(x) \coloneq \frac{|u'(x)|^2}{2}
        + \frac{|u(x)|^p}{p}
        - \lambda \frac{|u(x)|^2}{2}
    \end{equation*}
    is constant.
\end{proposition}

\begin{proof}
    Since
    \begin{equation*}
        H'(x)
        = u'(x) \Bigl(
        u''(x)
        + |u(x)|^{p-2}u(x)
        - \lambda u(x) \Bigr)
        = 0,
    \end{equation*}
    we deduce that $H$ is constant, so that
    $u(x)$ and $u'(x)$ remain bounded when $x$ varies,
    preventing blow-up.
\end{proof}

\begin{proposition}
    \label{periodic_sol}
    Let $p > 2$, $\lambda \in \IR$.
    Denote
    \begin{equation*}
      \ell_{\text{\upshape max}} \coloneq
        \begin{cases}
            +\infty
            &\text{if $\lambda \ge 0$},\\
            \frac{2\pi}{\sqrt{-\lambda}}
            &\text{if $\lambda < 0$}.
        \end{cases}
      \end{equation*}
    Then, for every $\ell \in \intervaloo{0, \ell_{\text{\upshape max}}}$,
    there exists a solution $u: \IR \to \IR$
    of \eqref{eq:NLS_ODE}
    which is $\ell$-periodic and such that
    $u(0) = u(\ell/2) = u(\ell) = 0$, $u > 0$
    on $\intervaloo{0, \ell/2}$ and
    $u < 0$ on $\intervaloo{\ell/2, \ell}$.
\end{proposition}

\begin{proof}
    It is standard to show
    (see e.g.~\cite[Chapter VIII, Corollary 1.6]{DeHab})
    that the problem
    \begin{equation*}
        \begin{cases}
          -v'' + \lambda v = |v|^{p-2}v,\\
          v(0) = v(\ell/2) = 0,
        \end{cases}
    \end{equation*}
    has a positive solution if $\lambda > -(2\pi/\ell)^2$,
    since $(2\pi/\ell)^2$ is the first eigenvalue
    of the Laplacian with Dirichlet boundary conditions
    on $\intervalcc{0, \ell/2}$. Since the ODE energy
    of $v$ is constant, $v'(0) = -v'(\ell/2)$
    so that $v$ is even because $v$ and $x \mapsto v(\ell/2-x)$
    solve the same Cauchy problem.
    Finally, the sign-changing periodic
    solution we seek is given by
    \begin{equation*}
      u(x) \coloneq \begin{cases}
        v(x)
        &\text{for } x \in \intervalcc{0, \ell/2},\\
        -v(x-\ell/2)
        &\text{for } x \in \intervalcc{\ell/2, \ell}.
      \end{cases}
      \tag*{\qedhere}
    \end{equation*}
\end{proof}


{\small
\begin{thebibliography}{99}




\bibitem{AST-CVPDE2015}
R. Adami, E. Serra and P. Tilli.
\newblock NLS ground states on graphs.
\newblock {\em Calc. Var. Partial Differential Equations} 54 (1):
743--761, 2015.

\bibitem{AST-JFA2016}
R. Adami, E. Serra and P. Tilli.
\newblock Threshold phenomena and existence results for NLS ground states on metric graphs.
\newblock {\em J. Funct. Anal.} 271 (1): 201--223, 2016.


\bibitem{AST11}
R.~Adami, E.~Serra and P.~Tilli.
\newblock Nonlinear dynamics on branched structures and networks.
\newblock {\em Riv. Math. Univ. Parma (N.S.)} 8(1):109--159, 2017.


\bibitem{ACT2024}
F.~Agostinho, S.~Correia and H.~Tavares.
\newblock Classification and stability of positive solutions to the NLS equation on the T-metric graph.
\newblock {\em Nonlinearity} 37(2), Paper No. 025005, 47~pp., 2024.

\bibitem{Ar}
A. H.~Ardila.
\newblock Orbital stability of standing waves for supercritical NLS with potential on graphs.
\newblock {\em Appl. Anal.} 99(8): 1359--1372, 2020.

\bibitem{BaLi}
A.~Bahri and P. L.~Lions.
\newblock Solutions of superlinear elliptic equations and their Morse indices.
\newblock {\em Comm. Pure. Appl. Math..} XLV: 1205--1215, 1992.





\bibitem{BK}
G.~Berkolaiko and P.~Kuchment.
\newblock {\em Introduction to quantum graphs}, Vol. 186 of {Mathematical
  Surveys and Monographs}.
\newblock American Mathematical Society, Providence, RI, 2013.


\bibitem{BCJS-2023}
J.\,Borthwick, X.\,J.\,Chang, L.\,Jeanjean and N.\,Soave.
\newblock  Normalized solutions of $L^2$-supercritical NLS equations on noncompact metric graphs with localized nonlinearities,
\newblock{\it Nonlinearity}, 36,  3776--3795, 2023.

\bibitem{BCJS-2024}
J.~Borthwick, X.\,J.\,Chang, L.\,Jeanjean and N.\,Soave.
\newblock Bounded Palais-Smale sequences with Morse type information for some constrained functionals.
\newblock {\it  Transactions of the American Mathematical Society} 377 (06),  4481--4517, 2024.

\bibitem{BCT-2019}
W.~Borrelli, R.~Carlone and L.~Tentarelli.
\newblock
An overview on the standing waves of nonlinear Schrödinger and Dirac equations on metric graphs with localized nonlinearity,
\newblock {\it Symmetry}, 11(2), 169, 22~pp., 2019.

\bibitem{BMP}
G.~Berkolaiko, J.~Marzuola and D.~Pelinovsky.
\newblock Edge-localized states on quantum graphs in the limit of large mass.
\newblock {\em Ann. Inst. H. Poincar\'{e}  Anal. Non Lin\'{e}aire} 38(5): 1295-1335, 2021.


\bibitem{CaFiNo}
C.~Cacciapuoti, D.~Finco and D.~Noja.
\newblock Ground state and orbital stability for the NLS equation on a general starlike graph with potentials.
\newblock {\em Nonlinearity} 30(8): 3271-3303, 2017.

\bibitem{CGJT-2025}
P.~Carrillo, D.~Galant, L.~Jeanjean and C.~Troestler.
\newblock Infinitely many normalized solutions of
$L^2$-supercritical NLS equations on noncompact
metric graphs with localized nonlinearities.
\newblock {\em Discrete and Continuous Dynamical Systems} 53: 82--105, 2026.

\bibitem{CJS-2022}
X.~J.~Chang, L.~Jeanjean and N.~Soave.
\newblock Normalized solutions of $L^2$-supercritical {NLS} equations on compact metric graphs.
\newblock  {\it Ann. Inst. H. Poincaré Anal. Non Linéaire} 41, no. 4, 933–959, 2024.
%
\bibitem{DaPa}
L.~Damascelli and F.~Pacella.
\emph{Morse Index of
Solutions of Nonlinear
Elliptic Equations},
De Gruyter Series in Nonlinear
Analysis and Applications 30, 2019.

\bibitem{DeDoGaSe}
C.~De Coster, S.~Dovetta, D.~Galant and E.~Serra.
On the notion of ground state for nonlinear Schrödinger equations on metric graphs.
\newblock {\it Calc. Var. Partial Differential Equations} 62, no. 5, Paper No. 159, 28~pp., 2023.

\bibitem{DeDoGaSeTr}
C.~De Coster, S.~Dovetta, D.~Galant, E.~Serra and C.~Troestler
\newblock  Constant sign and sign changing NLS ground states on noncompact metric graphs.
\newblock
{\it Analysis \& PDE} 19, no.~2, 203--240, 2026.


\bibitem{DeHab}
C.~De Coster and P.~Habets,
\emph{Two-point boundary value problems:
Lower and upper solutions}, Mathematics in Science and Engineering 205, Elsevier, Amsterdam, 2006.

\bibitem{D-JDE2018}
S.~Dovetta.
\newblock Existence of infinitely many stationary solutions of the $L^2$-
subcritical and critical NLSE on compact metric graphs.
\newblock {\em J. Differential
Equations} 264 (7): 4806-4821, 2018.


\bibitem{DGMP}
S.~Dovetta, M.~Ghimenti, A.~M.~Micheletti and A.~Pistoia.
\newblock Peaked and low action solutions of {NLS} equations on graphs with terminal edges.
\newblock {\em SIAM J. Math. Anal.} 52 (3): 2874-2894, 2020.


\bibitem{DT-OTAA-2020}
S.~Dovetta and L.~Tentarelli. Ground states of the $L^2$-critical NLS equation with localized non-linearity on a tadpole graph. \newblock {\em Discrete and continuous models in the theory of networks (Oper. Theory Adv. Appl.)} 281: 113-125, 2020.

\bibitem{DHR2004}
O.~Druet, E.~Hebey and F.~Robert
\newblock {\em Blow-up theory for elliptic PDEs in Riemannian geometry,}.
\newblock Mathematical Notes, vol. 45, Princeton University
Press, Princeton, NJ, 2004.
%
\bibitem{Espetal}
P.~Esposito, G.~Mancini, S.~Santra and P.~N.~Srikanth.
\newblock Asymptotic behavior of radial solutions for a semilinear elliptic problem on an annulus through Morse index.
\newblock {\em J. Differential Equations} 239 (1): 1-15, 2007.

\bibitem{EspPet}
P.~Esposito and M.~Petralla.
\newblock Pointwise blow-up phenomena for a Dirichlet problem.
\newblock {\em Commun. Partial Differ. Equations} 36 (7-9): 1654-1682, 2011.



\bibitem{GiSp}
B.~Gidas and J.~Spruck.
\newblock Global and local behavior of positive solutions of nonlinear elliptic equations.
\newblock {\em Comm. Pure Appl. Math.} 34 (4): 525-598, 1981.


\bibitem{GiTr}
D.~Gilbarg and N.S.~Trudinger,
\newblock {\it Elliptic Partial Differential Equations of
Second Order}, 2nd edition, Springer Verlag, 1983.


\bibitem{Ha}
P.~Hartman.
\newblock On boundary value problems for superlinear second order differential equations.
\newblock{\em Journal of Differential Equations} 26 (1): 37-53, 1977.





\bibitem{KNP}
A.~Kairzhan, D.~Noja and D.~E. Pelinovsky.
\newblock Standing waves on quantum graph.
{\it J. Phys. A: Math. Theor.} 55 243001, 51~pp., 2022.






\bibitem{No}
D.~Noja.
\newblock Nonlinear {S}chr\"{o}dinger equation on graphs: recent results and
  open problems.
\newblock {\em Philos. Trans. R. Soc. Lond. Ser. A Math. Phys. Eng. Sci.},
  372 no. 2007, 20130002, 20~pp., 2014.

\bibitem{NP}
D.~Noja and D.~E. Pelinovsky.
\newblock Standing waves of the quintic {NLS} equation on the tadpole graph.
\newblock {\em Calc. Var. Partial Differential Equations} 59 (5): Paper No. 173, 31~pp., 2020.

\bibitem{NPS}
D.~Noja, D.~Pelinovsky and G.~Shaikhova.
\newblock Bifurcations and stability of standing waves in the nonlinear {S}chr\"odinger equation on the tadpole graph.
\newblock {\em Nonlinearity} 28, 2343-2378, 2015.



\bibitem{PieVer}
D.~Pierotti and G.~Verzini.
\newblock Normalized bound states for the nonlinear {S}chr\"odinger equation in bounded domains.
\newblock {\em Calc. Var. Partial Differential Equations} 56 (5): Paper No. 133, 27 p., 2017.


\bibitem{RTT-JFA1998}
M. Ramos, S. Terracini and C. Troestler.
\newblock Superlinear indefinite elliptic problems and Pohozaev type identities.
\newblock {\em J. Funct. Anal.} 159 (2): 596-628, 1998.






\end{thebibliography}}
\end{document}